\documentclass{amsart}
\usepackage{amsfonts}
\usepackage{graphicx,amsfonts,amssymb,amsfonts,mathrsfs,amscd,amssymb,amsmath,graphicx,verbatim,xy,xypic,bm}
\usepackage{amscd,amssymb,amsmath,graphicx,verbatim}
\usepackage[dvips]{hyperref}
\usepackage[TS1,OT1,T1]{fontenc}

\newtheorem{theorem}{Theorem}[section]
\newtheorem{lemma}[theorem]{Lemma}
\newtheorem{corollary}[theorem]{Corollary}
\newtheorem{proposition}[theorem]{Proposition}

\theoremstyle{definition}
\newtheorem{definition}[theorem]{Definition}

\theoremstyle{remark}
\newtheorem{remark}[theorem]{Remark}

\numberwithin{equation}{section}

\begin{document}
\title[Topologically conjugate classifications]{Topologically conjugate classifications of the translation actions
on low-dimensional compact connected Lie groups}

\author{Xiaotian Pan}
\address{Xiaotian Pan, School of Mathematics, Jilin university, 130012, Changchun, P.R.China} \email{616280954@qq.com}

\author{Bingzhe Hou*}\thanks{*Corresponding author}
\address{Bingzhe Hou, School of Mathematics, Jilin university, 130012, Changchun, P.R.China} \email{houbz@jlu.edu.cn}
\date{}
\begin{abstract}
In this article, we focus on the left translation actions on noncommutative compact connected Lie groups
with topological dimension 3 or 4, consisting of ${\rm SU}(2),\,{\rm U}(2),\,{\rm SO}(3),\,{\rm SO}(3) \times S^1$ and ${{\rm Spin}}^{\mathbb{C}}(3)$.
We define the rotation vectors (numbers) of the left actions induced by the elements in the maximal tori
of these groups,
and utilize rotation vectors (numbers) to give the topologically conjugate
classification of the left actions. Algebraic conjugacy and smooth conjugacy are also considered.
As a by-product, we show that for any homeomorphism $f:L(p, -1)\times S^1\rightarrow L(p, -1)\times S^1$, the induced isomorphism $(\pi\circ f\circ i)_*$ maps each element in the fundamental group of $L(p, -1)$ to itself or its inverse, where $i:L(p, -1)\rightarrow L(p, -1)\times S^1$ is the natural inclusion and $\pi:L(p, -1)\times S^1\rightarrow L(p, -1)$ is the projection.
\end{abstract}
\maketitle

\section{Introduction}
The main research direction of this paper comes from the issue originally raised by S. Smale \cite{SS-1963} in the 1960s.
He hoped to classify all smooth self-maps of differential manifolds, and to obtain a
better understanding of the dynamical properties of some smooth self-maps.

First of all, let us review the concepts of topological conjugacy, smooth conjugacy, linear conjugacy and algebraic conjugacy.

Assume that $X$ is a topological manifold, and $f,\,g$ are continuous self-maps of $X$.
If there exists a homeomorphism $h:\,X \rightarrow X$ such that
$$
h \circ f=g \circ h,
$$
$f$ and $g$ are said to be topologically conjugate.
Moreover, if $X$ is a smooth manifold and $h$ is a smooth self-homeomorphism of $X$, $f$ and $g$ are said to be smooth conjugate.
If $X$ is a vector space and $h$ is a linear automorphism of $X$, $f$ and $g$ are said to be linearly conjugate.
If $X$ is a topological group and $h$ is an automorphism of $X$, $f$ and $g$ are said to be algebraically conjugate.

Topologically conjugate classification is an important and difficult task in dynamical systems.
One can associate to a self-homeomorphism $h$, a $C^*$-dynamical system $(C(X),\,h^*)$ and a crossed product $C^*$-algebra $C(X)\rtimes_h\mathbb{Z}$.
Hence we could study topological conjugate classification by studying the corresponding $C^*$-algebra dynamical systems and crossed product $C^*$-algebras.
In particular, the classification of $C^*$-algebras has been intensively studied in Elliott classification program \cite{Ell95}. So far the researchers such as G. Elliott, G. Gong, H. Lin, etc. have achieved great success for the classification of simple $C^*$-algebras (see \cite{EG96J}, \cite{EG96}, \cite{EGL07}, \cite{EI07}, \cite{Gong98}, \cite{Lin04}, \cite{Lin15M} and so on). It becomes increasingly important to classify minimal dynamical systems since their crossed product $C^*$-algebras are simple.
The remarkable work of Giordano, Putnam and Skau \cite{GPK} showed that the transformation
group $C^*$-algebras of two minimal homeomorphisms of the Cantor set are isomorphic if and only if the homeomorphisms are strong orbit equivalent.
Furthermore, J. Tomiyama proved that two transitive systems $(X, \alpha)$ and $(X, \beta)$ are flip conjugate if and only if their crossed product $C^*$-algebras are isomorphic preserving $C(X)$.
Toms and Winter classified the crossed product $C^*$-algebras of a special class of minimal homeomorphisms on finite dimensional metric spaces in \cite{TW09},
which improved a result of Lin and Phillips in \cite{LP10}.
H. Lin had done a series of work on minimal systems by the ordered K-theory of their crossed product $C^*$-algebras.
In particular, he proved that the crossed product $C^*$-algebras of two minimal homeomorphisms on metric spaces with finite covering dimension are isomorphic if and only if they have isomorphic Elliott invariants in \cite{Lin18}.

Unfortunately, there is a limitation to use the $C^*$-algebras to study classification of topological conjugacies.
From a result of Toms and Winter \cite{TW09},
the minimal diffeomorphisms of $S^3$ and $S^5$ correspond to the same crossed product $C^*$-algebras. On the other hand, a self-diffeomorphism corresponds to a smooth subalgebra (a dense subalgebra) of the crossed product $C^*$-algebra which is called a noncommutative differential manifold by A. Connes \cite{AC-1998}.
H. Liu \cite{HL-2017} proved the smooth subalgebras of diffeomorphisms of $S^3$ and $S^5$ have different cyclic cohomologies.
And then he \cite{HL-2018} utilized the smooth algebras of crossed products to classify certain self-diffeomorphisms up to smooth conjugacies.
Hence the smooth subalgebras of crossed products can be used to classify smooth conjugacies.

In this paper, we not only study topological conjugacy, but also investigate algebraic conjugacy and smooth conjugacy.
We expect our research to provide more examples for the study of $C^*$-algebras (as noncommutative topological spaces) and smooth subalgebras (as noncommutative differential manifolds).

For non-minimal dynamical systems, the conclusions of topologically conjugate classifications mainly lie in linear automorphisms of finite dimensional vector spaces, automorphisms of abelian topological groups and affine transformations of tori.
R. Adler and R. Palais \cite{AP-1965} discussed some topologically conjugate questions of the automorphisms of $T^n$.
And then for the topologically conjugate classifications of the Anosov diffeomorphisms on differential manifolds, especially $T^n$, there have been some satisfactory results (see \cite{JF-1970,MA-1974}). Furthermore,  Adler, Tresser and Worfolk \cite{ATW-1997} studied the linear endomorphisms on $T^2$ in detail. P. Walters discussed the topologically conjugate questions of the affine transformations of $T^n$ and more general compact abelian groups in \cite{WP-1968,WP-1969}, respectively. S. Bhattacharya also studied the questions of the orbit equivalence and topological conjugacy of the affine actions on compact abelian groups \cite{BS-2000}, and the automorphisms flows on compact connected Lie groups \cite{BS2-2000}.
J. Robin~\cite{RJW-1972} studied the topological conjugacy and structural stability for discrete dynamical systems, and then he cooperated with N. Kuiper to investigate the topologically conjugate classifications of the linear endomorphisms on finite linear spaces in \cite{KR-1973}.
R. Schult \cite{SR-1977} proved that the topologically conjugate classification of the linear representations of any compact Lie group is equivalent to their algebraically conjugate classification.

In the present article, we study the translation actions on some noncommutative compact connected Lie groups.
One can see that our research objects are not contained in the dynamical systems mentioned above.
For translation actions on compact Lie groups,
S. Weinberger \cite{SW-2008} stated that $(G, g)$ is topologically conjugate to $(G, h)$ if and only if $G/{\rm cl}(<g>) \cong G/{\rm cl}(<h>)$ by an
isomorphism that pulls back the principle ${\rm cl}<g>={\rm cl}<h>$ bundles, where $<g>$ is the subgroup generated by $g$.
In his paper, he investigated the case that $<g>$ is dense in the the
maximal torus $T$ of $G$,
and proved that the ``Weyl mixed'' translations are all topologically conjugate for ${\rm SU}(n),\,n > 2$.
Our aim is to find an explicit and computable way to classify the translation actions on a class of noncommutative compact connected Lie groups completely.

It is well-known that the rotation number plays important role in the study of orientation-preserving homeomorphisms on the unit circle,
as one can see from the classification of irrational rotation $C*$-algebras by Pimsner, Voiculescu \cite{PV} and Rieffel \cite{Rie}. Moreover,  H. Lin used rotation vectors to classify higher dimensional noncommutative tori \cite{Lin}, and he also studied approximate conjugacy for Furstenberg transformations on torus \cite{Lin2008}; Elliott and Li considered Morita equivalence of smooth noncommutative tori in \cite{EL07} and strong Morita equivalence of higher-dimensional noncommutative tori in \cite{EL08} and \cite{Li04}; B. Hou, H. Liu and X. Pan  \cite{HouLiuPan} studied mutual embeddability equivalence for rotation algebras.

\begin{definition}\label{def:1}
Assume that $f$ is an orientation-preserving self-homeomorphism of $S^1$,
and $F:\,\mathbb{R} \rightarrow \mathbb{R}$ is a lift of $f$. Then the limit
\begin{center}
${\lim\limits_{n\to\infty}}{\dfrac{F^{n}(x)-x}{n}}={\lim\limits_{n\to\infty}}{\dfrac{F^{n}(x)}{n}}$
\end{center}
exists, and it only depends on $F$. We denote this limit by $\rho(F)$.
If $F_1$ is another lift of $f$, there exists some $l \in \mathbb{Z}$ such that $F_1-F=l$, and then we get ${\rho(F_{1})}={\rho(F)}+l$, i.e.,
$$
{\rho(F_{1})}={\rho(F)}\quad({\rm mod}\,\,\,\mathbb{Z}).
$$
This fact indicates that the decimal parts of $\rho(F)$ and $\rho(F_{1})$ are equal. Hence we denote the common decimal part by $\rho(f)$ which is called the rotation number of $f$.
\end{definition}

There is an important result for rotation numbers and the topologically conjugate classification of the orientation-preserving self-homeomorphisms of $S^1$

\begin{proposition}\label{prop:1}
Assume that $f,\,g:\,S^1 \rightarrow S^1$ are two orientation-preserving self-homeomorphisms on $S^1$.
Then $f$ and $g$ are topologically conjugate if and only if
$$
\rho(g)=\pm\,\rho(f)\quad({\rm mod}\,\,\,\mathbb{Z}).
$$
\end{proposition}

In fact, Proposition \ref{prop:1} implies that we can give a topologically conjugate classification of the orientation-preserving self-homeomorphisms of $S^1$ by rotation numbers.

Similar to rotation numbers, people also defined rotation vectors of the rotations of $T^n$.
For any
$$
\alpha=\left(\begin{array}{c}
{\textrm{e}}^{2\pi{\rm i}\theta_1}\\
{\textrm{e}}^{2\pi{\rm i}\theta_2}\\
\vdots\\
{\textrm{e}}^{2\pi{\rm i}\theta_n}
\end{array}\right) \in T^n,\qquad
\theta_1,\,\theta_2,\, \cdots \,{\theta}_{n}\in [0,1),
$$
we define a rotation $f_\alpha$ of $T^n$ by
$$
f_\alpha:\,u \rightarrow \left(\begin{array}{cccc}
{\textrm{e}}^{2\pi{\rm i}\theta_1}&&&\\
&{\textrm{e}}^{2\pi{\rm i}\theta_2}&&\\
&&\ddots&\\
&&&{\textrm{e}}^{2\pi{\rm i}\theta_n}
\end{array}\right)u,\qquad
\forall\, u\in T^n.
$$
Set
$$
\mathfrak{M}_{T^n}=\big\{f_{\alpha}:\,T^n \rightarrow T^n\,|\,\alpha \in {T^n} \big\}.
$$
Then $\mathfrak{M}_{T^n}$ is the set consisting of all rotations of $T^n$

\begin{definition}\label{def:2}
Assume that $f_{\alpha} \in \mathfrak{M}_{T^n}$, and $F_{\alpha}:\,\mathbb{R}^n \rightarrow \mathbb{R}^n$ is a lift of $f_{\alpha}$. Then the limit
$$
\rho(F_{\alpha})={\lim\limits_{n\to\infty}}{\dfrac{F_{\alpha}^{n}({\bm x})-{\bm x}}{n}}
={\lim\limits_{n\to\infty}}{\dfrac{F_{\alpha}^{n}({\bm x})}{n}}
$$
exists, and it only depends on $F_\alpha$.
We denote the decimal part of $\rho(F_{\alpha})$ (it means the decimal parts of all components of the vector $\rho(F_{\alpha})$) by $\rho(f_{\alpha})$ which is called the rotation vector of $f_\alpha$.
\end{definition}

\begin{remark}
For the general self-homeomorphisms (not rotations) of $T^n$, we can not define the rotation vectors of them by the same way,
because the limit
$$
{\lim\limits_{n\to\infty}}{\dfrac{F^n({\bm x})-{\bm x}}{n}}
={\lim\limits_{n\to\infty}}{\dfrac{F^n({\bm x})}{n}}
$$
may be dependent on the choices of ${\bm x}$, or does not exist.
\end{remark}

For rotation vectors and topologically conjugate classifications, there is also an important result.

\begin{proposition}\label{prop:2}
Assume that $f,\,g \in \mathfrak{M}_{T^n}$ are two rotations of $T^n$.
Then $f$ and $g$ are topologically conjugate if and only if
$$
\rho(g)={\bm A}\rho(f)\quad({\rm mod}\,\,\,\mathbb{Z}),
$$
where the matrix ${\bm A}\in{\rm GL}_n(\mathbb {Z})$ is just the matrix form of the isomorphism $h_*:\,\pi_1(T^n) \rightarrow \pi_1(T^n)$ induced by $h$ which is the topological conjugacy from $f$ to $g$.
\end{proposition}

In fact, Proposition \ref{prop:2} implies that we can give a topologically conjugate classification of the rotations of $T^n$ by rotation vectors.

Following the development of rotation theory, many mathematicians, such as P. D\'{a}valos, J. Franks, A. Koropecki and so on, raised the concept of rotation sets for a class of homeomorphisms which are homotopic to the identity map on $T^n$, and proved a few important results in order to study the dynamical properties of these homeomorphisms (see \cite{DP-2013,FJ-1989,KA-2008,SW-1996,TF-2012}).

In this article, we also use some relevant results about of Lie groups \cite{PJF-1977,SMR-2007}, fiber bundles \cite{HD-1966} and lens spaces \cite{HA-2000}.

Firstly, we review the definition of $3$-dimensional lens space $L(p, q)$.

\begin{definition}\label{def:2.5}
Regard $S^3$ as the unit sphere in $\mathbb{C}^2$, i.e.,
$$
S^3=\{(z_1,\,z_2);\,\,|z_1|^2+|z_2|^2=1,\,\,z_1,\,z_2 \in \mathbb{C}\}.
$$
Define a period homeomorphism $f:\,S^3 \rightarrow S^3$ by
$$
f:\,(z_1,\,z_2) \mapsto ({\textrm{e}}^{2\pi{\rm i}/p}z_1,\,{\textrm{e}}^{2\pi{\rm i} q/p}z_2),\qquad
\forall\,(z_1,\,z_2) \in S^3,
$$
where $p,\,q\in \mathbb{Z}$ and ${\rm gcd}\,(p, q)=1$.
Then we denote the quotient space $S^3/f$ by $L(p, q)$ which is called lens space.
\end{definition}

Secondly, we review some relevant concepts and results about compact connected Lie groups.
For any compact connected Lie group $G$, there exists some maximal commutative subgroup homeomorphic to $T^n$ which is called the maximal torus of $G$. We denote the maximal torus of $G$ by $T_G$.
It is well-known that the maximal torus of any compact connected Lie group may be not unique, and a compact connected Lie group $G$ is commutative if and only if
$$
G=T_G \cong T^n.
$$
\begin{proposition}\label{prop:3}
Fix one maximal torus $T_G$ of $G$. Then for any $g \in G$, there exist some $t \in T_G, s \in G$ such that $sg=ts$.
\end{proposition}
Assume that $g,\,s \in G$ and $t \in T_G$ satisfy $sg=ts$. Then we study the left actions $L_g,\,L_t$, and $L_s$.
For any $g' \in G$, we have $sgg'=tsg'$, and then
$$
L_s \circ L_g(g')=L_t \circ L_s(g'),
\quad\hbox{i.e.}\quad
L_s \circ L_g=L_t \circ L_s.
$$
Obviously, every left action is a self-homeomorphism of $G$.
Hence we have the following corollary.

\begin{corollary}\label{GT}
Fix one maximal torus $T_G$ of $G$. Then any left action $L_g$ induced by an element $g$ in $G$ can be topologically conjugate to some left action $L_t$ induced by an element $t$ in $T_G$.
\end{corollary}
Therefore, it suffices to consider the left actions induced by the elements in the set $\mathfrak{M}_{T_G}=\{L_g:\,G\rightarrow G;\,\,g \in T_G\}$ to classify all left actions on $G$.

\begin{definition}\label{def:3}
Assume that $T_G$ is a maximal torus of a compact connected Lie group $G$, and $\Phi:\,T_G \rightarrow T^n$ is an isomorphism. Then $(T_G, \Phi)$ is called a maximal torus representation of $G$.
For every left action $L_g\in\mathfrak{M}_{T_G}$, set $f=\Phi\circ L_g|_{T_G}\circ\Phi^{-1}$.
Since $T_G$ and $T^n$ are both commutative, we see that $f:\,T^n \rightarrow T^n$ is a rotation (left action) of the normal $n$-dimensional torus $T^n$. Then we define
$$
\rho(L_g)\triangleq\rho(f),
$$
and $\rho(L_g)$ is called the rotation vector (number) of $L_g$ under the representation $(T_G, \Phi)$.
\end{definition}

We focus on some noncommutative compact connected Lie groups with topological dimension 3 or 4, and utilize the rotation vectors (numbers) defined above to give the topologically conjugate classifications of the left actions on these Lie groups. In addition, it is related to how the isomorphisms, induced by homeomorphisms on lens spaces $L(p, -1)$ and on $L(p, -1)\times S^1$, act on the fundamental groups.
Furthermore, for the left actions on these Lie groups,
we study the relationship among their topological conjugacy, algebraic conjugacy and smooth conjugacy.

\begin{remark}
For convenience, we only investigate the left actions on some compact connected Lie groups, and the relevant conclusions of the right actions on these Lie groups are the same as left actions.
\end{remark}

\section{Preliminaries}
In this section, we introduce some concepts, and prove a few useful results.

\begin{lemma}\label{etoe}
Suppose $G$ is a compact connected Lie group,
$L_g,\,L_{g'}$ are two left translation actions on $G$.
If $L_g$ and $L_{g'}$ are topologically conjugate,
there exists some topological conjugacy $h$ from $L_g$ to $L_{g'}$ such that
$$
h(e)=e,
$$
where $e$ is the identity element of $G$.
\end{lemma}
\begin{proof}
If $h$ is a topological conjugacy from $L_g$ to $L_{g'}$ (maybe $h(e)\neq e$),
we can define a self-homeomorphism of $G$ by
$$
h'(u)=h(u){h(e)}^{-1},\qquad \forall\, u \in G.
$$
It is easy to see that  $h'(e)=e$, and for any $u \in G$, we have
$$
h' \circ L_g(u)=h(L_g(u)){h(e)}^{-1}=L_{g'} \circ h(u){h(e)}^{-1}=L_{g'} \circ h'(u).
$$
Therefore, $h'$ is also a topological conjugacy from $L_g$ to $L_{g'}$ such that $h'(e)=e$.
\end{proof}

In this article, we always choose those topological conjugacies preserving the identity element.

\begin{lemma}\label{lem:1}
Assume that $G_1,\,G_2$ are topological groups,
$G_1',\,G_2'$ are the subgroups of $G_1$ and $G_2$, respectively,
and $h:\,G_1 \rightarrow G_2$ is a topological conjugacy from $\Gamma_{G_1'}$ to $\Gamma_{G_2'}$ satisfying $h(e_1)=e_2$,
where $\Gamma_{G_i'}$ is the $G_i'$ group action on $G_i$, and ${e_i}$ is the identity element of $G_i$, for $i=1,2$. Let
$\pi_1:\,G_1 \rightarrow G_1/G_1',\,\,\,\pi_2:\,G_2 \rightarrow G_2/G_2'$ be the quotient maps.
Then there exists some homeomorphism $h_\pi:\,G_1/G_1' \rightarrow G_2/G_2'$ induced by $h$ such that
$$
\pi_2 \circ h=h_\pi \circ \pi_1.
$$
In particular, if $G_1',\,G_2'$ are normal subgroups and $h$ is an isomorphism,
then $h_\pi$ is also an isomorphism induced by $h$.
\end{lemma}
The proof is simple.

\begin{remark}
In the present paper, for any topological group $G$ and its subgroup $G'$,
we always use the right coset of $G'$ to define the equivalence relation in $G$.
That means for any $g_1,\,g_2 \in G$, $[g_1]=[g_2]$ if and only if there exists some $g' \in G$ such that $g'g_1=g_2$, where $[g_1],\,[g_2]\in G/G'$ are the equivalence class of $g_1$ and $g_2$, respectively.
\end{remark}

\begin{lemma}\label{lem:2}\cite{SH-2010}
Suppose that $f$ is a degree 1 self-map on $L(p, q)$.
Then $f$ is homotopic to an orientation-preserving homeomorphism if and only if
$$
f_*(l)=\left\{
\begin{array}{l}
\pm\,l,\qquad\qquad p \nmid q^2-1,
\\[0.25 cm]
\pm\,l, \pm ql,\qquad p \mid q^2-1,
\end{array}
\right.
\qquad \forall\,l \in \pi_1(L(p, q)),
$$
where $f_*$ is the endomorphism of the fundamental group $\pi_1(L(p, q))$ induced by $f$.
\end{lemma}

\begin{lemma}\label{lem:3}
Assume that $f_1$ is a rotation of $T^2$,
$f_2$ is a rotation of $S^1$,
$f:\,T^2 \rightarrow S^1$ is a continuous surjection,
and
$$
f \circ f_1=f_2 \circ f,
\qquad \rho(f_1)=\left(\arraycolsep=1pt\begin{array}{c}
\theta\\
\varphi
\end{array}\right),
\qquad \rho(f_2)=\varphi'.
$$
For any $(m, n)\in\pi_1(T^2)\cong\mathbb{Z}\oplus\mathbb{Z}$,
if $f_*:\,\pi_1(T^2) \rightarrow \pi_1(S^1)$ satisfies
$$
f_*(m, n)=pm+qn,\qquad p,\,q\in\mathbb{Z},
$$
we have
$$
\varphi'=p\theta+q\varphi\quad({\rm mod}\,\,\,\mathbb{Z}),\qquad p,\,q\in\mathbb{Z}.
$$
\end{lemma}

\begin{proof}
Assume that $F_1:\,\mathbb{R}^2 \rightarrow \mathbb{R}^2,\,\,\,F_2:\,\mathbb{R} \rightarrow \mathbb{R}$
are the lifts of $f_1$ and $f_2$ under the universal covering maps defined by
$$
F_1(\bm{x})=\bm{x}+\left(\arraycolsep=1pt\begin{array}{c}
\theta\\
\varphi
\end{array}\right),\quad \forall\, {\bm x} \in \mathbb{R}^2,
\qquad
F_2(x)=x+\varphi',\quad \forall\, x \in \mathbb{R},
$$
respectively.
Since $\pi_1(\mathbb{R})$ and $\pi_1(\mathbb{R}^2)$ are both trivial,
and $f \circ \pi$ is a continuous map,
according to the map lifting theorem, there exists some continuous map $F:\,\mathbb{R}^2 \rightarrow \mathbb{R}$ such that
$$
\pi' \circ F=f \circ \pi,
$$
where $\pi:\,\mathbb{R}^2 \rightarrow T^2,\,\,\, \pi':\,\mathbb{R} \rightarrow S^1$ are the universal covering maps.
Through a simple calculation, we obtain
$$
\pi' \circ (F \circ F_1)=(f \circ f_1) \circ \pi,\qquad \pi' \circ (F_2 \circ F)=(f_2 \circ f) \circ \pi.
$$
And together with the condition $f \circ f_1=f_2 \circ f$, we have
$$
\pi' \circ (F \circ F_1)=\pi' \circ (F_2 \circ F).
$$
Therefore, it follows from the definition of universal covering map $\pi'$ that there exists some integer vector $l \in \mathbb{Z}$ such that
\begin{align}\label{eq:1}
F_2 \circ F=F \circ F_1+l.
\end{align}
Let $F_2$ act on the both sides of \eqref{eq:1}. Then we get
$$
F_2^2 \circ F=F_2 \circ (F \circ F_1+l).
$$
and the definition of $F_2$ indicates that
$$
F_2(x+l)-F_2(x)=l,
\qquad \forall\,x \in \mathbb{R}.
$$
so
\begin{align*}
F_2^2 \circ F
&=
F_2 \circ (F \circ F_1+l)
=F_2 \circ (F \circ F_1)+l
\\
&=
(F_2 \circ F) \circ F_1+l
=(F \circ F_1) \circ F_1+l+l
=F \circ F_1^2+2l.
\end{align*}
Thus, let $F_2$ act on the both sides of \eqref{eq:1} by $n$ times, we obtain
$$
F_2^n \circ F=F \circ F_1^n+nl,
$$
i.e.,
\begin{align}\label{eq:2}
\dfrac{F_2^n \circ F}{n}=\dfrac{F \circ F_1^n}{n}+l.
\end{align}
It is known that for any $(m, n)\in\pi_1(T^2) \cong \mathbb{Z}\oplus\mathbb{Z}$,
$$
f_*(m, n)=pm+qn,\qquad p,\,q\in\mathbb{Z}.
$$
Then by the definitions of the fundamental groups of $T^2$ and $S^1$,
one can see that
$$
F(\bm{x}+\bm{k})-F(\bm{x})=pm+qn=\bm{A}\bm{k},
\qquad \forall\,\bm{k}=\left(\arraycolsep=1pt\begin{array}{c}
m\\
n
\end{array}\right) \in \mathbb{Z}^2,\,\,\,\bm{x}\in\mathbb{R}^2,
$$
where $\bm{A}=\left(\arraycolsep=3pt\begin{array}{cc}
p&q
\end{array}\right)$ is a $1\times2$ integer matrix,
and then
$$
F(\bm{x}+\bm{k})-\bm{A}(\bm{x}+\bm{k})=F(\bm{x})-\bm{A}\bm{x}.
$$
This fact implies that $F(\bm{x})-\bm{A}\bm{x}$ is a continuous periodic function.
So there exists some $M > 0$ such that
$$
\|F(\bm{x})-\bm{A}\bm{x}\| \leq M,
$$
and then
$$
\|F(F_1^n(\bm{x}))-\bm{A}F_1^n(\bm{x})\| \leq M.
$$
Consequently, we have
$$
{\lim\limits_{n\to\infty}}\dfrac{\|F(F_1^n(\bm{x}))-\bm{A}F_1^n(\bm{x})\|}{n}=0,
$$
i.e.,
$$
\lim\limits_{n\to\infty}\dfrac{F(F_1^n(\bm{x}))}{n}
=\lim\limits_{n\to\infty}\dfrac{\bm{A}F_1^n(\bm{x})}{n}.
$$
As $n\to\infty$, \eqref{eq:2} indicates that
$$
\lim\limits_{n\to\infty}\dfrac{F_2^n(F(\bm{x}))}{n}
=\lim\limits_{n\to\infty}\dfrac{F(F_1^n(\bm{x}))}{n}+l
=\lim\limits_{n\to\infty}\dfrac{\bm{A}F_1^n(\bm{x})}{n}+l,
$$
It follows from Definition \ref{def:1} and Definition \ref{def:2} that
$$
\rho(F_2)=\bm{A}\rho(F_1)+l, \qquad l \in \mathbb{Z},
$$
and hence
$$
\rho (f_2)={\bm A}\rho (f_1)\quad({\rm mod}\,\,\,\mathbb{Z}),
$$
i.e.,
$$
\varphi'=p\theta+q\varphi\quad({\rm mod}\,\,\,\mathbb{Z}),\qquad p,\,q\in\mathbb{Z}.
$$
\end{proof}

\begin{lemma}\label{lem:4}
Let $p,\,q \in \mathbb{Z}$ with ${\rm gcd}\,(p, q)=1$.
Suppose that
$$
G=\left\{\left(\begin{array}{cc}
z\omega&0\\
0&\bar{z}
\end{array}\right);\,\,\,z^p\omega^q=1,\,\,\,z,\,\omega \in \mathbb{C},\,\,\,|z|=|\omega|=1\right\}.
$$
Then $G$ is not only a subgroup of ${\rm U}(2)$ but also a simple closed curve homeomorphic to $S^1$, and
$$
{{\rm U}}(2)/G \cong L(p, -1),
$$
where $L(p, -1)$ is a $3$-dimensional lens space.
\end{lemma}

\begin{proof}
In fact, it is obvious that $G$ is a subgroup of ${\rm U}(2)$,
and together with the condition ${\rm gcd}\,(p, q)=1$, one can see that $G$ is a simple closed curve, so $G\,\cong\,S^1$.

Regard $S^3$ as the unit sphere in quaternion space, and then denote the elements of $S^3$ by $z+{\rm j}z'$, where $|z|^2+|z'|^2=1,\,\,z,\,z' \in \mathbb{C}$.
Notice that ${\rm SU}(2) \cong S^3$ can be obtained by gluing the outside surfaces of two solid tori denoted by $D^2_+ \times S^1$ and $D^2_- \times S^1$, where $D^2_+,\,D^2_-$ are two unit closed disks,
and then define continuous maps $f_1:\,D^2_+ \times S^1 \rightarrow S^3,
\,f_2:\,D^2_- \times S^1 \rightarrow S^3$ by
\begin{gather*}
f_1:\,(\lambda_1, z_1)\mapsto\dfrac{z_1(1+{\rm j}\lambda_1)}{\sqrt{|\lambda_1|^2+1}},
\qquad \forall\, (\lambda_1, z_1) \in D^2_+ \times S^1,
\\
f_2:\,(\lambda_2, z_2)\mapsto\dfrac{z_2(\lambda_2+{\rm j})}{\sqrt{|\lambda_2|^2+1}},
\qquad \forall\, (\lambda_2, z_2) \in D^2_- \times S^1,
\end{gather*}
respectively.
Set
$$
f=f_1 \sqcup f_2:\,D^2_+ \times S^1 \sqcup D^2_- \times S^1 \rightarrow S^3,
$$
where ``$\sqcup$'' means disjoint union.
Obviously, $f$ is a surjection, so $f$ can induce an equivalence relation ``$\sim$'' defined by
$$
(\lambda_1, z_1) \sim (\lambda_2, z_2) \Longleftrightarrow
f_1(\lambda_1, z_1)=f_2(\lambda_2, z_2).
$$
Then
$$
D^2_+ \times S^1 \sqcup D^2_- \times S^1/\sim\,\,\,\cong S^3.
$$
According to the equation $f_1(\lambda_1, z_1)=f_2(\lambda_2, z_2)$, we get
$$
\left\{
\begin{array}{l}
\dfrac{z_1}{\sqrt{|\lambda_1|^2+1}}=\dfrac{z_2\lambda_2}{\sqrt{|\lambda_2|^2+1}},
\\[0.25 cm]
\dfrac{z_1\bar{\lambda_1}}{\sqrt{|\lambda_1|^2+1}}=\dfrac{z_2}{\sqrt{|\lambda_2|^2+1}},
\end{array}
\right.
\quad\hbox{i.e.}\quad
\left\{
\begin{array}{l}
\lambda_2=\lambda_1,
\\[0.25 cm]
z_2=z_1\bar{\lambda_1}.
\end{array}
\right.
$$
Therefore, we can define an identification map $F:\,\partial D^2_+ \times S^1 \rightarrow \partial D^2_- \times S^1$ by
$$
F:\,(\lambda_1, z_1)\mapsto(\lambda_1, z_1\bar{\lambda_1}),
\qquad \forall\,(\lambda_1, z_1) \in \partial D^2_+ \times S^1.
$$
By this way, we can regard ${\rm SU}(2)\,\cong\,S^3$ as a circle bundle on $S^2$,
and know that every fiber is a maximal torus of ${\rm SU}(2)$.
If we denote the arguments of $\lambda_1,\,\lambda_2,\,z_1,\,z_2$ by $\gamma,\,\gamma',\,\alpha,\,\alpha'$, respectively,
$F$ corresponds to a $2\times2$ matrix which describes the relationship of arguments as follows.
$$
\left(\arraycolsep=1pt\begin{array}{c}
\gamma'\\
\alpha'
\end{array}\right)=\left(\begin{array}{cc}
1&0\\
-1&1
\end{array}\right)\left(\arraycolsep=1pt\begin{array}{c}
\gamma\\
\alpha\
\end{array}\right).
$$
As well-known, ${\rm U}(2) \cong {\rm SU}(2) \times S^1$, so ${\rm U}(2)$ can be regarded as a $T^2$ bundle on $S^2$, and every fiber is a maximal torus of ${\rm U}(2)$.
Here, we state that the symbol ``$\times$'' between two Lie groups in the present paper just denotes Cartesian product in topology, not direct sum in algebra.
Hence by the definition of $F$, we can define an identification map $F':\,\partial D^2_+ \times S^1 \times S^1 \rightarrow \partial D^2_- \times S^1 \times S^1$ in ${\rm U}(2)$ by
$$
F':\,(\lambda_1, z_1, z_1')\mapsto(\lambda_2, z_2, {z_2}'),
\quad \forall\, (\lambda_1, z_1, z_1') \in \partial D^2_+ \times S^1 \times S^1
$$
satisfying
$$
F'(\lambda_1, z_1, {z_1}')=(\lambda_1, z_1\bar{\lambda_1}, {z_1}'),
\quad \forall\, (\lambda_1, z_1, z_1') \in \partial D^2_+ \times S^1 \times S^1.
$$
Similarly, $F'$ corresponds to a $3\times3$ matrix which describes the relationship of arguments as follows.
$$
\left(\arraycolsep=1pt\begin{array}{c}
\gamma'\\
\alpha'\\
\beta'
\end{array}\right)=\left(\begin{array}{ccc}
1&0&0\\
-1&1&0\\
0&0&1
\end{array}\right)\left(\arraycolsep=1pt\begin{array}{c}
\gamma\\
\alpha\\
\beta
\end{array}\right),
$$
where $\beta$ and $\beta'$ are the arguments of ${z_1}'$ and ${z_2}'$, respectively.
The known condition ${\rm gcd}\,(p, q)=1$ indicates that there exist some $m,\,n \in \mathbb{Z}$ such that $mp+nq=1$.
Through a transformation of arguments, we have
$$
\left(\arraycolsep=1pt\begin{array}{c}
\gamma\\
p\alpha+q\beta\\
-n\alpha+m\beta
\end{array}\right)=\left(\begin{array}{ccc}
1&0&0\\
0&p&q\\
0&-n&m
\end{array}\right)\left(\arraycolsep=1pt\begin{array}{c}
\gamma\\
\alpha\\
\beta
\end{array}\right).
$$
Set $\bm{A}=\left(\begin{array}{ccc}
1&0&0\\
0&p&q\\
0&-n&m
\end{array}\right)$. Then
$$
\left(\arraycolsep=1pt\begin{array}{c}
\gamma\\
\alpha\\
\beta
\end{array}\right)={\bm{A}}^{-1}\left(\arraycolsep=1pt\begin{array}{c}
\gamma\\
p\alpha+q\beta\\
-n\alpha+m\beta
\end{array}\right),
\qquad \left(\arraycolsep=1pt\begin{array}{c}
\gamma'\\
\alpha'\\
\beta'
\end{array}\right)={\bm{A}}^{-1}\left(\arraycolsep=1pt\begin{array}{c}
\gamma'\\
p\alpha'+q\beta'\\
-n\alpha'+m\beta'
\end{array}\right).
$$
Thus, we get
$$
\left(\arraycolsep=1pt\begin{array}{c}
\gamma'\\
p\alpha'+q\beta'\\
-n\alpha'+m\beta'
\end{array}\right)={\bm{A}}\left(\begin{array}{ccc}
1&0&0\\
-1&1&0\\
0&0&1
\end{array}\right){\bm{A}}^{-1}\left(\arraycolsep=1pt\begin{array}{c}
\gamma\\
p\alpha+q\beta\\
-n\alpha+m\beta
\end{array}\right).
$$
It follows from a simple calculation that
\begin{align}\label{eq:3}
\left(\arraycolsep=1pt\begin{array}{c}
\gamma'\\
p\alpha'+q\beta'\\
-n\alpha'+m\beta'
\end{array}\right)=\left(\begin{array}{ccc}
1&0&0\\
-p&1&0\\
n&0&1
\end{array}\right)\left(\arraycolsep=1pt\begin{array}{c}
\gamma\\
p\alpha+q\beta\\
-n\alpha+m\beta
\end{array}\right).
\end{align}
According to the definition of the subgroup $G$, it is easy to know that $G$ is included in some maximal torus of ${\rm U}(2)$.
Then by this way, one can see that the quotient space ${{\rm U}}(2)/G$ is also a circle bundle on $S^2$,
that means ${{\rm U}}(2)/G$ can also be obtained by gluing the outside surfaces of two solid tori $D^2_+ \times S^1$ and $D^2_- \times S^1$.
Thus, $F'$ can induce an identification map $F'':\,\partial D^2_+ \times S^1 \rightarrow \partial D^2_- \times S^1$ in ${{\rm U}}(2)/G$.
According to \eqref{eq:3}, we see that $F''$ corresponds to a $2\times2$ matrix which describes the relationship of arguments,
and this matrix is just the matrix in \eqref{eq:3} except the third row and the third column, i.e.,
$$
\left(\arraycolsep=1pt\begin{array}{c}
\gamma'\\
p\alpha'+q\beta'
\end{array}\right)=\left(\begin{array}{cc}
1&0\\
-p&1
\end{array}\right)\left(\arraycolsep=1pt\begin{array}{c}
\gamma\\
p\alpha+q\beta
\end{array}\right).
$$
Then the identification map $F''$ satisfies
$$
F'':\,(\lambda_1, z_1)\mapsto(\lambda_1, z_1\bar{\lambda_1}^p),
\qquad \forall\, (\lambda_1, z_1) \in D^2_+ \times S^1.
$$
Using the above observations, we have
\begin{gather*}
{\rm U}(2) \cong D^2_+ \times S^1 \times S^1 \sqcup D^2_- \times S^1 \times S^1/\sim_1,
\\
{{\rm U}}(2)/G \cong D^2_+ \times S^1 \sqcup D^2_- \times S^1/\sim_2,
\end{gather*}
where ``$\sim_1$'' and ``$\sim_2''$ are the equivalence relations induced by $F'$ and $F''$, respectively.
Denote the elements of ${\rm U}(2)$ and ${\rm U}(2)/G$ by $(z, \textrm{e}^{2\pi{\rm i}\alpha}, \textrm{e}^{2\pi{\rm i}\beta})$ and $(z, \textrm{e}^{2\pi{\rm i}\gamma})$, respectively,
where $|z| \leq 1$, $z \in D^2_+$ or $z \in D^2_-$, and $\alpha,\,\beta,\,\gamma \in [0, 1)$.
Hence it is easy to prove that the quotient map $\pi:\,{\rm U}(2) \rightarrow {\rm U}(2)/G$ satisfies
$$
\pi(z, \textrm{e}^{2\pi{\rm i}\alpha}, \textrm{e}^{2\pi{\rm i}\beta})=(z, \textrm{e}^{2\pi{\rm i}(p\alpha+q\beta)}),
\qquad \forall\, (z, \textrm{e}^{2\pi{\rm i}\alpha}, \textrm{e}^{2\pi{\rm i}\beta}) \in {\rm U}(2).
$$
On the other hand,
define two maps $\pi_+:\,D^2_+ \times S^1 \rightarrow D^2_+ \times S^1$ and $\pi_-:\,D^2_- \times S^1 \rightarrow D^2_- \times S^1$ by
\begin{gather*}
\pi_+:\,(z_1, \lambda_1)\mapsto(z_1, \lambda_1^p),
\qquad \forall\, (z_1, \lambda_1) \in D^2_+ \times S^1,
\\
\pi_-:\,(z_2, \lambda_2)\mapsto(z_2, \lambda_2^p),
\qquad \forall\, (z_2, \lambda_2) \in D^2_- \times S^1,
\end{gather*}
respectively. Then it follows from a simple verification that
$$
\pi_- \circ F=F'' \circ \pi_+.
$$
This fact illustrates that $S^3$ is a $p$-fold covering space of ${{\rm U}}(2)/G$, and the local representation of the covering map from $S^3$ to ${{\rm U}}(2)/G$ is $\{\pi_+,\,\pi_-\}$.
It is well-known that $S^3$ is a $p$-fold covering space of $L(p, -1)$, and according to one definition of $L(p, -1)$, the local representation of the covering map from $S^3$ to $L(p, -1)$ is also $\{\pi_+,\,\pi_-\}$.
Therefore, we obtain
$$
{{\rm U}}(2)/G \cong L(p, -1).
$$
\end{proof}

\begin{lemma}\label{lem:5}
Assume that $B$ is a topological space,
$(E, \pi)$ is the covering space of $B$,
$f,\,g$ are two continuous self-maps of $B$,
$h$ is a self-homeomorphism of $B$, and $f,\,g,\,h$ can be lifted under the covering map $\pi$.
Then $f$ and $g$ are topologically conjugate,
and the homeomorphism $h:\,B \rightarrow B$ is a topological conjugacy from $f$ to $g$, that means
$$
h \circ f=g \circ h
$$
if and only if for any fixed lift $\tilde{f}$ of $f$ and any fixed lift $\tilde{h}$ of $h$,
there exists some certain lift $\tilde{g}$ of $g$ such that $\tilde{f}$ and $\tilde{g}$ are topologically conjugate,
and the homeomorphism $\tilde{h}:\,E \rightarrow E$ is a topological conjugacy from $\tilde{f}$ to $\tilde{g}$, that means
$$
\tilde{h} \circ \tilde{f}= \tilde{g} \circ \tilde{h}.
$$
\end{lemma}

\begin{proof}
Firstly, we prove the sufficiency.

Assume that $\tilde{f},\,\tilde{g}$ and $\tilde{h}$ are the lifts of $f,\,g$ and $h$, respectively, satisfying
$$
\tilde{h} \circ \tilde{f}=\tilde{g} \circ \tilde{h}.
$$
Then we have
$$
\pi \circ \tilde{h}=h \circ \pi,
\qquad \pi \circ \tilde{f}=f \circ \pi,
\qquad \pi \circ \tilde{g}=g \circ \pi,
$$
where $\pi:\,E \rightarrow B$ is the covering map.
Hence for any $x \in E$, we get
\begin{gather*}
\pi \circ \tilde{h} \circ \tilde{f}(x)=h \circ \pi \circ \tilde{f}(x)=h \circ f \circ \pi(x),
\\
\pi \circ \tilde{g} \circ \tilde{h}(x)=g \circ \pi \circ \tilde{h}(x)=g \circ h \circ \pi(x),
\end{gather*}
and then
$$
h \circ f \circ \pi(x)=g \circ h \circ \pi(x).
$$
Since $\pi:\,E \rightarrow B$ is a surjection, we have
$$
h \circ f(x)=g \circ h(x),\qquad \forall\,x \in B,
$$
that means $f$ and $g$ are topologically conjugate.

Next, we prove the necessity.

Assume that $\pi:\,E \rightarrow B$ is a $|\Lambda|$-fold covering map, where $\Lambda$ is an index set, and $|\Lambda|$ is the cardinal number of $\Lambda$.
Fix one lift of $f$ denoted by $\tilde{f}$ and one lift of $h$ denoted by $\tilde{h}$.
Then the known condition $h \circ f=g \circ h$ indicates that $\tilde{h} \circ \tilde{f}$ is not only a lift of $h \circ f$, but also a lift of $g \circ h$.
One can see that there exist $|\Lambda|$ lifts of $g$ denoted by $\{\tilde{g}_\alpha\}_{\alpha \in \Lambda}$,
so $\{\tilde{g}_\alpha \circ \tilde{h}\}_{\alpha \in \Lambda}$ are the $|\Lambda|$ lifts of $g \circ h$. Thus, there exists some $\alpha_0 \in \Lambda$ such that
$$
\tilde{h} \circ \tilde{f}=\tilde{g}_{\alpha_0} \circ \tilde{h},
$$
that means $\tilde{f}$ and $\tilde{g}_{\alpha_0}$ are topologically conjugate.
\end{proof}

\begin{lemma}\label{lem:6}\cite{KR-1989}
Let $W$ be a topological 4-dimensional $h$-cobordism of lens spaces
$L_0$ and $L_1$ which preserves an orientation and generator of fundamental group.
Then $W \cong L_0 \times [0, 1]$.
\end{lemma}

\section{Topologically conjugate classification of the left actions on ${\rm SU}(2)$}
\label{sec:3}
In this section, we define the rotation numbers of the left actions in the set
$$
\mathfrak{M}_{T_{{\rm SU}(2)}}=\{L_g:\,{\rm SU}(2)\rightarrow {\rm SU}(2);\,\, g \in T_{{\rm SU}(2)}\},
$$
and utilize the rotation numbers defined to give the topologically conjugate classification of the left actions in $\mathfrak{M}_{T_{{\rm SU}(2)}}$,
and then give the topologically conjugate
classification of all left actions on ${\rm SU}(2)$.
Furthermore, we also study the relationship among their topological conjugacy, algebraic conjugacy and smooth conjugacy.

Now, fix one maximal torus of ${\rm SU}(2)$ as
$$
T_{{\rm SU}(2)}=\left\{\left(\begin{array}{cc}
z&0\\
0&\bar{z}
\end{array}\right); z={\textrm{e}}^{2\pi{{\rm i}}\theta} \in \mathbb{C},\,\theta \in [0, 1)\right\}.
$$
Obviously,
$$
T_{{\rm SU}(2)} \cong S^1.
$$
Define an isomorphism $\Phi:\,T_{{\rm SU}(2)} \rightarrow S^1$ by
$$
\Phi:\,u\mapsto z,
\qquad \forall\, u=\left(\begin{array}{cc}
z&0\\
0&\bar{z}
\end{array}\right) \in T_{{\rm SU}(2)}.
$$
Then for any $L_g \in \mathfrak{M}_{T_{{\rm SU}(2)}}$,
set
$$
f=\Phi \circ L_g|_{T_{{\rm SU}(2)}} \circ \Phi^{-1}.
$$
One can see that $f$ is a rotation of $S^1$ satisfying
$$
f(z)={\textrm{e}}^{2\pi{{\rm i}}\theta}z,\qquad\forall\, z \in S^1,
$$
where $\theta \in [0, 1)$.
This fact indicates that $L_g|_{T_{{\rm SU}(2)}}$ is topologically conjugate to some rotation $f$ of $S^1$.
Thus, according to Definition \ref{def:3}, we define the rotation number of the left action $L_g$ under the representation $(T_{{\rm SU}(2)}, \Phi)$ by
$$
\rho(L_g)\triangleq\rho(f)=\theta,
\qquad \theta \in [0,1).
$$

\begin{theorem}\label{the:1}
For the left actions $L_g,\,L_{g'} \in \mathfrak{M}_{T_{{\rm SU}(2)}}$,
$L_g$ and $L_{g'}$ are topologically conjugate if and only if
$$
\rho(L_{g'})=\pm\,\rho(L_g)\quad(\rm{mod}\,\,\,\mathbb{Z}).
$$
\end{theorem}

\begin{proof}
First, we prove the sufficiency.

If $\rho(L_{g'})=\rho(L_g)$, we have $L_g=L_{g'}$.
Hence the identity map of ${\rm SU}(2)$ is a topological conjugacy from $L_g$ to $L_{g'}$.
If $\rho(L_{g'})=1-\rho(L_g)$, set
$$
g=\left(\begin{array}{cc}
z&0\\
0&\bar{z}
\end{array}\right),\qquad g'=\left(\begin{array}{cc}
\bar{z}&0\\
0&z
\end{array}\right).
$$
Take a homeomorphism $h:\,{\rm SU}(2) \rightarrow {\rm SU}(2)$ defined by
$$
h:\,u \mapsto \left(\begin{array}{cc}
0&1\\
-1&0
\end{array}\right)u\left(\begin{array}{cc}
0&-1\\
1&0
\end{array}\right),
\qquad \forall\, u=\left(\begin{array}{cc}
a&b\\
c&d
\end{array}\right) \in {\rm SU}(2).
$$
In fact, $h$ is an inner automorphism of ${\rm SU}(2)$.
Then through a simple calculation, we obtain
$$
h \circ L_g(u)=\left(\begin{array}{cc}
d\bar{z}&-c\bar{z}\\
-bz&az
\end{array}\right)=L_{g'} \circ h(u),
\quad \forall\, u=\left(\begin{array}{cc}
a&b\\
c&d
\end{array}\right) \in {\rm SU}(2),
$$
so $L_g$ and $L_{g'}$ are topologically conjugate.

Next, we prove the necessity.

Suppose that $L_g$ and $L_{g'}$ are topologically conjugate.
Then there exists a self-homeomorphism $h$ of ${\rm SU}(2)$ such that
$$
h \circ L_g=L_{g'} \circ h,\qquad h(e)=e.
$$
and then
$$
h \circ L_g^n(e)=L_{g'}^n \circ h(e),\qquad \forall\,n\in\mathbb{Z}_+,
$$
so we have
$$
h(L_g^n(e))=L_{g'}^n(e),\qquad \forall\,n\in\mathbb{Z}_+.
$$
This fact implies that
$$
h(\overline{Orb_{L_g}(e)})=\overline{Orb_{L_{g'}}(e)}.
$$
Thus, we divide following proof into three cases.

(1) Assume that $\rho(L_g)=q/p$ is a rational number with ${\rm gcd}\,(p, q)=1$,
and $\rho(L_{g'})$ is an irrational number.
It is easy to see that $\overline{Orb_{L_g}(e)}$ is a discrete set containing $p$ elements, but
$$
\overline{Orb_{L_{g'}}(e)}=T_{{\rm SU}(2)} \cong S^1.
$$
This fact is in contradiction to $h$ being a homeomorphism.
Therefore, $L_g$ and $L_{g'}$ can not be topologically conjugate in this case.

(2) Assume that $\rho(L_g)$ and $\rho(L_{g'})$ are both irrational numbers.
Then one can see that
$$
\overline{Orb_{L_g}(e)}=\overline{Orb_{L_{g'}}(e)}=T_{{\rm SU}(2)} \cong S^1.
$$
Naturally, we can restrict the dynamical system to the maximal torus $T_{{\rm SU}(2)}$, and then obtain
$$
h|_{T_{{\rm SU}(2)}} \circ L_g|_{T_{{\rm SU}(2)}}=L_{g'}|_{T_{{\rm SU}(2)}} \circ h|_{T_{{\rm SU}(2)}},
$$
where $h|_{T_{{\rm SU}(2)}}$ is a self-homeomorphism of $T_{{\rm SU}(2)}$,
so $L_g|_{T_{{\rm SU}(2)}}$ and $L_{g'}|_{T_{{\rm SU}(2)}}$ are topologically conjugate.
And the previous discussion indicates that $L_g|_{T_{{\rm SU}(2)}}$ and $L_{g'}|_{T_{{\rm SU}(2)}}$ are topologically conjugate to some rotations $f$ and $f'$ of $S^1$, respectively, and
$$
\rho(L_g)=\rho(L_g|_{T_{{\rm SU}(2)}})=\rho(f),\qquad \rho(L_{g'})=\rho(L_{g'}|_{T_{{\rm SU}(2)}})=\rho(f').
$$
Hence $f$ and $f'$ are topologically conjugate.
Therefore, according to Proposition \ref{prop:1}, we have
$$
\rho(L_{g'})=\pm\,\rho(L_g)\quad(\rm{mod}\,\,\,\mathbb{Z}).
$$

(3) Assume that
$$
\rho(L_g)=\dfrac{q}{p},\qquad \rho(L_{g'})=\dfrac{q'}{p'}
$$
are both rational numbers with ${\rm gcd}\,(p, q)={\rm gcd}\,(p', q')=1$.
Obviously, $L_g$ is a $p$-periodic left action, and $L_{g'}$ is a $p'$-periodic left action, so $p=p'$.
Consider the topological structure of ${\rm SU}(2)$,
and we can identify ${\rm SU}(2)$ with $S^3$.
Then it is easy to see that $L_g$ and $L_{g'}$ are both orbit equivalent to a periodic homeomorphism $f:\,S^3 \rightarrow S^3$ defined by
$$
f:\,(z_1, z_2)\mapsto({\textrm{e}}^{2\pi{\rm i}/p}z_1, {\textrm{e}}^{-2\pi{\rm i}/p}z_2),
\qquad \forall\, (z_1, z_2) \in S^3.
$$
Notice that
$$
\overline{Orb_{L_g}(e)}=\overline{Orb_{L_{g'}}(e)} \cong \mathbb{Z}_p,
$$
and then it follows from Definition \ref{def:2.5} that
$$
{{\rm SU}(2)}/\overline{Orb_{L_g}(e)}={\rm SU}(2)/\overline{Orb_{L_{g'}}(e)} \cong {\rm SU}(2)/\mathbb{Z}_p \cong S^3/f \cong L(p, -1).
$$
As well-known, ${\rm SU}(2)$ is the universal covering space of $L(p, -1)$, and
$$
\pi_1(L(p, -1)) \cong \mathbb{Z}_p.
$$
According to the above analysis, the definition of the fundamental group of $L(p, -1)$ and Lemma \ref{lem:1}, we get a commutative diagram as follows.
$$
\xymatrixcolsep{3pc}
\xymatrix{
{\,\,I\,\,} \ar[r]^{\tilde{\alpha}} \ar[rd]_{\alpha}
& {\,\,{\rm SU}(2)\,\,} \ar[r]^{h} \ar[d]^{\pi} &
{\,\,{\rm SU}(2)\,\,} \ar[d]^{\pi}\\
&{\,\,L(p, -1)\,\,} \ar[r]^{\bar{h}} & {\,\,L(p, -1)\,\,}.}
$$
In this diagram, $I=[0, 1]$, $h$ is a topological conjugacy from $L_g$ to $L_{g'}$, $\pi$ is the universal covering map from ${\rm SU}(2)$ to $L(p, -1)$, $\alpha$ is a loop in $L(p, -1)$ with the base point $\pi(e)$, where $e=\left(\begin{array}{cc}
1&0\\
0&1
\end{array}\right)$ is the identity element of ${\rm SU}(2)$,
$\tilde{\alpha}$ is the unique lift of $\alpha$ satisfying $\tilde{\alpha}(0)=e$, and $\bar{h}$ is a self-homeomorphism of $L(p, -1)$ induced by $h$.
Set
$$
\pi_1(L(p,-1))=\{[\alpha_k];\,\,\,k=1,\,2,\,\cdots,\,p\}\,\cong\,\mathbb{Z}_p,
$$
where $\alpha_k$ is a loop in $L(p, -1)$ with the base point $\pi(e)$,
and $[\alpha_k]$ denotes the homotopy class of $\alpha_k$.
Let $\tilde{\alpha}_k$ be the unique lift of $\alpha_k$ satisfying $\tilde{\alpha}_k(0)=e$.
Then $\tilde{\alpha}_k$ is a continuous curve in ${\rm SU}(2)$ connecting the elements $e=\left(\begin{array}{cc}
1&0\\
0&1
\end{array}\right)$ and $\left(\begin{array}{cc}
{\textrm{e}}^{2\pi{\rm i} k/p}&0\\
0&{\textrm{e}}^{-2\pi{\rm i} k/p}
\end{array}\right)$.
By the known conditions
$$
h \circ L_g=L_{g'} \circ h,\quad h(e)=e,
\quad \rho(L_g)=\dfrac{q}{p},\quad \rho(L_{g'})=\dfrac{q'}{p},
$$
we have
$$
h(\tilde{\alpha}_q)=\tilde{\alpha}_q',
$$
and then the above commutative diagram implies that
$$
\bar{h}(\alpha_q)=\alpha_q'.
$$
Let $\bar{h}_*:\,\pi_1(L(p, -1)) \rightarrow \pi_1(L(p, -1))$ be the automorphism of $\pi_1(L(p, -1))$ induced by $\bar{h}$.
Then we have
$$
\bar{h}_*([\alpha_q])=[\alpha_q'].
$$
Notice that $\bar{h}$ is a degree 1 orientation-preserving self-homeomorphism of $L(p, -1)$,
and then if follows from Lemma \ref{lem:2} that
$$
\bar{h}_*([\alpha_q])=\pm\,[\alpha_q]=[\alpha_q'].
$$
Thus, we obtain
$$
q'=q \quad\hbox{or}\quad q'=p-q,
$$
so
$$
\rho(L_{g'})=\pm\,\rho(L_g)\quad(\rm{mod}\,\,\,\mathbb{Z}).
$$
\end{proof}

Finally, we investigate the relationship among the topological conjugacy, algebraic conjugacy and smooth conjugacy of the left actions on ${\rm SU}(2)$.

\begin{corollary}\label{prop:4}
Suppose that $L_g,\,L_{g'}$ are two left actions on ${\rm SU}(2)$.
Then the following conditions are equivalent.

(a) $L_g$ and $L_{g'}$ are topologically conjugate;

(b) $L_g$ and $L_{g'}$ are algebraically conjugate;

(c) $L_g$ and $L_{g'}$ are smooth conjugate.
\end{corollary}

\begin{proof}
(b) $\Longrightarrow$ (a) and (c) $\Longrightarrow$ (a) are obviously true.

Assume that $L_g$ and $L_{g'}$ are topologically conjugate.
Proposition \ref{prop:3} indicates that there exist some $s,\,s' \in {\rm SU}(2)$ and $t,\,t' \in T_{{\rm SU}(2)}$ such that
$$
L_s \circ L_g=L_t \circ L_s,\qquad L_{s'} \circ L_{g'}=L_{t'} \circ L_{s'}.
$$
Hence $L_g$ and $L_t$ are topologically conjugate,  $L_{g'}$ and $L_{t'}$ are topologically conjugate,
so $L_t$ and $L_{t'}$ are topologically conjugate.
Then according to Theorem \ref{the:1}, we have
$$
\rho(L_{t'})=\pm\,\rho(L_t)\quad({\rm{mod}}\,\,\,\mathbb{Z}).
$$
Thus, combining with the proof of the sufficiency of Theorem \ref{the:1},
if $\rho(L_{t'})=\rho(L_t)$,
we can choose the topological conjugacy $h$ equal to ${\rm{id}}_{{\rm SU}(2)}$,
and then $h$ is obviously a smooth inner automorphism of ${\rm SU}(2)$,
i.e., $L_t$ and $L_{t'}$ are algebraically conjugate and smooth conjugate.
If $\rho(L_{t'})=1-\rho(L_t)$,
we can choose the topological conjugacy $h$ defined by
$$
h:\,u\mapsto\left(\begin{array}{cc}
0&1\\
-1&0
\end{array}\right)u\left(\begin{array}{cc}
0&-1\\
1&0
\end{array}\right),\qquad \forall\, u \in {\rm SU}(2),
$$
and then one can see that $h$ is a smooth inner automorphism of ${\rm SU}(2)$,
i.e., $L_t$ and $L_{t'}$ are algebraically conjugate and smooth conjugate.
Take smooth inner automorphisms $h,\,h'$ satisfying
$$
h(u)=sus^{-1},\qquad h'(u)=s'us'^{-1},
\qquad \forall\, u \in {\rm SU}(2).
$$
It is easy to verify that
$$
h \circ L_g=L_t \circ h,\qquad h' \circ L_{g'}=L_{t'} \circ h',
$$
that means $L_g$ and $L_t$ are algebraically conjugate and smooth conjugate,  $L_{g'}$ and $L_{t'}$ are algebraically conjugate and smooth conjugate,
so $L_g$ and $L_{g'}$ are algebraically conjugate and smooth conjugate.
Therefore, (a) $\Longrightarrow$ (b) and (a) $\Longrightarrow$ (c) have been proved.
\end{proof}

\section{Topologically conjugate classification of the left actions on ${\rm U}(2)$}
\label{sec:4}

In this section, we define the rotation vectors of the left actions in the set
$$
\mathfrak{M}_{T_{{\rm U}(2)}}=\{L_g:\,{\rm U}(2)\rightarrow {\rm U}(2);\,\, g \in T_{{\rm U}(2)}\},
$$
and utilize the rotation vectors defined to give the topologically conjugate
classification of the left actions in $\mathfrak{M}_{T_{{\rm U}(2)}}$,
and then give the topologically conjugate
classification of all left actions on ${\rm U}(2)$.
Furthermore, we study the relationship among their topological conjugacy, algebraic conjugacy and smooth conjugacy.

Now, fix one maximal torus of ${\rm U}(2)$ as
$$
T_{{\rm U}(2)}=\left\{\left(\begin{array}{cc}
\lambda z&0\\
0&\bar{z}
\end{array}\right); z={\textrm{e}}^{2\pi{{\rm i}}\theta},
\,\lambda={\textrm{e}}^{2\pi{{\rm i}}\varphi} \in \mathbb{C},\,\,\theta,\,\varphi \in [0, 1)\right\}.
$$
Obviously,
$$
{\rm U}(2) \cong {\rm SU}(2) \times S^1,\qquad T_{{\rm U}(2)} \cong T^2.
$$
Define an isomorphism $\Phi:\,T_{{\rm U}(2)} \rightarrow T^2$ by
$$
\Phi:\,u\mapsto\left(\begin{array}{cc}
z&0\\
0&\lambda
\end{array}\right),
\qquad \forall\, u=\left(\begin{array}{cc}
\lambda z&0\\
0&\bar{z}
\end{array}\right) \in T_{{\rm U}(2)}.
$$
Then for any $L_g \in \mathfrak{M}_{T_{{\rm U}(2)}}$,
set
$$
f=\Phi \circ L_g|_{T_{{\rm U}(2)}} \circ \Phi^{-1}.
$$
One can see that $f$ is a rotation of $T^2$ satisfying
$$
f(t)=\left(\begin{array}{cc}
{\textrm{e}}^{2\pi{\rm i}\theta}&0\\
0&{\textrm{e}}^{2\pi{\rm i}\varphi}
\end{array}\right)t,
\qquad \forall\, t \in T^2,
$$
where $\theta,\,\varphi \in [0, 1)$.
This fact indicates that $L_g|_{T_{{\rm U}(2)}}$ is topologically conjugate to some rotation $f$ of $T^2$.
Thus, according to Definition \ref{def:3}, we define the rotation vector of the left action $L_g$ under the representation $(T_{{\rm U}(2)}, \Phi)$ by
$$
\rho(L_g)\triangleq\rho(f)=\left(\begin{array}{c}
\theta\\
\varphi
\end{array}\right),
\qquad \theta,\,\varphi \in [0,1).
$$

\begin{theorem}\label{the:2}
For the left actions $L_g,\,L_{g'} \in \mathfrak{M}_{T_{{\rm U}(2)}}$ with
$$
\rho(L_g)=\left(\arraycolsep=1pt\begin{array}{c}
\theta\\
\varphi
\end{array}\right),\qquad \rho(L_{g'})=\left(\arraycolsep=1pt\begin{array}{c}
\theta'\\
\varphi'
\end{array}\right),
$$
$L_g$ and $L_{g'}$ are topologically conjugate if and only if
$$
\left\{
\begin{array}{l}
\theta'=\pm\,\theta+n\varphi+n',\qquad n,\,n'\in\mathbb{Z},
\\
\varphi'=\pm\,\varphi\quad(\rm{mod}\,\,\,\mathbb{Z}).
\end{array}
\right.
$$
\end{theorem}

\begin{proof}
Firstly, we prove the sufficiency.

For the known conditions
$$
\rho(L_g)=\left(\arraycolsep=1pt\begin{array}{c}
\theta\\
\varphi
\end{array}\right),\qquad \rho(L_{g'})=\left(\arraycolsep=1pt\begin{array}{c}
\theta'\\
\varphi'
\end{array}\right),
$$
and
$$
\left\{
\begin{array}{l}
\theta'=\pm\,\theta+n\varphi+n',\qquad n,\,n'\in\mathbb{Z},
\\[0.25 cm]
\varphi'=\pm\,\varphi\quad(\rm{mod}\,\,\,\mathbb{Z}),
\end{array}
\right.
$$
we see that there are four different cases.

(1) Assume that
$$
\left\{
\begin{array}{l}
\theta'=\theta+n\varphi+n',\qquad n,\,n'\in\mathbb{Z},
\\[0.25 cm]
\varphi'=\varphi.
\end{array}
\right.
$$
Define $h_1,\,h'_1:\,{\rm U}(2) \rightarrow {\rm U}(2)$ by
$$
h_1:\,u\mapsto\left(\begin{array}{cc}
\lambda^n&0\\
0&\bar{\lambda}^n
\end{array}\right)u,
\qquad \forall\, u=\left(\begin{array}{cc}
\lambda z_1&-\lambda\bar{z_2}\\
z_2&\bar{z_1}
\end{array}\right) \in {\rm U}(2),
$$
$$
h'_1:\,u\mapsto\left(\begin{array}{cc}
\bar{\lambda}^n&0\\
0&\lambda^n
\end{array}\right)u,
\qquad \forall\, u=\left(\begin{array}{cc}
\lambda z_1&-\lambda\bar{z_2}\\
z_2&\bar{z_1}
\end{array}\right) \in {\rm U}(2),
$$
respectively, where $z_1,\,z_2,\,\lambda \in \mathbb{C},\,\,\,|z_1|^2+|z_2|^2=1,\,|\lambda|=1$.
It is easy to see that $h_1$ and $h'_1$ are both continuous,
and $h'_1 \circ h_1={\rm id}_{{\rm U}(2)}$,
where ${\rm id}_{{\rm U}(2)}$ is the identity map of ${\rm U}(2)$.
Hence $h_1$ is a self-homeomorphism of ${\rm U}(2)$.
Moreover, it follows from a simple calculation that for any $u=\left(\begin{array}{cc}
\lambda z_1&-\lambda\bar{z_2}\\
z_2&\bar{z_1}
\end{array}\right) \in {\rm U}(2)$,
$$
h_1 \circ L_g(u)=\left(\begin{array}{cc}
{\textrm{e}}^{2\pi{\rm i}\varphi}{\textrm{e}}^{2\pi{\rm i}(\theta+n\varphi)}&0\\
0&{\textrm{e}}^{-2\pi{\rm i}(\theta+n\varphi)}
\end{array}\right)\left(\begin{array}{cc}
\lambda^n&0\\
0&\bar{\lambda}^n
\end{array}\right)u=L_{g'} \circ h_1(u).
$$
Therefore, $L_g$ and $L_{g'}$ are topologically conjugate.

(2) Assume that
$$
\left\{
\begin{array}{l}
\theta'=\theta+n\varphi+n',\qquad n,\,n'\in\mathbb{Z},
\\[0.25 cm]
\varphi'=1-\varphi.
\end{array}
\right.
$$
Define $h_2,\,h'_2:\,{\rm U}(2) \rightarrow {\rm U}(2)$ by
$$
h_2:\,u\mapsto\left(\begin{array}{cc}
\lambda^{n-2}&0\\
0&\bar{\lambda}^n
\end{array}\right)u,
\qquad \forall\, u=\left(\begin{array}{cc}
\lambda z_1&-\lambda\bar{z_2}\\
z_2&\bar{z_1}
\end{array}\right) \in {\rm U}(2),
$$
$$
h'_2:\,u\mapsto\left(\begin{array}{cc}
\lambda^{n-2}&0\\
0&\bar{\lambda}^n
\end{array}\right)u,
\qquad \forall\, u=\left(\begin{array}{cc}
\lambda z_1&-\lambda\bar{z_2}\\
z_2&\bar{z_1}
\end{array}\right) \in {\rm U}(2),
$$
respectively, where $z_1,\,z_2,\,\lambda \in \mathbb{C},\,\,\,|z_1|^2+|z_2|^2=1,\,|\lambda|=1$.
Similar to Case (1), we can obtain that $h_2$ is a self-homeomorphism of ${\rm U}(2)$ satisfying for any $u=\left(\begin{array}{cc}
\lambda z_1&-\lambda\bar{z_2}\\
z_2&\bar{z_1}
\end{array}\right) \in {\rm U}(2)$,
$$
h_2 \circ L_g(u)=\left(\begin{array}{cc}
{\textrm{e}}^{2\pi{\rm i}(-\varphi)}{\textrm{e}}^{2\pi{\rm i}(\theta+n\varphi)}&0\\
0&{\textrm{e}}^{-2\pi{\rm i}(\theta+n\varphi)}
\end{array}\right)\left(\begin{array}{cc}
\lambda^{n-2}&0\\
0&\bar{\lambda}^n
\end{array}\right)u=L_{g'} \circ h_2(u).
$$
Therefore, $L_g$ and $L_{g'}$ are topologically conjugate.

(3) Assume that
$$
\left\{
\begin{array}{l}
\theta'=-\theta+n\varphi+n',\qquad n,\,n'\in\mathbb{Z},
\\[0.25 cm]
\varphi'=\varphi.
\end{array}
\right.
$$
Define $h_3,\,h'_3:\,{\rm U}(2) \rightarrow {\rm U}(2)$ by
$$
h_3:\,u\mapsto\left(\begin{array}{cc}
\lambda^n&0\\
0&\bar{\lambda}^n
\end{array}\right)\left(\begin{array}{cc}
\lambda\bar{z}_1&-\lambda z_2\\
\bar{z}_2&z_1
\end{array}\right),
\quad \forall\, u=\left(\begin{array}{cc}
\lambda z_1&-\lambda\bar{z_2}\\
z_2&\bar{z_1}
\end{array}\right) \in {\rm U}(2).
$$
$$
h'_3:\,u\mapsto\left(\begin{array}{cc}
\lambda^n&0\\
0&\bar{\lambda}^n
\end{array}\right)\left(\begin{array}{cc}
\lambda\bar{z}_1&-\lambda z_2\\
\bar{z}_2&z_1
\end{array}\right),
\quad \forall\, u=\left(\begin{array}{cc}
\lambda z_1&-\lambda\bar{z_2}\\
z_2&\bar{z_1}
\end{array}\right) \in {\rm U}(2),
$$
respectively, where $z_1,\,z_2,\,\lambda \in \mathbb{C},\,\,\,|z_1|^2+|z_2|^2=1,\,|\lambda|=1$.
Similar to Case (1), we can obtain that $h_3$ is a self-homeomorphism of ${\rm U}(2)$ satisfying for any $u=\left(\begin{array}{cc}
\lambda z_1&-\lambda\bar{z_2}\\
z_2&\bar{z_1}
\end{array}\right) \in {\rm U}(2)$,
$$
h_3 \circ L_g(u)=\left(\begin{array}{cc}
{\lambda^n}{\textrm{e}}^{2\pi{\rm i}\varphi}{\textrm{e}}^{2\pi{\rm i}(\theta+n\varphi)}&0\\
0&{\bar{\lambda}^n}{\textrm{e}}^{-2\pi{\rm i}(\theta+n\varphi)}
\end{array}\right)\left(\begin{array}{cc}
\lambda\bar{z}_1&-\lambda z_2\\
\bar{z}_2&z_1
\end{array}\right)=L_{g'} \circ h_3(u).
$$
Therefore, $L_g$ and $L_{g'}$ are topologically conjugate.

(4) Assume that $$
\left\{
\begin{array}{l}
\theta'=-\theta+n\varphi+n',\qquad n,\,n'\in\mathbb{Z},
\\[0.25 cm]
\varphi'=1-\varphi.
\end{array}
\right.
$$
Define $h_4,\,h'_4:\,{\rm U}(2) \rightarrow {\rm U}(2)$ by
$$
h_4:\,u\mapsto\left(\begin{array}{cc}
\lambda^{n-2}&0\\
0&\bar{\lambda}^n
\end{array}\right)\left(\begin{array}{cc}
\lambda\bar{z}_1&-\lambda z_2\\
\bar{z}_2&z_1
\end{array}\right),
\quad \forall\, u=\left(\begin{array}{cc}
\lambda z_1&-\lambda\bar{z_2}\\
z_2&\bar{z_1}
\end{array}\right) \in {\rm U}(2),
$$
$$
h'_4:\,u\mapsto\left(\begin{array}{cc}
\bar{\lambda}^{n-2}&0\\
0&\lambda^n
\end{array}\right)\left(\begin{array}{cc}
\lambda\bar{z}_1&-\lambda z_2\\
\bar{z}_2&z_1
\end{array}\right),
\quad \forall\, u=\left(\begin{array}{cc}
\lambda z_1&-\lambda\bar{z_2}\\
z_2&\bar{z_1}
\end{array}\right) \in {\rm U}(2),
$$
respectively, where $z_1,\,z_2,\,\lambda \in \mathbb{C},\,\,\,|z_1|^2+|z_2|^2=1,\,|\lambda|=1$.
Similar to Case (1), we can obtain that $h_4$ is a self-homeomorphism of ${\rm U}(2)$ satisfying for any $u=\left(\begin{array}{cc}
\lambda z_1&-\lambda\bar{z_2}\\
z_2&\bar{z_1}
\end{array}\right) \in {\rm U}(2)$,
$$
h_4 \circ L_g(u)=\left(\begin{array}{cc}
{\lambda^{n-2}}{\textrm{e}}^{2\pi{\rm i}\varphi}{\textrm{e}}^{2\pi{\rm i}(\theta+n\varphi)}&0\\
0&{\bar{\lambda}^n}{\textrm{e}}^{-2\pi{\rm i}(\theta+n\varphi)}
\end{array}\right)\left(\begin{array}{cc}
\lambda\bar{z}_1&-\lambda z_2\\
\bar{z}_2&z_1
\end{array}\right)=L_{g'} \circ h_4(u).
$$
Therefore, $L_g$ and $L_{g'}$ are topologically conjugate.

Next, we prove the necessity.

Notice that ${\rm U}(2) \cong {\rm SU}(2) \times S^1$,
and there exists a topological conjugacy from $L_g$ to $L_{g'}$ such that
$$
h \circ L_g=L_{g'} \circ h,\qquad h(e)=e.
$$
Let us investigate the following diagram.
$$
\xymatrixcolsep{3pc}
\xymatrix{
{\,\,T^2\,\,} \ar[d]_-{f_1} \ar[r]^-{i}
& {\,\,{{\rm SU}}(2) \times S^1\,\,} \ar[d]_-{L_g} \ar[r]^-{h}
& {\,\,{{\rm SU}}(2) \times S^1\,\,} \ar[d]_-{L_{g'}} \ar[r]^-{\pi}
& {\,\,S^1\,\,} \ar[d]_-{f_2}\\
{\,\,T^2\,\,} \ar[r]_-{i}
& {\,\,{{\rm SU}}(2) \times S^1\,\,} \ar[r]_-{h}
& {\,\,{{\rm SU}}(2) \times S^1\,\,} \ar[r]_-{\pi}
& {\,\,S^1\,\,}.}
$$
In this diagram, $i$ is a nature inclusion map preserving the identity elements, $\pi$ is a projection defined by
$$
\pi:\,(u, t) \mapsto t,\qquad \forall\,u \in {\rm SU}(2),\,\,\,t \in S^1,
$$
$f_1$ is a rotation of $T^2$, $f_2$ is rotation of $S^1$,
$L_g,\,L_{g'}$ are two left actions on ${\rm U}(2)$,
$h$ is a topological conjugacy from $L_g$ to $L_{g'}$,
and
$$
\rho(f_1)=\rho(L_g)=\left(\arraycolsep=1pt\begin{array}{c}
\theta\\
\varphi
\end{array}\right),
\qquad \rho(L_{g'})=\left(\arraycolsep=1pt\begin{array}{c}
\theta'\\
\varphi'
\end{array}\right),\qquad \rho(f_2)=\varphi'.
$$
It is easy to see that
$$
i \circ f_1=L_g \circ i,\qquad \pi \circ L_{g'}=f_2 \circ \pi.
$$
Then the diagram above is commutative.
Thus, we get
$$
f \circ f_1=f_2 \circ f,
$$
where $f=\pi \circ h \circ i:\,T^2 \rightarrow S^1$ is a continuous surjection.
As well-known,
$$
\pi_1({\rm SU}(2) \times S^1) \cong \pi_1(S^1) \cong \mathbb{Z},
\qquad \pi_1(T^2) \cong \mathbb{Z}\oplus\mathbb{Z}.
$$
Then for any $(m, n) \in \pi_1(T^2) \cong \mathbb{Z} \oplus \mathbb{Z}$,
$l,\,l' \in \pi_1({\rm SU}(2) \times S^1)$, we have
$$
i_*(m, n)=\pm\,n,
\qquad h_*(l)=\pm\,l,
\qquad \pi_*(l')=\pm\,l',
$$
where $i_*,\,h_*$ and $\pi_*$ are the group homomorphisms on fundamental groups induced by $i,\,h$ and $\pi$, respectively.
Hence the group homomorphism $f_*$ on fundamental groups induced by $f$ satisfies
$$
f_*(m, n)=(\pi \circ h \circ i)_*(m, n)
=\pi_* \circ h_* \circ i_*(m, n)=\pm\,n.
$$
So It follows from Lemma \ref{lem:3} that
$$
\varphi'=\pm\,\varphi\quad(\rm{mod}\,\,\,\mathbb{Z}).
$$
Consequently, it suffices to prove
$$
\theta'=\pm\,\theta+n\varphi+n',\qquad n,\,n'\in\mathbb{Z}.
$$

Firstly, let us review the definitions of the rational dependence and rational independence of real numbers.
Assume that $\alpha$ and $\beta$ are real numbers.
If there exist some $l,\,m,\,n \in \mathbb{Z}$ which are not all equal to $0$ such that
$$
l+m\alpha+n\beta=0,
$$
$\alpha$ and $\beta$ are said to be rationally dependent.
If not, $\alpha$ and $\beta$ are said to be rationally independent.

Let $L_g,\,L_{g'}$ be two left actions on ${\rm U}(2)$ with
$$
\rho(L_g)=\left(\arraycolsep=1pt\begin{array}{c}
\theta\\
\varphi
\end{array}\right),\qquad \rho(L_{g'})=\left(\arraycolsep=1pt\begin{array}{c}
\theta'\\
\varphi'
\end{array}\right),
$$
and $h$ be a topological conjugacy from $L_g$ to $L_{g'}$ preserving the identify element.
Then according to the above definitions,
one can see that there are four different cases in detail,
and then we divide the following proof into four parts corresponding to these cases, respectively.

{\bf Case 1.} Assume that $\theta$ and $\varphi$ are rationally dependent,
but $\theta'$ and $\varphi'$ are rationally independent.
According to the discussion in Section \ref{sec:3}, we have
$$
h(\overline{Orb_{L_g}(e)})=\overline{Orb_{L_{g'}}(e)}.
$$
Notice that $\overline{Orb_{L_g}(e)}$ consists of finite elements or some mutually disjoint simple closed curves, but $\overline{Orb_{L_{g'}}(e)} \cong T^2$.
(The reasons will be given in the following three cases.)
This fact is in contradiction to $h$ being a homeomorphism.
Thus, $L_g$ and $L_{g'}$ can not be topologically conjugate in this case.

{\bf Case 2.} Assume that $\theta$ and $\varphi$ are rationally independent,
$\theta'$ and $\varphi'$ are rationally independent, too.
Similarly, we have
$$
h(\overline{Orb_{L_g}(e)})=\overline{Orb_{L_{g'}}(e)}.
$$
Notice that
$$
\overline{Orb_{L_g}(e)}=\overline{Orb_{L_{g'}}(e)}=T_{{\rm U}(2)} \cong T^2.
$$
Naturally, we can restrict the dynamical system to the maximal torus $T_{{\rm U}(2)}$, and then obtain
$$
h|_{T_{{\rm U}(2)}} \circ L_g|_{T_{{\rm U}(2)}}=L_{g'}|_{T_{{\rm U}(2)}} \circ h|_{T_{{\rm U}(2)}},
$$
where $h|_{T_{{\rm U}(2)}}$ is a self-homeomorphism of $T_{{\rm U}(2)}$, so $L_g|_{T_{{\rm U}(2)}}$ and $L_{g'}|_{T_{{\rm U}(2)}}$ are topologically conjugate.
And the previous discussion indicates that $L_g|_{T_{{\rm U}(2)}}$ and $L_{g'}|_{T_{{\rm U}(2)}}$ are topologically conjugate to some rotations $f$ and $f'$ of $T^2$, respectively, and
$$
\rho(L_g)=\rho(L_g|_{T_{{\rm U}(2)}})=\rho(f)=\left(\arraycolsep=1pt\begin{array}{c}
\theta\\
\varphi
\end{array}\right),
\ \rho(L_{g'})=\rho(L_{g'}|_{T_{{\rm U}(2)}})=\rho(f')=\left(\arraycolsep=1pt\begin{array}{c}
\theta'\\
\varphi'
\end{array}\right).
$$
Hence $f$ and $f'$ are topologically conjugate.
Therefore, according to Proposition \ref{prop:2}, one can see that
$$
\left(\arraycolsep=1pt\begin{array}{c}
\theta'\\
\varphi'
\end{array}\right)=\bm{A}\left(\arraycolsep=1pt\begin{array}{c}
\theta\\
\varphi
\end{array}\right)\quad(\rm{mod}\,\,\,\mathbb{Z}),
$$
where $\bm{A} \in \rm{GL}_2(\mathbb{Z})$, and then together with the result
$$
\varphi'=\pm\,\varphi\quad(\rm{mod}\,\,\,\mathbb{Z}),
$$
we obtain $\bm{A}=\left(\begin{array}{cc}
\pm\,1&n\\
0&\pm\,1
\end{array}\right)$, where $n \in \mathbb{Z}$,
so
$$
\theta'=\pm\,\theta+n\varphi+n',\qquad n,\,n'\in\mathbb{Z}.
$$

{\bf Case 3.} Assume that $\theta$ and $\varphi$ are rationally dependent, $\theta'$ and $\varphi'$ are also rationally dependent,
and $\varphi'=\pm\,\varphi\,\,\,(\rm{mod}\,\,\,\mathbb{Z})$ is an irrational number.
There are three different cases.

(i) Suppose that $\theta=q/p,\,\theta'=q'/p'$ are rational numbers with ${\rm gcd}\,(p, q)={\rm gcd}\,(p', q')=1$,
and $h$ a topological conjugacy from $L_g$ to $L_{g'}$ satisfying
$$
h(\overline{Orb_{L_g}(e)})=\overline{Orb_{L_{g'}}(e)}.
$$
Firstly, we discuss the closures of the orbits of $e$ under $L_g$ and $L_{g'}$.
For the left action $L_g$, we see that $\rho(L_g^p)=\left(\arraycolsep=1pt\begin{array}{c}
0\\
p\varphi+n
\end{array}\right),\,n \in \mathbb{Z}$.
And for $j=0, 1, \cdots, p-1$, define $F_j:\,S^1 \rightarrow \rm{U}(2)$ by
$$
F_j:\,z\mapsto\left(\begin{array}{cc}
z&0\\
0&1
\end{array}\right)\left(\begin{array}{cc}
\textrm{e}^{2\pi{\rm i} j\varphi}\textrm{e}^{2\pi{\rm i} j\theta}&0\\
0&\textrm{e}^{-2\pi{\rm i} j\theta}
\end{array}\right),\qquad \forall\, z \in S^1.
$$
One can see that every $F_j$ is an embedding map,
and consequently, every $H_j:\,S^1 \rightarrow F_j(S^1)$ is a homeomorphism.
Take a rotation $f$ of $S^1$ with $\rho(f)=p\varphi+n$.
Then it is not difficult to verify that
$$
H_j \circ f=L_g^p|_{F_j(S^1)} \circ H_j,\qquad H_j(e')=e_j,
$$
and hence
$$
H_j(\overline{Orb_f(e')})=\overline{Orb_{L_g^p}(e_j)}=F_j(S^1),
$$
where
$$
e_j=\left(\begin{array}{cc}
\textrm{e}^{2\pi{\rm i} j\varphi}\textrm{e}^{2\pi{\rm i} j\theta}&0\\
0&\textrm{e}^{-2\pi{\rm i} j\theta}
\end{array}\right),\,\,\,j=0, 1, \cdots, p-1,
$$
and $e',\,e_0=e$ are the identity elements of $S^1$ and ${\rm U}(2)$, respectively.
Thus, $\overline{Orb_{L_g^p}(e_j)} \cong S^1$.
This fact indicates that every $\overline{Orb_{L_g^p}(e_j)}$ is a simple closed curve in ${\rm{U}}(2)$.
Furthermore, we have
$$
\overline{Orb_{L_g}(e)}=\overline{Orb_{L_g^p}(e_0)}\cup\overline{Orb_{L_g^p}(e_1)}\cup
\cdots\cup\overline{Orb_{L_g^p}(e_{p-1})}.
$$
And according to the definitions of $F_j$, where $j=0, 1, \cdots, p-1$, one can see that the simple closed curves
$$
\overline{Orb_{L_g^p}(e_0)},\quad\overline{Orb_{L_g^p}(e_1)},\quad\cdots,\quad
\overline{Orb_{L_g^p}(e_{p-1})}
$$
are mutually disjoint,
so $\overline{Orb_{L_g}(e)}$ contains $p$ mutually disjoint simple closed curves in ${\rm{U}}(2)$.
Similarly, $\overline{Orb_{L_{g'}}(e)}$ contains $p'$ mutually disjoint simple closed curves in ${\rm{U}}(2)$.
Thus, $p=p'$.

Set $\theta=\dfrac{q}{p},\,\theta'=\dfrac{q'}{p}$. It is easy to see that
$$
\overline{Orb_{L_g^p}(e)}=\overline{Orb_{L_{g'}^p}(e)}=\left\{\left(\begin{array}{cc}
z\omega&0\\
0&\bar{z}
\end{array}\right);\,z^1\omega^0=1,\,z,\,\omega \in \mathbb{C},\,|z|=|\omega|=1\right\}.
$$
Then it follows from Lemma \ref{lem:4} that $\overline{Orb_{L_g^p}(e)}=\overline{Orb_{L_{g'}^p}(e)}\,\cong\,S^1$ is a subgroup of ${\rm U}(2)$, and
$$
{\rm U}(2)/\overline{Orb_{L_g^p}(e)}={\rm U}(2)/\overline{Orb_{L_{g'}^p}(e)} \cong L(1, -1) \cong S^3 \cong {\rm SU}(2).
$$
Thus, according to Lemma \ref{lem:1}, we know that there exists some homeomorphism $h_\pi:\,{\rm SU}(2) \rightarrow {\rm SU}(2)$ induced by $h$ such that
$$
\pi \circ h=h_\pi \circ \pi,
$$
where $\pi:\,{{\rm U}}(2) \rightarrow {{\rm SU}}(2)$ is the quotient map.
And then the proof of Lemma \ref{lem:4} implies that the quotient map $\pi:\,{{\rm U}}(2) \rightarrow {\rm SU}(2)$ satisfies
$$
\pi(z, \textrm{e}^{2\pi{\rm i}\alpha}, \textrm{e}^{2\pi{\rm i}\beta})=(z, \textrm{e}^{2\pi{\rm i}\alpha}).
$$
Take $L,\,L' \in \mathfrak{M}_{T_{{\rm SU}(2)}}$ with
$$
\rho(L)=\dfrac{q}{p},\qquad\rho(L')=\dfrac{q'}{p}.
$$
Then it is easy to prove that
$$
\pi \circ L_g=L \circ \pi,\qquad \pi \circ L_{g'}=L' \circ \pi.
$$
Therefore, for any $u \in {\rm U}(2)$, we have
\begin{gather*}
\pi \circ h \circ L_g(u)=h_\pi \circ \pi \circ L_g(u)=h_\pi \circ L \circ \pi(u),
\\
\pi \circ L_{g'} \circ h(u)=L' \circ \pi \circ h(u)=L' \circ h_\pi \circ \pi(u).
\end{gather*}
And then the known condition $h \circ L_g=L_{g'} \circ h$ indicates that
$$
h_\pi \circ L \circ \pi(u)=\pi \circ h \circ L_g(u)=\pi \circ L_{g'} \circ h(u)=L' \circ h_\pi \circ \pi(u),
\quad \forall\, u \in {\rm U}(2).
$$
Notice that $\pi:\,{{\rm U}}(2) \rightarrow {{\rm SU}}(2)$ is a surjection,
so we obtain
$$
h_\pi \circ L(u)=L' \circ h_\pi(u)\qquad \forall\, u \in {\rm U}(2),
$$
that means $L$ and $L'$ are topologically conjugate.
Then by Theorem \ref{the:1}, one can see that
$$
\theta'=\theta \quad\hbox{or}\quad \theta'=1-\theta.
$$
Consequently, in this case, $\theta,\,\theta'$ and $\varphi$ satisfy
$$
\theta'=\pm\,\theta+n\varphi+n',\qquad n,\,n' \in \mathbb{Z}.
$$.

(ii) Suppose that $\theta$ and $\theta'$ are both irrational numbers.
Then according to the definition of rational dependence, we have
\begin{align}\label{eq:4}
\left\{
\begin{array}{l}
k+pd\theta+qd\varphi=0,
\\
k'+p'd'\theta'+q'd'\varphi'=0,
\end{array}
\right.
\end{align}
where
$$
k,\,k',\,d,\,d',\,p,\,p',\,q, \,q' \in \mathbb{Z},\quad p,\,p',\,q,\,q',\,d,\,d'  \neq 0, \quad p,\,p'>0,
$$
and
$$
{\rm gcd}\,(p, q)={\rm gcd}\,(p', q')={\rm gcd}\,(d, k)={\rm gcd}\,(d', k')=1.
$$
Take $m,\,m',\,n,\,n' \in \mathbb{Z}$ satisfying
$$
mp+nq=1,\qquad m'p'+n'q'=1.
$$
Then \eqref{eq:4} can be transformed into
$$
\left\{
\begin{array}{l}
d\theta+km=\dfrac{-q(d\varphi+kn)}{p},
\\[0.3 cm]
d'\theta'+k'm'=\dfrac{-q'(d'\varphi'+k'n')}{p'}.
\end{array}
\right.
$$

Firstly, we discuss the closures of the orbits of $e$ under $L_g$ and $L_{g'}$.
Similar to Case (i), for the left action $L_g$ and  $j=0, 1, \cdots, d-1$, define $F_j:\,S^1 \rightarrow \rm{U}(2)$ by
$$
F_j:\,z\mapsto\left(\begin{array}{cc}
z^p \cdot z^{-q}&0\\
0&z^q
\end{array}\right)\left(\begin{array}{cc}
\textrm{e}^{2\pi{\rm i} j\varphi}\textrm{e}^{2\pi{\rm i} j\theta}&0\\
0&\textrm{e}^{-2\pi{\rm i} j\theta}
\end{array}\right),\qquad \forall\, z \in S^1.
$$
One can see that every $F_j$ is an embedding map,
and consequently, every $H_j:\,S^1 \rightarrow F_j(S^1)$ is a homeomorphism.
Let $f$ be a rotation of $S^1$ with
$$
\rho(f)=\dfrac{d\varphi+kn}{p}+N_1,\qquad N_1 \in \mathbb{Z}.
$$
Then it is not difficult to verify that
$$
H_j \circ f=L_g^d|_{F_j(S^1)} \circ H_j,\qquad H_j(e')=e_j,
$$
and hence
$$
H_j(\overline{Orb_f(e')})=\overline{Orb_{L_g^d}(e_j)}=F_j(S^1),
$$
where
$$
e_j=\left(\begin{array}{cc}
\textrm{e}^{2\pi{\rm i} j\varphi}\textrm{e}^{2\pi{\rm i} j\theta}&0\\
0&\textrm{e}^{-2\pi{\rm i} j\theta}
\end{array}\right),\,\,\,j=0, 1, \cdots, d-1,
$$
and $e',\,e_0=e$ are the identity elements of $S^1$ and ${\rm U}(2)$, respectively.
Thus, $\overline{Orb_{L_g^d}(e_j)} \cong S^1$.
This fact indicates that every $\overline{Orb_{L_g^d}(e_j)}$ is a simple closed curve in ${\rm{U}}(2)$. Furthermore, we have
$$
\overline{Orb_{L_g}(e)}=\overline{Orb_{L_g^d}(e_0)}\cup\overline{Orb_{L_g^d}(e_1)}\cup
\cdots\cup\overline{Orb_{L_g^d}(e_{d-1})}.
$$
According to the definitions of $F_j$, where $j=0, 1, \cdots, d-1$, we see that the simple closed curves
$$
\overline{Orb_{L_g^d}(e_0)},\quad\overline{Orb_{L_g^d}(e_1)}, \cdots,
\overline{Orb_{L_g^d}(e_{d-1})}
$$
are mutually disjoint,
so, $\overline{Orb_{L_g}(e)}$ contains $d$ mutually disjoint simple closed curves in ${\rm{U}}(2)$.
Similarly, $\overline{Orb_{L_{g'}}(e)}$ contains $d'$ mutually disjoint simple closed curves in ${\rm{U}}(2)$.
Thus, $d=d'$.

Next, for the left action $L_{g'}$, define $F':\,S^1 \rightarrow {\rm U}(2)$ by
$$
F':\,z\mapsto\left(\begin{array}{cc}
z^{p'} \cdot z^{-q'}&0\\
0&z^{q'}
\end{array}\right),\qquad \forall\, z \in S^1,
$$
and take a rotation $f'$ on $S^1$ satisfying $$
\rho(f')=\dfrac{d\varphi'+k'n'}{p'}+N_2,\qquad N_2 \in \mathbb{Z}.
$$
Using the above observations, we obtain that $f'$ and $L_{g'}^d|_{F'(S^1)}$ are topologically conjugate,
and
$$
\overline{Orb_{L_{g'}^d}(e)}=F'(S^1).
$$
Since $f$ and $L_g^d|_{F_0(S^1)}$ are topologically conjugate,
$L_g^d|_{F_0(S^1)}$ and $L_{g'}^d|_{F'(S^1)}$ are topologically conjugate,
and then $f$ and $f'$ are topologically conjugate.
Therefore, it follows from Proposition \ref{prop:1} that
$$
\rho(f')=\pm\,\rho(f)\quad({\rm{mod}}\,\,\mathbb{Z}),
$$
i.e.,
$$
\dfrac{d\varphi'+k'n'}{p'}+N_2=\pm\,\left(\dfrac{d\varphi+kn}{p}+N_1\right)\quad({\rm{mod}}\,\,\mathbb{Z}).
$$
Then by the known conditions that $p,\,p > 0$ and $\varphi'=\pm\,\varphi\,\,\,({\rm{mod}}\,\,\mathbb{Z})$ is an irrational number, we obtain $p=p'$.
Hence we can simplify \eqref{eq:4} as follows.
\begin{align}\label{eq:5}
\left\{
\begin{array}{l}
k+pd\theta+qd\varphi=0,
\\
k'+pd\theta'+q'd\varphi'=0.
\end{array}
\right.
\end{align}

On the other hand, it is well-known that ${\rm U}(2)$ can be regarded as a $p$-fold covering space of itself,
and the covering map $\pi:\,{\rm U}(2) \rightarrow {\rm U}(2)$ is defined by
$$
\pi:\,(u, \lambda)\mapsto(u, \lambda^p),\qquad \forall\,\,(u, \lambda) \in {\rm U}(2) \cong {{\rm SU}}(2) \times S^1.
$$
According to the map lifting theorem, we know that the self-homeomorphisms of ${\rm U}(2)$ can always be lifted under the covering map $\pi$.
Then there exist $p$ lifts of $L_g$ denoted by ${\tilde L}_{g_j}$ with
$$
\rho({\tilde L}_{g_j})=\left(\arraycolsep=1pt\begin{array}{c}
\theta\\
\varphi/p+j/p
\end{array}\right),
$$
for $j=0,\,1,\,\cdots,\,p-1$.
Thus, \eqref{eq:5} can be written by
$$
\left\{
\begin{array}{l}
k+pd\theta+qpd\dfrac{\varphi}{p}=0,
\\[0.25 cm]
k'+pd\theta'+q'pd\dfrac{\varphi'}{p}=0.
\end{array}
\right.
$$
Fix one lift ${\tilde L}_{g_1}$ of $L_g$ and one lift $\tilde{h}$ of $h$ satisfying
$$
\rho({\tilde L}_{g_1})=\left(\arraycolsep=1pt\begin{array}{c}
\theta\\
\varphi/p
\end{array}\right),
\qquad\tilde{h}(e)=e.
$$
Lemma \ref{lem:5} implies that there exists some certain lift of $L_{g'}$ denoted by ${\tilde L}_{g_1'}$ such that
$$
\rho({\tilde L}_{g_1'})=\left(\arraycolsep=1pt\begin{array}{c}
\theta\\
\varphi''
\end{array}\right),
\qquad
\tilde{h} \circ {\tilde L}_{g_1}={\tilde L}_{g_1'} \circ \tilde{h}.
$$
Then the previous proof indicates that
$$
\varphi''=\pm\,\dfrac{\varphi}{p}\quad({\rm{mod}}\,\,\mathbb{Z}).
$$
In the next part, we discuss these two cases, respectively.

(a) Assume that $\varphi''=\dfrac{\varphi}{p}$.
Then it is easy to confirm $\varphi'=\varphi$,
so we have
$$
\left\{
\begin{array}{l}
k+pd\theta+qpd\dfrac{\varphi}{p}=0,
\\[0.25 cm]
k'+pd\theta'+q'pd\dfrac{\varphi}{p}=0,
\end{array}
\right.
$$
and
$$
\rho({\tilde L}_{g_1})=\left(\arraycolsep=1pt\begin{array}{c}
\theta\\
\varphi/p                                                                                       \end{array}\right),
\qquad \rho({\tilde L}_{g_1'})=\left(\arraycolsep=1pt\begin{array}{c}
\theta'\\
\varphi/p
\end{array}\right).
$$
Set
$$
G_0=\left\{\left(\begin{array}{cc}
z\omega&0\\
0&\bar{z}
\end{array}\right);\,z\omega^q=1,\,z,\,\omega \in \mathbb{C},\,|z|=|\omega|=1\right\},
$$
$$
G_0'=\left\{\left(\begin{array}{cc}
z\omega&0\\
0&\bar{z}
\end{array}\right);\,z\omega^{q'}=1,\,z,\,\omega \in \mathbb{C},\,|z|=|\omega|=1\right\}.
$$
Then it follows from Lemma \ref{lem:4} that $G_0$ and $G_0'$ are two subgroups of ${\rm U}(2)$ homeomorphic to $S^1$, and
$$
{\rm U}(2)/G_0\,\cong\,{\rm U}(2)/G_0'\,\cong\,L(1, -1)\,\cong S^3\,\cong\,{\rm SU}(2).
$$
In fact, through a simple verification, one can see that
$$
\overline{Orb_{{\tilde L}_{g_1}^{d_1}}(e)}=G_0,\qquad\overline{Orb_{{\tilde L}_{g_1'}^{d_1}}(e)}=G_0',
$$
where $d_1=\dfrac{pd}{{\rm gcd}\,(k, p)}=\dfrac{pd}{{\rm gcd}\,(k', p)}$,
so we have $\tilde{h}(G_0)=G_0'$.
Take left actions $L_1,\,L_1' \in {\mathfrak M}_{T_{{\rm SU}(2)}}$ with
$$
\rho(L_1)=-\dfrac{k}{pd}\quad({\rm{mod}}\,\,\mathbb{Z}),
\qquad \rho(L_1')=-\dfrac{k'}{pd}\quad({\rm{mod}}\,\,\mathbb{Z}).
$$
Then according to the proof of Lemma \ref{lem:4}, we know that the quotient maps
$$
\pi_0:\,{\rm U}(2) \rightarrow {\rm U}(2)/G_0 \cong {\rm SU}(2),
\qquad \pi_0':\,{\rm U}(2) \rightarrow {\rm U}(2)/G_0' \cong {\rm SU}(2)
$$
satisfy that for any $u=(z, \textrm{e}^{2\pi{\rm i}\alpha}, \textrm{e}^{2\pi{\rm i}\beta}) \in {\rm U}(2)$,
$$
\pi_0(z, \textrm{e}^{2\pi{\rm i}\alpha}, \textrm{e}^{2\pi{\rm i}\beta})=(z, \textrm{e}^{2\pi{\rm i}(\alpha+q\beta)}),
\quad \pi_0(z, \textrm{e}^{2\pi{\rm i}\alpha}, \textrm{e}^{2\pi{\rm i}\beta})=(z, \textrm{e}^{2\pi{\rm i}(\alpha+q'\beta)}).
$$
And then we can easily verify that
$$
\pi_0 \circ {\tilde L}_{g_1}(u)=L_1 \circ \pi_0(u),\qquad \pi_0' \circ {\tilde L}_{g_1'}(u)=L_1' \circ \pi_0'(u),
\qquad \forall\,u \in {\rm U}(2)
$$
i.e.,
$$
\pi_0 \circ {\tilde L}_{g_1}=L_1 \circ \pi_0,\qquad \pi_0' \circ {\tilde L}_{g_1'}=L_1' \circ \pi_0'.
$$
And Lemma \ref{lem:1} implies that there exists some homeomorphism $h':\,{\rm SU}(2) \rightarrow {\rm SU}(2)$ such that $$
\pi_0' \circ \tilde{h}=h' \circ \pi_0,
$$
and obviously, $h'(e'')=e''$, where $e''$ is the identity element of ${\rm SU}(2)$.
Thus, by the same way in the previous analysis, we obtain
$$
h' \circ L_1(u)=L_1' \circ h'(u),
\qquad \forall\, u \in {\rm SU}(2),
$$
that means $L_1$ and $L_1'$ are topologically conjugate.
Therefore, it follows from Theorem \ref{the:1} that
$$
\rho(L_1')=\pm\,\rho(L_1)\quad({\rm{mod}}\,\,\mathbb{Z}),
\quad\hbox{i.e.}\quad \dfrac{k'}{pd}=\pm\,\dfrac{k}{pd}\quad({\rm{mod}}\,\,\mathbb{Z}).
$$
If we fix another lift of $L_g$ denoted by ${\tilde L}_{g_2}$ satisfying $\rho({\tilde L}_{g_2})=\left(\arraycolsep=1pt\begin{array}{c}
\theta\\
\varphi/p+1/p
\end{array}\right)$,
according to Lemma \ref{lem:5}, there exists some certain lift ${\tilde L}_{g_2'}$ of $L_{g'}$ such that
$$
\tilde{h} \circ {\tilde L}_{g_2}={\tilde L}_{g_2'} \circ \tilde{h}.
$$
And together with the fact $\varphi'=\varphi$,
we get $\rho({\tilde L}_{g_2'})=\left(\arraycolsep=1pt\begin{array}{c}
\theta'\\
\varphi/p+1/p                                                                                     \end{array}\right)$.
Similar to $L_1,\,L_1'$, we can take two left actions $L_2,\,L_2' \in {\mathfrak M}_{T_{{\rm SU}(2)}}$ with
$$
\rho(L_2)=-\dfrac{k}{pd}+\dfrac{q}{p}\quad({\rm{mod}}\,\,\mathbb{Z}),
\qquad \rho(L_2')=-\dfrac{k'}{pd}+\dfrac{q'}{p}\quad({\rm{mod}}\,\,\mathbb{Z}),
$$
and then we can prove that
$$
\pi_0 \circ {\tilde L}_{g_2}=L_2 \circ \pi_0,\qquad \pi_0' \circ {\tilde L}_{g_2'}=L_2' \circ \pi_0',
$$
and $h'$ is also the topological conjugacy from $L_2$ to $L_2'$.
Thus, Theorem \ref{the:1} indicates that
$$
\rho(L_2')=\pm\,\rho(L_2)\quad({\rm{mod}}\,\,\mathbb{Z}),
\quad\hbox{i.e.}\quad
\dfrac{k'-qd}{pd}=\pm\,\dfrac{k-q'd}{pd}\quad({\rm{mod}}\,\,\mathbb{Z}).
$$
Assume that
$$
k = m_1\quad({\rm{mod}}\,\,pd),\qquad k' = m_1'\quad({\rm{mod}}\,\,pd),
$$
and
$$
k-qd = m_2\quad({\rm{mod}}\,\,pd),\qquad k'-q'd = m_2'\quad({\rm{mod}}\,\,pd).
$$
Then one can see that
$$
m_1'=\pm\,m_1\quad({\rm{mod}}\,\,pd),
\qquad m_2'=\pm\,m_2\quad({\rm{mod}}\,\,pd).
$$

On the other hand, similarly, take left actions $L_j,\,L_j' \in {\mathfrak M}_{T_{{\rm SU}(2)}}$ with
$$
\rho(L_j)=-\dfrac{k}{pd}+\dfrac{qj}{p}\quad({\rm{mod}}\,\,\mathbb{Z}),
\qquad \rho(L_j')=-\dfrac{k'}{pd}+\dfrac{q'j}{p}\quad({\rm{mod}}\,\,\mathbb{Z}),
$$
and then we can prove
$$
\pi_0 \circ {\tilde L}_{g_j}=L_j \circ \pi_0,\qquad \pi_0' \circ {\tilde L}_{g_j'}=L_j' \circ \pi_0',
$$
and $h'$ is a topological conjugacy from $L_j$ to $L_j'$.
Set $p={\rm gcd}\,(p, d)\cdot p_1$. Then $pd={\rm gcd}\,(p, d) \cdot p_1 \cdot d$.
And the known condition ${\rm gcd}\,(k, d)=1$ implies that ${\rm gcd}\,(-k+jd, d)=1$,
and hence ${\rm gcd}\,(-k+jd, (p, d))=1$.
Since ${\rm gcd}\,(p_1, d)=1$, there exist some $m,\,n \in \mathbb{Z}$ such that
$$
md-np_1=k+1,\quad\hbox{i.e.}\quad -k+md=np_1+1.
$$
And it is easy to know that ${\rm gcd}\,(-k+md, p_1)=1$.
Thus, we get
$$
{\rm gcd}\,(-k+md, pd)=1.
$$
Notice that ${\rm gcd}\,(p, q)=1$, and then there exists some $j_0 \in \{0, 1,\cdots, p-1\}$ such that $j_0q = m\,\,\,({\rm mod}\,\,\,p)$,
so we have
$$
{\rm gcd}\,(-k+qj_0d, pd)=1.
$$
It is known that $h'$ is a topological conjugacy from $L_{j_0}$ to $L_{j_0'}$ satisfying $h'(e'')=e''$,
where $e''$ is the identity element of ${{\rm SU}}(2)$.
Then together with the fact ${\rm gcd}\,(-k+qj_0d, pd)=1$, one can see that
$$
h'(\overline{Orb_{L_{j_0}}(e'')})=\overline{Orb_{L_{j_0'}}(e'')} \cong \mathbb{Z}_{pd},
$$
and then we have
$$
{\rm SU}(2)/\overline{Orb_{L_{j_0}}(e'')}\cong {\rm SU}(2)/\overline{Orb_{L_{j_0'}}(e'')}\cong  {\rm SU}(2)/\mathbb{Z}_{nd} \cong L(pd, -1).
$$
Hence Lemma \ref{lem:1} indicates that $h'$ naturally induces a homeomorphism $h'':\,L(pd, -1) \rightarrow L(pd, -1)$ such that
$$
\pi' \circ h'=h'' \circ \pi',
$$
where $\pi'$ is the universal covering map from ${{\rm SU}}(2)$ to $L(pd, -1)$.

Set
$$
\pi_1(L(pd, -1))=\{[\alpha_j];\,\,\,j=1,\,2,\,\cdots,\,pd\}\,\cong\,\mathbb{Z}_{pd},
$$
where $\alpha_j$ is a loop in $L(pd, -1)$ with the base point $\pi'(e'')$,
and $[\alpha_j]$ denotes the homotopy class of $\alpha_j$.
Let ${\tilde \alpha}_j$ be the unique lift of $\alpha_j$ satisfying ${\tilde \alpha}_j(0)=e''$.
Then ${\tilde \alpha}_j$ is a continuous curve connecting the elements $e''=\left(\begin{array}{cc}
1&0\\
0&1
\end{array}\right)$ and $\left(\begin{array}{cc}
{\textrm{e}}^{2\pi{\rm i} j/pd}&0\\
0&{\textrm{e}}^{-2\pi{\rm i} j/pd}
\end{array}\right)$.
By the known conditions,
we have
$$
h'(\tilde{\alpha}_{m_1})=\tilde{\alpha}_{m_1'},
$$
and then the commutative diagram in Section \ref{sec:3} indicates that
$$
h''(\alpha_{m_1})=\alpha_{m_1'},\quad\hbox{i.e.}\quad h_*''([\alpha_{m_1}])=[\alpha_{m_1'}],
$$
where $h_*''$ is the automorphism of $\pi_1(L(pd, -1))$ induced by $h''$.
When $m_1'=m_1$, we have $h_*''([\alpha_{m_1}])=[\alpha_{m_1}]$.
Then it follows from Lemma \ref{lem:2} that
$$
h_*''(l)=l,
\qquad \forall\, l \in \pi_1(L(pd, -1)),
$$
So we get
$$
k' = k\quad({\rm{mod}}\,\,pd),\quad\hbox{i.e.}\quad pd \mid k-k'.
$$
And then for $L_2$ and $L_2'$, one can see that
$$
h_*''([\alpha_{m_2}])=[\alpha_{m_2'}]=[\alpha_{m_2}].
$$
Hence $m_2'=m_2$, that means
$$
k'-q'd=k-qd\quad({\rm{mod}}\,\,pd).
$$
Sine $k' = k\,\,({\rm{mod}}\,\,pd)$, and then
$$
q' = q\quad({\rm{mod}}\,\,p),\quad\hbox{i.e.}\quad p \mid q-q'.
$$
Together with \eqref{eq:5}, we obtain
$$
\theta'=\theta+\dfrac{q-q'}{p}\varphi+\dfrac{k-k'}{pd},
$$
i.e.,
$$
\theta'=\theta+n\varphi+n',\qquad n, n' \in \mathbb{Z}.
$$
When $m_1'=pd-m_1$, we have $h_*''([\alpha_{m_1}])=-[\alpha_{m_1}]$.
Then it follows from Lemma \ref{lem:2} that
$$
h_*''(l)=-l,
\qquad \forall\, l \in \pi_1(L(pd, -1)),
$$
So we get
$$
k+k' = 0\quad({\rm{mod}}\,\,pd),\quad\hbox{i.e.}\quad pd \mid k+k'.
$$
And then for $L_2$ and $L_2'$, one can see that
$$
h''(\alpha_{m_2})=\alpha_{m_2'},\quad\hbox{i.e.}\quad h_*''([\alpha_{m_2}])=[\alpha_{m_2'}]=-[\alpha_{m_2}].
$$
Hence $m_2'=pd-m_2$, that means
$$
k'-q'd=-(k-qd)\quad({\rm{mod}}\,\,pd).
$$
Sine $k+k' = 0\,\,\,({\rm{mod}}\,\,\,pd)$, and then
$$
q+q' = 0\quad({\rm{mod}}\,\,p),\quad\hbox{i.e.}\quad p \mid q+q'.
$$
Combining with \eqref{eq:5}, we obtain
$$
\theta'=-\theta-\dfrac{q+q'}{p}\varphi-\dfrac{k+k'}{pd},
$$
i.e.,
$$
\theta'=-\theta+n\varphi+n',\qquad n, n' \in \mathbb{Z}.
$$

(b) Assume that $\varphi''=1-\dfrac{\varphi}{p}$.
Then it is easy make sure $\varphi'=1-\varphi$,
so we have
$$
\left\{
\begin{array}{l}
k+pd\theta+qpd\dfrac{\varphi}{p}=0,
\\[0.25 cm]
k'-(p-1)q'd+pd\theta'+q'pd\left(1-\dfrac{\varphi}{p}\right)=0,
\end{array}
\right.
$$
$$
\rho(\tilde{L}_{g_1})=\left(\arraycolsep=1pt\begin{array}{c}
\theta\\
\varphi/p                                                                                       \end{array}\right),
\qquad \rho(\tilde{L}_{g_1'})=\left(\arraycolsep=1pt\begin{array}{c}
\theta'\\
1-\varphi/p
\end{array}\right),
$$
and $h$ is a topological conjugacy from ${\tilde L}_{g_1}$ to ${\tilde L}_{g_1'}$ satisfying $\tilde{h}(e)=e$.
Similar to Case (a),
one can see that there exist some left actions $L_1,\,L_1' \in T_{{\rm SU}(2)}$ and a homeomorphism $h':\,{\rm SU}(2) \rightarrow {\rm SU}(2)$ with
$$
\rho(L_1)=-\dfrac{k}{pd}\,\,\,({\rm{mod}}\,\,\mathbb{Z}),
\quad \rho(L_1')=-\dfrac{k'-(p-1)q'd}{pd}\,\,\,({\rm{mod}}\,\,\mathbb{Z}),
\quad h'(e'')=e''
$$
such that
$$
\pi_0 \circ L_{g_1}=L_1 \circ \pi_0,\qquad \pi_0' \circ L_{g_1'}=L_1' \circ \pi_0',
\qquad \pi_0' \circ \tilde{h}=h' \circ \pi_0,
$$
where $\pi_0,\,\pi_0':\,{\rm U}(2) \rightarrow {\rm SU}(2)$ are the quotient maps like those in Case (a), $e''$ is the identity element of ${\rm SU}(2)$,
and then $h'$ is a topological conjugacy from $L_1$ to $L_1'$.
According to Theorem \ref{the:1}, we get
$$
\rho(L_1')=\pm\,\rho(L_1)\quad({\rm{mod}}\,\,\mathbb{Z}),
$$
i.e.,
$$
\dfrac{k'-(p-1)q'd}{pd}=\pm\,\dfrac{k}{pd}\quad({\rm{mod}}\,\,\mathbb{Z}).
$$
If we fix another lift ${\tilde L}_{g_2}$ of $L_g$ satisfying $\rho({\tilde L}_{g_2})=\left(\arraycolsep=1pt\begin{array}{c}
\theta\\
\varphi/p+1/p                                                                                    \end{array}\right)$, Lemma \ref{lem:5} indicates that there exists some lift ${\tilde L}_{g_2'}$ of $L_{g'}$ such that
$$
\tilde{h} \circ {\tilde L}_{g_2}={\tilde L}_{g_2'} \circ \tilde{h},
$$
and then we obtain $\rho({\tilde L}_{g_2'})=\left(\arraycolsep=1pt\begin{array}{c}
\theta'\\
(p-1)/p-\varphi/p                                                                                     \end{array}\right)$.
Hence by the same way, we can take two left actions $L_2,\,L_2' \in {\mathfrak M}_{T_{{\rm SU}(2)}}$ with
$$
\rho(L_2)=-\dfrac{k}{pd}+\dfrac{q}{p}\quad({\rm{mod}}\,\,\mathbb{Z}),
\qquad \rho(L_2')=-\dfrac{k'-(p-2)q'd}{pd}\quad({\rm{mod}}\,\,\mathbb{Z}),
$$
and then we can prove that
$$
\pi_0 \circ {\tilde L}_{g_2}=L_2 \circ \pi_0,\qquad \pi_0' \circ {\tilde L}_{g_2'}=L_2' \circ \pi_0'.
$$
and $h'$ is also a topological conjugacy from $L_2$ to $L_2'$.
Thus, it follows from Theorem \ref{the:1} that
$$
\rho(L_2')=\pm\,\rho(L_2)\quad({\rm{mod}}\quad\mathbb{Z}),
$$
i.e.,
$$
\dfrac{k'-(p-2)q'd}{pd}=\pm\,\dfrac{k-qd}{pd}\quad({\rm{mod}}\,\,\mathbb{Z}).
$$
Assume that
$$
k = m_1\quad({\rm{mod}}\,\,pd),\qquad k'-(p-1)q'd = m_1'\quad({\rm{mod}}\,\,pd),
$$
and
$$
k-qd = m_2\quad({\rm{mod}}\,\,pd),\qquad k'-(p-2)q'd = m_2'\quad({\rm{mod}}\,\,pd).
$$
Then one can see that
$$
m_1'=\pm\,m_1\quad({\rm{mod}}\,\,pd),
\qquad m_2'=\pm\,m_2\quad({\rm{mod}}\,\,pd).
$$
And then using the same way in Case (a), we have
$$
h''(\alpha_{m_1})=\alpha_{m_1'},\quad\hbox{i.e.}\quad h_*''([\alpha_{m_1}])=[\alpha_{m_1'}],
$$
where $h''$ is a self-homeomorphism of $L(pd, -1)$ induced by $h'$,
and $h_*''$ is the automorphism of $\pi_1(L(pd, -1))$ induced by $h''$.
When $m_1'=m_1$, we see $h_*''([\alpha_{m_1}])=[\alpha_{m_1}]$.
Then Lemma \ref{lem:2} implies that
$$
h_*''(l)=l,
\qquad \forall\, l \in \pi_1(L(pd, -1)),
$$
So we get
$$
k-k'+(p-1)q'd = 0\quad({\rm{mod}}\,\,pd),\quad\hbox{i.e.}\quad pd \mid k-k'-q'd.
$$
And then for $L_2$ and $L_2'$, one can see that
$$
h''(\alpha_{m_2})=\alpha_{m_2'},\quad\hbox{i.e.}\quad h_*''([\alpha_{m_2}])=[\alpha_{m_2'}]=[\alpha_{m_2}].
$$
Hence $m_2'=m_2$, that means
$$
k'-(p-2)q'd=k-qd\quad({\rm{mod}}\,\,pd).
$$
Since $k-k'+(p-1)q'd = 0\,\,\,({\rm{mod}}\,\,pd)$, and then
$$
q'+q = 0\quad({\rm{mod}}\,\,p),\quad\hbox{i.e.}\quad p \mid q+q'.
$$
Together with \eqref{eq:5}, we obtain
$$
\theta'=\theta+\dfrac{q+q'}{p}\varphi+\dfrac{k-k'-q'd}{pd},
$$
i.e.,
$$
\theta'=\theta+n\varphi+n',\qquad n, n' \in \mathbb{Z}.
$$
When $m_1'=pd-m_1$, we have $h_*''([\alpha_{m_1}])=-[\alpha_{m_1}]$.
Then Lemma \ref{lem:2} indicates that
$$
h_*''(l)=-l,
\qquad \forall\, l \in \pi_1(L(pd, -1)),
$$
So we get
$$
k+k'-(p-1)q'd = 0\quad({\rm{mod}}\,\,pd),\quad\hbox{i.e.}\quad pd \mid k+k'+q'd.
$$
And then for $L_2$ and $L_2'$, one can see that
$$
h''(\alpha_{m_2})=\alpha_{m_2'},\quad\hbox{i.e.}\quad h_*''([\alpha_{m_2}])=[\alpha_{m_2'}]=-[\alpha_{m_2}].
$$
Hence $m_2'=pd-m_2$, that means
$$
\dfrac{k'-(p-2)q'd}{pd}=-(\dfrac{k}{pd}-\dfrac{q}{p})\quad({\rm{mod}}\,\,\mathbb{Z}).
$$
Sine $k+k'-(p-1)q'd = 0\,\,\,({\rm{mod}}\,\,pd)$, and then
$$
q-q' = 0\quad({\rm{mod}}\,\,p),\quad\hbox{i.e.}\quad p \mid q-q'.
$$
Combining with \eqref{eq:5}, we obtain
$$
\theta'=-\theta-\dfrac{q-q'}{p}\varphi-\dfrac{k+k'+q'd}{pd},
$$
i.e.,
$$
\theta'=-\theta+n\varphi+n',\qquad n, n' \in \mathbb{Z}.
$$

To sum up, in this case, $\theta,\,\theta'$ and $\varphi$ satisfy
$$
\theta'=\pm\,\theta+n\varphi+n',\qquad n, n' \in \mathbb{Z}.
$$

(iii) Suppose that $\theta$ is an irrational number, and $\theta'=\dfrac{m}{d'}$ is a rational number with ${\rm gcd}\,(m, d')=1$.
Set
$$
k+pd\theta+qd\varphi=0,\qquad -m+d'\theta'+0\cdot\varphi'=0,
$$
where $d,\,k,\,p,\,q \in \mathbb{Z},\,d,\,p,\,q \neq 0,\,p > 0,\,{\rm gcd}\,(p, q)=1$,
and $h$ is a topological conjugacy from $L_g$ to $L_{g'}$ satisfying
$$
h(\overline{Orb_{L_g}(e)})=\overline{Orb_{L_{g'}}(e)}.
$$
Then the analysis in Case (ii) tells us that $d'=d$ and $p=1$.
So we have
$$
\theta+q\varphi+\dfrac{k}{d}=0,\quad \theta'-\dfrac{m}{d}=0,
\quad\hbox{i.e.}\quad \theta'=\theta+q\varphi+\dfrac{k+m}{d}.
$$
Assume that $k+m = d_1\,\,\,({\rm{mod}}\,\,\,d)$. Then we obtain
$$
\theta'-\dfrac{d_1}{d}=\theta+q\varphi+N,\qquad N \in \mathbb{Z}.
$$
Take another left action $L_{g''} \in \mathfrak{M}_{T_{{\rm U}(2)}}$ with
$$
\rho(L_{g''})=\left\{\begin{array}{ll}
{\left(\arraycolsep=1pt\begin{array}{c}
\theta'-\dfrac{d_1}{d}\\
\varphi'
\end{array}\right)},&\hbox{~$m \geq d_1$~},
\\[8 mm]
{\left(\arraycolsep=1pt\begin{array}{c}
\theta'-\dfrac{d_1}{d}+1\\
\varphi'
\end{array}\right)},&\hbox{~$m < d_1$~}.
\end{array}\right.
$$
It follows from the sufficiency of Theorem \ref{the:2} that $L_g$ and $L_{g''}$ are topologically conjugate, and then $L_{g'}$ and $L_{g''}$ are topologically conjugate.
Thus, according to Case (i), one can see that
$$
\theta'-\dfrac{d_1}{d}=\pm\,\theta'\quad({\rm{mod}}\,\,\,\mathbb{Z}).
$$
And together with the known condition $\theta'=\dfrac{m}{d}$, we get
$$
d_1 = 0\quad({\rm{mod}}\,\,\,d)
\quad\hbox{or}\quad d_1 = 2m\quad({\rm{mod}}\,\,\,d),
$$
so
$$
\theta'=\theta+q\varphi+n
\quad\hbox{or}\quad \theta'=-\theta-q\varphi+n',\qquad q,\,n,\,n' \in \mathbb{Z}.
$$

{\bf Case 4.} Assume that $\theta$ and $\varphi$ are rationally dependent, $\theta'$ and $\varphi'$ are also rationally dependent, and $\varphi=\dfrac{n}{m}$ is a rational number with ${\rm gcd}\,(m, n)=1$.
Then we know $\varphi'=\dfrac{n}{m}$ or $1-\dfrac{n}{m}$ according to the previous discussion.
There are three different cases.

(iv) Suppose that $\theta$ is a rational number, $\theta'$ is an irrational number,
and $h$ is a topological conjugacy from $L_g$ to $L_{g'}$ satisfying
$$
h(\overline{Orb_{L_g}(e)})=\overline{Orb_{L_{g'}}(e)}.
$$
It is easy to show that $\overline{Orb_{L_g}(e)}$ is a discrete set consisting of finite elements,
but by the similar proof to Case (i), we know that $\overline{Orb_{L_{g'}}(e)}$ contains $m$ mutually disjoint simple closed curves.
This fact is in contradiction to $h$ being a homeomorphism.
Thus, $L_g$ and $L_{g'}$ can not be topologically conjugate in this case.

(v) Suppose that $\theta$ and $\theta'$ are both irrational numbers.
According to the known conditions, one can see that $L_g^m$ and $L_{g'}^m$ are two left actions in $\mathfrak{M}_{T_{{\rm U}(2)}}$ with
$$
\rho(L_g^m)=\left(\arraycolsep=1pt\begin{array}{c}
m\theta+n_1\\
0
\end{array}\right),\quad \rho(L_{g'}^m)=\left(\arraycolsep=1pt\begin{array}{c}
m\theta'+n_2\\
0
\end{array}\right),\quad n_1,\,n_2 \in \mathbb{Z},
$$
and
$$
h \circ L_g^m=L_{g'}^m \circ h,\qquad h(e)=e,
$$
Define $f:\,S^1 \rightarrow {\rm U}(2)$ by
$$
f:\,z\mapsto\left(\begin{array}{cc}
z&0\\
0&\bar{z}
\end{array}\right),\qquad \forall\, z \in S^1.
$$
It is easy to see that $f$ is an embedding map, and then $h':\,S^1 \rightarrow f(S^1)$ is a homeomorphism.
Let $f_1,\,f_2$ be two rotations of $S^1$ with
$$
\rho(f_1)=m\theta+n_1,\qquad \rho(f_2)=m\theta'+n_2.
$$
Then it is not difficult to verify that
$$
h' \circ f_1=L_g^m|_{f(S^1)} \circ h',\qquad h' \circ f_2=L_{g'}^m|_{f(S^1)} \circ h',
$$
so $f_1$ and $L_g^m|_{f(S^1)}$ are topologically conjugate, $f_2$ and $L_{g'}^m|_{f(S^1)}$ are topologically conjugate,
and hence $f_1$ and $f_2$ are topologically conjugate.
Thus, Proposition \ref{prop:1} indicates that
$$
\rho(f_2)=\pm\,\rho(f_1)\quad({\rm{mod}}\,\,\mathbb{Z}),
\quad\hbox{i.e.}\quad m\theta'=\pm\,m\theta\quad({\rm{mod}}\,\,\mathbb{Z}),
$$
and then
$$
\theta'=\pm\,\theta+\dfrac{N}{m},\qquad N \in \mathbb{Z}.
$$
Assume that $N = m_1\,\,({\rm{mod}}\,\,m)$.
Then it follows from the fact ${\rm gcd}\,(m, n)$=1 that there exists some $n'\in\mathbb{Z}$ such that $$
nn' = m_1\quad({\rm{mod}}\,\,m),
$$
so we have
$$
\theta'=\pm\,\theta+\dfrac{N}{m}-\dfrac{nn'}{m}+\dfrac{nn'}{m}=\pm\,\theta+n'\varphi+n'',
\qquad n',\,n'' \in \mathbb{Z}.
$$

(vi) Suppose that $\theta$ and $\theta'$ are both rational numbers.
According to the known conditions, one can see that $L_g^m$ and $L_{g'}^m$ are two left actions in $\mathfrak{M}_{T_{{\rm U}(2)}}$ with
$$
\rho(L_g^m)=\left(\arraycolsep=1pt\begin{array}{c}
q/p\\
0
\end{array}\right),
\qquad\rho(L_{g'}^m)=\left(\arraycolsep=1pt\begin{array}{c}
q'/p'\\
0
\end{array}\right),
$$
where $p,\,p',\,q,\,q' \in \mathbb{Z},\,\,{\rm gcd}\,(p, q)={\rm gcd}\,(p', q')=1$,
and
$$
h \circ L_g^m=L_{g'}^m \circ h,\qquad h(e)=e,
$$
i.e.,
$$
h(\overline{Orb_{{L_g}^m}(e)})=\overline{Orb_{L_{g'}^m}(e)}.
$$
Thus, it is easy to know that $\overline{Orb_{{L_g}^m}(e)}$ contains $p$ points,
and $\overline{Orb_{L_{g'}^m}(e)}$ contains $p'$ points,
so $p=p'$, and then we have
$$
\overline{Orb_{{L_g}^m}(e)}=\overline{Orb_{L_{g'}^m}(e)}\,\cong\,\mathbb{Z}_p.
$$
Consequently, the forms of $\rho(L_g^m)$ and $\rho(L_{g'}^m)$ indicate that
$$
{\rm U}(2)/\overline{Orb_{{L_g}^m}(e)}={\rm U}(2)/\overline{Orb_{L_{g'}^m}(e)}\,\cong\,{\rm SU}(2)/\mathbb{Z}_p \times S^1\,\cong\,L(p, -1)\times S^1.
$$
By this way, ${\rm U}(2) \cong {\rm SU}(2) \times S^1$ can be regarded as a covering space of $L(p, -1) \times S^1$,
and the covering map $\pi_0$ satisfies that $\pi_0=\pi_0' \times {\rm id}_{S^1}$,
where $\pi_0'$ is the universal covering map from ${\rm SU}(2)$ to $L(p, -1)$,
and ${\rm id}_{S^1}$ is the identity map of $S^1$.
Then it follows from Lemma \ref{lem:1} that there exists some homeomorphism $h':\,L(p, -1)\times S^1 \rightarrow L(p, -1)\times S^1$ induced by $h$ such that
$$
\pi_0 \circ h=h' \circ \pi_0.
$$
Investigate the following diagram.
$$
\xymatrix{
{\,{{\rm SU}}(2)\,} \ar[d]_-{\pi_0'} \ar[r]^-{i}
& {\,{{\rm SU}}(2) \times S^1\,} \ar[d]_-{\pi_0} \ar[r]^-{h}
& {\,{{\rm SU}}(2) \times S^1\,} \ar[d]_-{\pi_0} \ar[r]^-{\pi}
& {\,{{\rm SU}}(2)\,} \ar[d]_-{\pi_0'}\\
{\,L(p, -1)\,\,} \ar[r]_-{i'}
& {\,L(p, -1) \times S^1\,} \ar[r]_-{h'}
& {\,L(p, -1) \times S^1\,} \ar[r]_-{\pi'}
& {\,L(p, -1)\,}.}
$$
In this diagram, $\pi_0,\,\pi_0'$ are the covering maps, $i,\,i'$ are the natural inclusion maps, $\pi,\,\pi'$ are the projections, and $i,\,i',\,\pi,\,\pi'$ satisfy
$$
\pi_0 \circ i=i' \circ \pi_0',
\qquad \pi_0' \circ \pi=\pi' \circ \pi_0.
$$
Then together with the fact $\pi_0 \circ h=h' \circ \pi_0$, one can see that the above diagram is commutative.

On the other hand, regard $L(p, -1) \times \mathbb{R}$ as another covering space of $L(p, -1) \times S^1$,
and the covering map $\pi_0''$ satisfies that $\pi_0''={\rm id}_{L(p, -1)} \times \pi_{\mathbb{R}}$,
where ${\rm id}_{L(p, -1)}$ is the identity map of $L(p, -1)$,
and $\pi_{\mathbb{R}}$ is the universal covering map from $\mathbb{R}$ to $S^1$.
Then according to the map lifting theorem, one can see that each self-homeomorphism of $L(p, -1) \times S^1$ can be lifted under the covering map $\pi_0''$.
Investigate the following diagram.
$$
\xymatrix{
{L(p, -1)} \ar[d]_-{{\rm id}_{L(p, -1)}} \ar[r]^-{i''}
& {L(p, -1) \times \mathbb{R}} \ar[d]_-{\pi_0''} \ar[r]^-{h''}
& {L(p, -1) \times \mathbb{R}} \ar[d]_-{\pi_0''} \ar[r]^-{\pi''}
& {L(p, -1)} \ar[d]_-{{\rm id}_{L(p, -1)}}\\
{L(p, -1)} \ar[r]_-{i'}
& {L(p, -1) \times S^1} \ar[r]_-{h'}
& {L(p, -1) \times S^1} \ar[r]_-{\pi'}
& {L(p, -1)}.}
$$
In the diagram, $h''$ is a lift of $h'$,
$i',\,i''$ are the natural inclusion maps, $\pi',\,\pi''$ are the projections, and $i',\,i'',\,\pi',\,\pi''$ satisfy
$$
i''(L(p, -1))=L(p, -1) \times \{0\},
\,\,\,\pi_0'' \circ i''=i' \circ {\rm id}_{L(p, -1)},
\,\,\,{\rm id}_{L(p, -1)} \circ \pi''=\pi' \circ \pi_0''.
$$
Then this diagram above is commutative.

Set $f=\pi'' \circ h'' \circ i''=\pi' \circ h' \circ i'$. It is easy to see that $\deg f=1$.
Then combining with the above commutative diagrams and the discussion in Section \ref{sec:3}, we obtain
$$
f_*([\alpha_q])=[\alpha_{q'}],
$$
where $f_*$ is the endomorphism of $\pi_1(L(p, -1))$ induced by $f$, and $[\alpha_q],\,[\alpha_q'] \in\pi_1(L(p, -1))$.

As well-known $L(p, -1)$ is a compact manifold, so obviously, there exists some $a\in\mathbb{R}$ such that
$$
h''(L(p, -1) \times [0, +\infty)) \subset L(p, -1) \times (a, +\infty).
$$
Set $W=L(p, -1) \times [a, +\infty)\backslash h''(L(p, -1) \times (0, +\infty))$, and
$$
M_0=h''(L(p, -1) \times \{0\}),\quad M_1=L(p, -1) \times \{a\} \cong L(p, -1).
$$
One can see that $(W,\,M_0,\,M_1)$ is a topological $h$-cobordism with $M_0$ and $M_1$.
Then Lemma \ref{lem:6} indicates that
$$
L(p, -1) \times [0, 1] \cong W.
$$
Let $H$ be a homeomorphism from $L(p, -1) \times [0, 1]$ to $W$ satisfying
$$
H(L(p, -1) \times \{0\})=M_0,\quad H(L(p, -1) \times \{1\})=M_1,
$$
and
$$
M_0'=L(p, -1) \times \{0\},\qquad M_1'=L(p, -1) \times \{1\},
$$
And take two embedding maps
$$
f_0:L(p, -1) \rightarrow L(p, -1) \times [0, 1],\quad f_1:\,L(p, -1) \rightarrow L(p, -1) \times [0, 1]
$$
satisfying
$$
f_0(L(p, -1))=M_0',\qquad f_1(L(p, -1))=M_1',
$$
respectively.
Then we know that $H \circ f_0,\,H \circ f_1:\,L(p, -1) \rightarrow W$ are also embedding maps.
Define a map $F:\,L(p, -1) \times [0, 1] \rightarrow L(p, -1) \times [0, 1]$ by
$$
F:\,(x, t) \mapsto (h_0^{-1} \circ h''|_{L(p, -1) \times \{0\}}f_0(x), t),
\qquad \forall\,(x, t) \in L(p, -1) \times [0, 1],
$$
where $h_0=H|_{M_0'}\circ f_0:\,L(p, -1) \rightarrow M_0$ is a homeomorphism.
Then for the map $H \circ F:\,L(p, -1) \times [0, 1] \rightarrow W$,
it is easy to see that $H \circ F|_{L(p, -1) \times \{t\}}$ is always an embedding map for any $t \in [0, 1]$, and $H \circ F$ satisfies
\begin{gather*}
H \circ F(x, 0)=h''|_{L(p, -1) \times \{0\}}\circ f_0(x),
\\H \circ F(x, 1)=h_1 \circ h_0^{-1} \circ h''|_{L(p, -1) \times \{0\}}\circ f_0(x),
\end{gather*}
where $h_1=H|_{M_1'} \circ f_1:\,L(p, -1) \rightarrow M_1$ is a homeomorphism.
Therefore, $H \circ F$ is an isotopy map from $h''|_{L(p, -1) \times \{0\}}\circ f_0$ to $h_1 \circ h_0^{-1} \circ h''|_{L(p, -1) \times \{0\}}\circ f_0$.
Then
$$
\pi'' \circ h''|_{L(p, -1) \times \{0\}}\circ f_0 \simeq \pi'' \circ h_1 \circ h_0^{-1} \circ h''|_{L(p, -1) \times \{0\}}\circ f_0,
$$
where $\pi'':\,L(p, -1) \times \mathbb{R} \rightarrow L(p, -1)$ is the projection in the second diagram above.
One can see that
$$
\pi'' \circ h''|_{L(p, -1) \times \{0\}}\circ f_0=\pi'' \circ h'' \circ i''=f,
$$
and $\pi \circ h_1 \circ h_0^{-1} \circ h|_{L(p, -1) \times \{0\}}\circ f_0$ is an orientation-preserving self-homeomorphism of $L(p, -1)$.
Then together with the fact $\deg f=1$ and Lemma \ref{lem:2},  we obtain
$$
f_*([a_q])=\pm [a_q]=[a_{q'}].
$$
Hence
$$
q'=q \quad\hbox{or}\quad q+q'=p,
$$
i.e.,
$$
\dfrac{q'}{p}=\pm\,\dfrac{q}{p}\quad(\rm{mod}\,\,\,\mathbb{Z}).
$$
And then if follows from the known conditions that
$$
m\theta'=\pm\,m\theta\quad(\rm{mod}\,\,\,\mathbb{Z}).
$$
Thus, we have $\theta'=\pm\,\theta+\dfrac{N}{m}$.
Assume that $N = m_1\,\,\,(\rm{mod}\,\,\,m)$.
Since $\varphi=\dfrac{n}{m}$ with ${\rm gcd}\,(m, n)=1$,
one can see that there exists some $n' \in \mathbb{Z}$ such that
$$
nn' = m_1\quad(\rm{mod}\,\,\,m).
$$
Therefore
$$
\theta'=\pm\,\theta+\dfrac{N}{m}=\pm\,\theta+\dfrac{N}{m}-\dfrac{nn'}{m}+\dfrac{nn'}{m}
=\pm\,\theta+n'\varphi+n'',
\quad n',\,n'' \in \mathbb{Z}.
$$

As a result , if $L_g$ and $L_{g}'$ are topologically conjugate satisfying
$$
\rho(L_g)=\left(\arraycolsep=1pt\begin{array}{c}
\theta\\
\varphi
\end{array}\right),\qquad \rho(L_{g'})=\left(\arraycolsep=1pt\begin{array}{c}
\theta'\\
\varphi'
\end{array}\right),
$$
we obtain
$$
\theta'=\pm\,\theta+n\varphi+n',
\qquad n,\,n' \in \mathbb{Z}.
$$
And together with the previous result
$$
\varphi'=\pm\,\varphi\quad(\rm{mod}\,\,\,\mathbb{Z}),
$$
we claim that the necessity of Theorem \ref{the:2} is true.
\end{proof}

\begin{remark}
It is noteworthy of the discussion in Case (vi). One can see that for any homeomorphism $h:\,L(p, -1)\times S^1\rightarrow L(p, -1)\times S^1$, the induced isomorphism $(\pi\circ h\circ i)_*$ maps each element in the fundamental group of $L(p,1)$ to itself or its inverse, where $i:L(p, -1)\rightarrow L(p, -1)\times S^1$ is the natural inclusion map and $\pi:L(p, -1)\times S^1\rightarrow L(p, -1)$ is the projection.
\end{remark}

Finally, we investigate the relationship among the topological conjugacy, algebraic conjugacy and smooth conjugacy of the left actions on ${\rm U}(2)$.

\begin{proposition}\label{prop:5}
There exist some left actions $L_g,\,L_{g'}$ on ${\rm U}(2)$ such that $L_g$ and $L_{g'}$ are topologically conjugate,
but not algebraically conjugate.
\end{proposition}

\begin{proof}
Take two left actions $L_g,\,L_{g'} \in \mathfrak{M}_{T_{{\rm U}(2)}}$ with
$$
\rho(L_g)=\left(\arraycolsep=1pt\begin{array}{c}
\theta\\
\varphi
\end{array}\right),\qquad \rho(L_{g'})=\left(\arraycolsep=1pt\begin{array}{c}
\theta'\\
\varphi'
\end{array}\right),
$$
and
$$
\left\{
\begin{array}{l}
\theta'=\theta+\varphi,
\\[0.25 cm]
\varphi'=\varphi,
\end{array}
\right.
$$
where $\theta,\,\varphi$ are two rational independent irrational numbers.
Then together with Theorem \ref{the:2}, one can see that $\theta',\,\varphi'$ are also two rational independent irrational numbers,
and $L_g$ and $L_{g'}$ are topologically conjugate.
Suppose that $L_g$ and $L_{g'}$ are algebraically conjugate, and the isomorphism $\Phi:\,{\rm U}(2) \rightarrow {\rm U}(2)$ is a algebraic conjugacy from $L_g$ to $L_{g'}$, i.e.,
$$
\Phi \circ L_g=L_{g'} \circ \Phi,\qquad \Phi(e)=e,
$$
where $e$ is the identity element of ${\rm U}(2)$.
It is well-known that for any element $u=\left(\arraycolsep=1pt\begin{array}{cc}
\lambda z_1&-\lambda \bar{z_2}\\
z_2&\bar{z_1}
\end{array}\right) \in {\rm U}(2)$,
where $z_1,\,z_2,\,\lambda \in \mathbb{C}$, $|z_1|^2+|z_2|^2=1$, $|\lambda|=1$,
the determinant of $u$ is $\lambda$.
This fact indicates that we can naturally define a group homomorphism ${\rm det}:\,{\rm U}(2) \rightarrow S^1$ by
$$
{\rm det}:\,u \mapsto \lambda,\qquad \forall\,u=\left(\arraycolsep=1pt\begin{array}{cc}
\lambda z_1&-\lambda \bar{z_2}\\
z_2&\bar{z_1}
\end{array}\right) \in {\rm U}(2).
$$
By the known conditions, we have $\Phi(\overline{Orb_{L_g}(e)})=\overline{Orb_{L_{g'}}(e)}$, and
$$
\overline{Orb_{L_g}(e)}=\overline{Orb_{L_{g'}}(e)}=T_{{\rm U}(2)}\,\cong\,T^2.
$$
Thus, $\Phi|_{T_{{\rm U}(2)}}$ is an automorphism of $T_{{\rm U}(2)}$ equivalent to an automorphism of $T^2$.
Together with Proposition \ref{prop:2}, the definition of the rotation vectors of the left actions in $\mathfrak{M}_{T_{{\rm U}(2)}}$ and the above relationship between $\rho(L_g)$ and $\rho(L_{g'})$,
one can see that ${\bm A}=\left(\arraycolsep=3pt\begin{array}{cc}
1&1\\
0&1
\end{array}\right)\in{\rm GL}_2(\mathbb{Z})$ is just the matrix form of the group homomorphism $\Phi_*$ on $\pi_1(T_{{\rm U}(2)})\,\cong\,\pi(T^2)$ induced by $\Phi$.
And then according to the well-known properties of the automorphisms of $\pi(T^2)$, it is easy to see that $\Phi$ satisfies that
$$
\Phi(t)=\left(\arraycolsep=3pt\begin{array}{cc}
\lambda^2z&0\\
0&\bar{\lambda} \bar{z}
\end{array}\right),\qquad \forall\,t=\left(\arraycolsep=3pt\begin{array}{cc}
\lambda z&0\\
0&\bar{z}
\end{array}\right)\in T_{{\rm U}(2)}.
$$
This fact implies that for any $t \in T_{{\rm U}(2)}$, ${\rm det}(\Phi(t))={\rm det}(t)$.
And it follows from Proposition \ref{prop:3} that for any $g \in {\rm U}(2)$, there exist some $t \in T_{{\rm U}(2)}$ and $s \in {\rm U}(2)$ such that $sg=ts$.
Then we have
$$
{\rm det}(s)\,{\rm det}(g)={\rm det}(sg)={\rm det}(ts)={\rm det}(t)\,{\rm det}(s),
$$
so ${\rm det}(g)={\rm det}(t)$.
On the other hand,
$$
\Phi(s)\,\Phi(g)=\Phi(sg)=\Phi(ts)=\Phi(t)\,\Phi(s),
$$
and similarly, one can see that ${\rm det}(\Phi(g))={\rm det}(\Phi(t))$.
Consequently, we obtain
$$
{\rm det}(\Phi(g))={\rm det}(\Phi(t))={\rm det}(t)={\rm det}(g),\qquad \forall\,g \in {\rm U}(2),
$$
that means $\Phi$ is an automorphism of ${\rm U}(2)$ preserving the determinants.
Regard ${\rm SU}(2)$ as a subgroup of ${\rm U}(2)$,
and it is easy to know that $\Phi|_{{\rm SU}(2)}$ is a automorphism of ${\rm SU}(2)$.
As well-known, each automorphism of ${\rm SU}(2)$ is an inner automorphism,
so $\Phi|_{{\rm SU}(2)}$ is an inner automorphism of ${\rm SU}(2)$ defined by
$$
\Phi|_{{\rm SU}(2)}:\,w \mapsto vwv^{-1},\qquad \forall\,w \in {\rm SU}(2),
$$
where $v \in {\rm SU}(2)$.
Since it is known that
$$
\Phi(t')=t',\qquad \forall\,t'=\left(\arraycolsep=3pt\begin{array}{cc}
z&0\\
0&\bar{z}
\end{array}\right) \in T_{{\rm SU}(2)},
$$
we have
$$
\left(\arraycolsep=2pt\begin{array}{cc}
z_0&-\bar{z_0'}\\
z_0'&\bar{z_0}
\end{array}\right)\left(\arraycolsep=2pt\begin{array}{cc}
z&0\\
0&\bar{z}
\end{array}\right)\left(\arraycolsep=2pt\begin{array}{cc}
\bar{z_0}&\bar{z_0'}\\
-z_0'&z_0
\end{array}\right)=vt'v^{-1}=t'=\left(\arraycolsep=2pt\begin{array}{cc}
z&0\\
0&\bar{z}
\end{array}\right),
$$
where $v=\left(\arraycolsep=2pt\begin{array}{cc}
z_0&-\bar{z_0'}\\
z_0'&\bar{z_0}
\end{array}\right)\in{\rm SU}(2),\,t'=\left(\arraycolsep=2pt\begin{array}{cc}
z&0\\
0&\bar{z}
\end{array}\right) \in T_{{\rm SU}(2)}$.
Then it follows from a simple calculation that $z_0'=0$.
This fact indicates that the element $v$ which induces the inner automorphism $\Phi|_{{\rm SU}(2)}$ must be a diagonal matrix like $\left(\arraycolsep=3pt\begin{array}{cc}
z_0&0\\
0&\bar{z_0}
\end{array}\right)$, where $z_0 \in \mathbb{C}$ and $|z_0|=1$.
Next, we denote the elements of $S^1$ by $2 \times 2$ matrixes like $\left(\begin{array}{cc}
\gamma&0\\
0&1
\end{array}\right)$, where $\gamma \in \mathbb{C}$ and $|\gamma|=1$,
and then for an element $u=\left(\arraycolsep=2pt\begin{array}{cc}
\lambda z_1&-\lambda\bar{z_2}\\
z_2&\bar{z_1}
\end{array}\right) \in {\rm U}(2)$, we have
$$
u=u_1u_2=\left(\arraycolsep=2pt\begin{array}{cc}
\lambda&0\\
0&1
\end{array}\right)\left(\arraycolsep=2pt\begin{array}{cc}
z_1&-\bar{z_2}\\
z_2&\bar{z_1}
\end{array}\right),
$$
where $u_1=\left(\arraycolsep=2pt\begin{array}{cc}
\lambda&0\\
0&1
\end{array}\right) \in S^1 \subseteq T_{{\rm U}(2)}$ and $\left(\arraycolsep=2pt\begin{array}{cc}
z_1&-\bar{z_2}\\
z_2&\bar{z_1}
\end{array}\right) \in {\rm SU}(2)$.
Thus, we get
$$
\Phi(u)=\Phi(u_1u_2)=\Phi(u_1)\Phi(u_2)=\left(\arraycolsep=2pt\begin{array}{cc}
\lambda^2&0\\
0&\bar{\lambda}
\end{array}\right)\left(\arraycolsep=2pt\begin{array}{cc}
z_0&0\\
0&\bar{z_0}
\end{array}\right)\left(\arraycolsep=2pt\begin{array}{cc}
z_1&-\bar{z_2}\\
z_2&\bar{z_1}
\end{array}\right)\left(\arraycolsep=2pt\begin{array}{cc}
\bar{z_0}&0\\
0&z_0
\end{array}\right).
$$
And for another element
$$
u'=u_1'u_2'=\left(\arraycolsep=2pt\begin{array}{cc}
\lambda'&0\\
0&1
\end{array}\right)\left(\arraycolsep=2pt\begin{array}{cc}
z_1'&-\bar{z_2'}\\
z_2'&\bar{z_1'}
\end{array}\right) \in {\rm SU}(2),
$$
where $u_1'=\left(\arraycolsep=2pt\begin{array}{cc}
\lambda'&0\\
0&1
\end{array}\right) \in T_{{\rm U}(2)}$ and $\left(\arraycolsep=2pt\begin{array}{cc}
z_1'&-\bar{z_2'}\\
z_2'&\bar{z_1'}
\end{array}\right)$,
similarly, we have
$$
\Phi(u')=\Phi(u_1'u_2')=\Phi(u_1')\Phi(u_2')=\left(\arraycolsep=2pt\begin{array}{cc}
\lambda'^2&0\\
0&\bar{\lambda'}
\end{array}\right)\left(\arraycolsep=2pt\begin{array}{cc}
z_0&0\\
0&\bar{z_0}
\end{array}\right)\left(\arraycolsep=2pt\begin{array}{cc}
z_1'&-\bar{z_2'}\\
z_2'&\bar{z_1'}
\end{array}\right)\left(\arraycolsep=2pt\begin{array}{cc}
\bar{z_0}&0\\
0&z_0
\end{array}\right).
$$
Assume that $z_1'=\lambda=1,\,z_2 \neq 0$, and $\lambda'$ satisfies that $\bar{\lambda'}\neq\lambda'$.
Then we have $z_2 \neq 0,\,z_2'=0$, and $\Phi|_{{\rm SU}(2)}$ is the identity map of ${\rm SU}(2)$.
Thus, it follows from a series of calculations that
$$
\Phi(uu')=\Phi(u_3u_3')=\Phi(u_3)\,\Phi(u_3')=\left(\arraycolsep=2pt\begin{array}{cc}
\lambda'^2z_1&-\lambda'\bar{z_2}z_0^2\\
z_2\bar{z_0}^2&\bar{\lambda'}\bar{z_1'}
\end{array}\right),
\qquad \Phi(u)\,\Phi(u')=\left(\arraycolsep=2pt\begin{array}{cc}
\lambda'^2z_1&-\bar{\lambda'}\bar{z_2}z_0^2\\
\lambda'^2z_2\bar{z_0}^2&\bar{\lambda'}\bar{z_1'}
\end{array}\right),
$$
where $u_3=\left(\arraycolsep=2pt\begin{array}{cc}
\lambda'&0\\
0&1
\end{array}\right) \in S^1$ and $u_3'=\left(\arraycolsep=2pt\begin{array}{cc}
z_1&-\bar{\lambda'}\bar{z_2}\\
\lambda'z_2&\bar{z_1'}
\end{array}\right) \in {\rm SU}(2)$.
Since $z_2,\,z_0\neq 0$ and $\bar{\lambda'}\neq\lambda'$,
one can see that
$$
-\lambda'\bar{z_2}z_0^2 \neq -\bar{\lambda'}\bar{z_2}z_0^2,\qquad z_2\bar{z_0}^2 \neq \lambda'^2z_2\bar{z_0}^2,
$$
that means
$$
\Phi(uu') \neq \Phi(u)\,\Phi(u').
$$
This fact is in contradiction to $\Phi$ being a isomorphism,
so the left actions $L_g$ and $L_{g'}$ can not be algebraic conjugate if $\rho(L_g)$ and $\rho(L_{g'})$ satisfy the conditions at the beginning of this proof.
\end{proof}

In fact, according to Proposition \ref{prop:5}, we see that the topologically conjugate classification of the left actions on ${\rm U}(2)$ is not equivalent to their algebraically conjugate classification.

\begin{proposition}\label{prop:6}
For any left actions $L_g,\,L_{g'}$ on ${\rm U}(2)$,
$L_g$ and $L_{g'}$ are topologically conjugate if and only if $L_g$ and $L_{g'}$ are smooth conjugate.
\end{proposition}

\begin{proof}
The sufficiency is obviously true.

If $L_g$ and $L_{g'}$ are topologically conjugate, then it follows from Proposition \ref{prop:3} that there exist some $t,\,t' \in T_{{\rm U}(2)}$ such that $L_g$ and $L_t$ are topologically conjugate, $L_{g'}$ and $L_{t'}$ are topologically conjugate,
and hence $L_t$ and $L_{t'}$ are topologically conjugate.
Assume that
$$
\rho(L_t)=\left(\arraycolsep=1pt\begin{array}{c}
\theta\\
\varphi
\end{array}\right),\qquad \rho(L_{t'})=\left(\arraycolsep=1pt\begin{array}{c}
\theta'\\
\varphi'
\end{array}\right).
$$
Then Theorem \ref{the:2} indicates that
$$
\left\{
\begin{array}{l}
\theta'=\pm\,\theta+n\varphi+n',\qquad n,\,n'\in\mathbb{Z},
\\[0.25 cm]
\varphi'=\pm\,\varphi\quad(\rm{mod}\,\,\,\mathbb{Z}).
\end{array}
\right.
$$
Therefore, we use the same way as the proof of the sufficiency of Theorem \ref{the:2} to construct the topological conjugacies from $L_t$ to $L_{t'}$.
It is easy to see that these topological conjugacies are all smooth homeomorphisms, so $L_t$ and $L_{t'}$ are smooth conjugate.
Since each left action on ${\rm U}(2)$ is a smooth self-homeomorphism,
it follows from Proposition \ref{prop:3} that $L_g$ and $L_t$ are smooth conjugate, $L_{g'}$ and $L_{t'}$ are smooth conjugate,
and then $L_g$ and $L_{g'}$ are smooth conjugate.
\end{proof}

Proposition \ref{prop:6} implies that the topologically conjugate classification of the left actions on ${\rm U}(2)$ is equivalent to their smooth conjugate classification.

\section{Topologically conjugate classifications of the left actions on the quotient groups of ${\rm SU}(2)$ and ${\rm U}(2)$}

In this section, we give the topologically conjugate classifications of the left actions on the quotient groups of ${\rm SU}(2)$ and ${\rm U}(2)$, and then study the relationship among their topological conjugacy, algebraic conjugacy and smooth conjugacy.

Firstly, set
$$
G=\{e_1,\,-e_1\},\qquad G'=\{(e_1, e_2),\,(-e_1, -e_2)\},
$$
where $e_1,\,e_2$ are the identity elements of ${\rm SU}(2)$ and $S^1$, respectively.
One can see that $G \cong \mathbb{Z}_2$ is the only non-trivial normal subgroup of ${\rm SU}(2)$, and
$$
G_1 \cong \mathbb{Z}_q,\quad G_2 \cong G \oplus \mathbb{Z}_p \cong \mathbb{Z}_2 \oplus \mathbb{Z}_p,
\quad G_3=G' \cong \mathbb{Z}_2
$$
are the non-trivial normal subgroups of ${\rm U}(2)$,
where $q \geq 2$, and $\mathbb{Z}_2$ is not a non-trivial direct summand of $\mathbb{Z}_q$.
Thus, there is only one non-trivial quotient group of  ${\rm SU}(2)$, i.e.,
$$
{\rm SU}(2)/\mathbb{Z}_2 \cong {\rm SO}(3),
$$
and there are three quotient groups of ${\rm U}(2)$ including
\begin{gather*}
H_1 \cong {{\rm U}}(2)/G_1 \cong {\rm SU}(2) \times S^1/\mathbb{Z}_q \cong {\rm U}(2),
\\
H_2 \cong {{\rm U}}(2)/G_2 \cong {{\rm SU}}(2)/\mathbb{Z}_2
\oplus S^1/\mathbb{Z}_p \cong {\rm SO}(3) \times S^1,
\\
H_3 \cong {{\rm U}}(2)/G_3 \cong {\rm Spin}^{\mathbb{C}}(3).
\end{gather*}
Notice that the case $H_1 \cong {\rm U}(2)$ has been discussed in Section 4, and then we divide this section into three parts to study the quotient groups ${\rm SO}(3)$, ${\rm SO}(3) \times S^1$ and ${\rm Spin}^{\mathbb{C}}(3)$, respectively.

\begin{remark}
Assume that $G={\rm SU}(2)$ or ${\rm U}(2)$, $G'$ is a subgroup of them,
and $\pi:\,G \rightarrow G'$ is the quotient map.
Naturally, $\pi$ is both a covering map and a group homomorphism.
Then we have
$$
\pi(g_1g_2)=\pi(g_1)\,\pi(g_2),\qquad \forall\,g_1,\,g_2 \in G.
$$
Thus, for any left action $L_{g_\pi}$ on $G'$,
there exits some left action $L_g$ on $G$ satisfying $\pi(g)=g_\pi$ such that
$$
\pi(gg')=\pi(g)\,\pi(g')=g_\pi\,\pi(g'), \qquad \forall\,g' \in G,
$$
i.e.,
$$
\pi \circ L_g(g')=L_{g_\pi} \circ \pi(g'), \qquad \forall\,g' \in G,
$$
that means $L_g$ is a lift of $L_{g_\pi}$ under the covering map $\pi$.
This fact indicates that the left actions on $G'$ can always be lifted under the covering map $\pi$.
\end{remark}

\subsection{Topologically conjugate classification of the left actions on ${\rm SO}(3)$}
\label{sec:5.1}

In this part, we define the rotation numbers of the left actions in the set
$$
\mathfrak{M}_{T_{{\rm SO}(3)}}=\{L_g:\,{\rm SO}(3)\rightarrow {\rm SO}(3);\,\, g \in T_{{\rm SO}(3)}\},
$$
and utilize the rotation numbers defined to give the topologically conjugate classification of the left actions in $\mathfrak{M}_{T_{{\rm SO}(3)}}$,
and then give the topologically conjugate
classification of all left actions on ${\rm SO}(3)$.
Furthermore, for the left actions on ${\rm SO}(3)$,
we study the relationship among their topological conjugacy, algebraic conjugacy and smooth conjugacy.

Regard $S^3$ as the unit sphere in quaternion space denoted by
$$
S^3=\{a+b{\rm i}+c{\rm j}+d{\rm k};\,\,a,\,b,\,c,\,d \in \mathbb{R},\,\,|a|^2+|b|^2+|c|^2+|d|^2=1\},
$$
Notice that $S^3 \cong {\rm SU}(2)$, and then define a homeomorphism $h:\,S^3 \rightarrow {\rm SU}(2)$ by
$$
h:\,(a+b{\rm i}+c{\rm j}+d{\rm k})\mapsto\left(\begin{array}{cc}
a+d{\rm i}&-b+c{\rm i}\\
b+c{\rm i}&a-d{\rm i}
\end{array}\right),
\quad \forall\, a+b{\rm i}+c{\rm j}+d{\rm k} \in S^3.
$$
For any point $p=(x, y, z) \in \mathbb{R}^3$, we associate a quaternion $x{\rm i}+y{\rm j}+d{\rm k}$ which is also called $p$.
Then every quaternion $r$ corresponds to a rotation $R_r$ of $\mathbb{R}^3$ defined by
$$
R_r:\,p\mapsto rpr^{-1},\qquad \forall\, p \in \mathbb{R}^3,
$$
where $r^{-1}=\dfrac{\bar{r}}{\|r\|^2}$.
If $r=a+b{\rm i}+c{\rm j}+d{\rm k} \in S^3$, then it follows from a simple calculation that $R_r$ can be denoted by a $3 \times 3$ matrix, i.e.,
$$
R_r=\left(\begin{array}{ccc}
a^2+b^2-c^2-d^2&2bc-2ad&2ac+2bd\\
2bc+2ad&a^2-b^2+c^2-d^2&2cd-2ab\\
2bd-2ac&2ab+2cd&a^2-b^2-c^2+d^2
\end{array}\right).
$$
It is easy to prove that $R_r \in {\rm SO}(3)$.
Therefore, define $f:\,S^3 \rightarrow {\rm SO}(3)$ by
$$
f:\,r\mapsto R_r,\qquad \forall\, r=a+b{\rm i}+c{\rm j}+d{\rm k} \in S^3,
$$
and then there exists some map $\pi:\,{\rm SU}(2) \rightarrow {\rm SO}(3)$ such that
$$
\pi \circ h=f.
$$
One can see that $\pi$ is a $2$-fold covering map from ${\rm SU}(2)$ to ${\rm SO}(3)$ defined by
$$
\pi:\,u_r \mapsto R_r,
\qquad \forall\, u_r=\left(\begin{array}{cc}
a+d{\rm i}&-b+c{\rm i}\\
b+c{\rm i}&a-d{\rm i}
\end{array}\right) \in {\rm SU}(2).
$$
Assume that $T_{{\rm SU}(2)}$ is the maximal torus of ${\rm SU}(2)$ chosen in Section \ref{sec:3}.
Then fix one maximal torus $T_{{\rm SO}(3)}$ of ${\rm SO}(3)$
such that $T_{{\rm SO}(3)}$ is just the image of $T_{{\rm SU}(2)}$ under the covering map $\pi$.
According to the definition of $\pi:\,{\rm SU}(2) \rightarrow {\rm SO}(3)$,
we see that every element $u$ of $T_{{\rm SO}(3)}$ can be denoted by
$$
u=\left(\begin{array}{ccc}
\cos2\pi\theta&-\sin2\pi\theta&0\\
\sin2\pi\theta&\cos2\pi\theta&0\\
0&0&1
\end{array}\right),\qquad\theta \in [0, 1).
$$
Obviously,
$$
T_{{\rm SO}(3)}\,\cong\,S^1.
$$
Define an isomorphism $\Phi:\,T_{{\rm SO}(3)} \rightarrow S^1$ by
$$
\Phi:\,u\mapsto\textrm{e}^{2\pi{\rm i}\theta},
\qquad \forall\, u=\left(\begin{array}{ccc}
\cos2\pi\theta&-\sin2\pi\theta&0\\
\sin2\pi\theta&\cos2\pi\theta&0\\
0&0&1
\end{array}\right) \in T_{{\rm SO}(3)}.
$$
Then for any $L_g \in \mathfrak{M}_{T_{{\rm SO}(3)}}$,
set
$$
f=\Phi \circ L_g|_{T_{{\rm SO}(3)}} \circ \Phi^{-1}.
$$
One can see that $f$ is a rotation of $S^1$ satisfying
$$
f(z)={\textrm{e}}^{2\pi{{\rm i}}\theta}z,\qquad\forall\, z \in S^1,
$$
where $\theta \in [0, 1)$.
This fact indicates that $L_g|_{T_{{\rm SO}(3)}}$ is topologically conjugate to some rotation $f$ of $S^1$.
Therefore, according to Definition \ref{def:3}, we define the rotation number of the left action $L_g$ under the representation $(T_{{\rm SO}(3)}, \Phi)$  by
$$
\rho(L_g)\triangleq\rho(f)=\theta,
\qquad \theta \in [0,1).
$$

\begin{theorem}\label{the:3}
For the left actions $L_g,\,L_{g'} \in \mathfrak{M}_{T_{{\rm SO}(3)}}$,
$L_g$ and $L_{g'}$ are topologically conjugate if and only if
$$
\rho(L_{g'})=\pm\,\rho(L_g)\quad(\rm{mod}\,\,\,\mathbb{Z}).
$$
\end{theorem}

\begin{proof}
First, we prove the sufficiency.

If $\rho(L_{g'})=\rho(L_g)$, we have $L_g=L_{g'}$.
Hence the identity map of ${\rm SO}(3)$ is a topological conjugacy from $L_g$ to $L_{g'}$.
If $\rho(L_{g'})=1-\rho(L_g)$, set
$$
g=\left(\begin{array}{ccc}
\cos2\pi\theta&-\sin2\pi\theta&0\\
\sin2\pi\theta&\cos2\pi\theta&0\\
0&0&1
\end{array}\right),
\quad g'=\left(\begin{array}{ccc}
\cos2\pi\theta&\sin2\pi\theta&0\\
-\sin2\pi\theta&\cos2\pi\theta&0\\
0&0&1
\end{array}\right),
$$
and define a homeomorphism $h:\,{\rm SO}(3) \rightarrow {\rm SO}(3)$ by
$$
h:\,u\mapsto\left(\begin{array}{ccc}
0&1&0\\
1&0&0\\
0&0&-1
\end{array}\right)u\left(\begin{array}{ccc}
0&1&0\\
1&0&0\\
0&0&-1
\end{array}\right),\qquad \forall\, u \in {\rm SO}(3).
$$
Then for any $u \in {\rm SO}(3)$, we obtain
$$
h \circ L_g(u)=\left(\begin{array}{ccc}
\sin2\pi\theta&\cos2\pi\theta&0\\
\cos2\pi\theta&-\sin2\pi\theta&0\\
0&0&-1
\end{array}\right)u\left(\begin{array}{ccc}
0&1&0\\
1&0&0\\
0&0&-1
\end{array}\right)=L_{g'} \circ h(u),
$$
so $L_g$ and $L_{g'}$ are topologically conjugate.

Next, we prove the necessity.

One can see that ${\rm SU}(2)$ is a 2-fold covering space of ${\rm SO}(3)$, then there exist two lifts of every left action on ${\rm SO}(3)$, and both of them are in the set $\mathfrak{M}_{T_{{\rm SU}(2)}}$.
Set
$$
g=\left(\begin{array}{ccc}
\cos2\pi\theta&-\sin2\pi\theta&0\\
\sin2\pi\theta&\cos2\pi\theta&0\\
0&0&1
\end{array}\right) \in T_{{\rm SO}(3)}.
$$
Then $\rho(L_g)=\theta$, and it follows from the definition of the covering map $\pi$ that the two lifts of $L_g$ denoted by ${\tilde L}_g$ and ${\tilde L}_g'$ satisfy
$$
\rho({\tilde L}_g)=\dfrac{\theta}{2},\qquad \rho({\tilde L}_g')=\dfrac{\theta}{2}+\dfrac{1}{2}.
$$
Assume that $L_g,\,L_{g'} \in \mathfrak{M}_{T_{{\rm SO}(3)}}$ satisfy
$$
\rho(L_g)=\theta,\qquad \rho(L_{g'})=\theta',
$$
and the homeomorphism $h:\,{\rm SO}(3) \rightarrow {\rm SO}(3)$ is a  topological conjugacy from $L_g$ to $L_{g'}$.
As well-known, $\pi_1({\rm SU}(2))$ is trivial, and then $\pi$ is the universal covering map from ${\rm SU}(2)$ to ${\rm SO}(3)$.
Hence according to the map lifting theorem, we know that the self-homeomorphisms of ${\rm SO}(3)$ can always be lifted under the universal covering map $\pi$.
Then Lemma \ref{lem:5} implies that if we fix a lift of $L_g$ denoted by ${\tilde L}_g$ with $\rho({\tilde L}_g)=\dfrac{\theta}{2}$ and a lift of $h$,
there exists some certain lift of $L_{g'}$ denoted by ${\tilde L}_{g'}$ satisfying
$$
\rho({\tilde L}_{g'})=\dfrac{\theta'}{2}\quad\hbox{or}\quad\dfrac{\theta'}{2}+\dfrac{1}{2},
$$
such that ${\tilde L}_g$ and ${\tilde L}_{g'}$ are topologically conjugate.
And then it follows from Theorem \ref{the:1} that
$$
\dfrac{\theta'}{2}=\dfrac{\theta}{2}
\quad\hbox{or}\quad\dfrac{\theta'}{2}+\dfrac{1}{2}=1-\dfrac{\theta}{2}.
$$
Therefore,
$$
\theta'=\pm\,\theta\quad({\rm{mod}}\,\,\,\mathbb{Z}),
\quad\hbox{i.e.}\quad\rho(L_{g'})=\pm\,\rho(L_g)\quad({\rm{mod}}\,\,\,\mathbb{Z}).
$$
\end{proof}

Finally, together with the proofs of Corollary \ref{prop:4} and Theorem \ref{the:3},
we can give the relationship among the topological conjugacy, algebraic conjugacy and smooth conjugacy of the left actions on ${\rm SO}(3)$ by the following corollary.

\begin{corollary}\label{prop:7}
Suppose that $L_g,\,L_{g'}$ are the left actions on ${\rm SO}(3)$.
Then the following conditions are equivalent.

(a) $L_g$ and $L_{g'}$ are topologically conjugate;

(b) $L_g$ and $L_{g'}$ are algebraically conjugate;

(c) $L_g$ and $L_{g'}$ are smooth conjugate.
\end{corollary}

\subsection{Topologically conjugate Classification of the left actions on ${\rm SO}(3) \times S^1$}

In this part, we define the rotation vectors of the left actions in the set
$$
\mathfrak{M}_{T_{{\rm SO}(3) \times S^1}}=\{L_g:\,{\rm SO}(3) \times S^1\rightarrow {\rm SO}(3) \times S^1;\,\, g \in T_{{\rm SO}(3) \times S^1}\},
$$
and utilize the rotation vectors defined to give the topologically conjugate classification of the left actions in $\mathfrak{M}_{T_{{\rm SO}(3) \times S^1}}$,
and then give the topologically conjugate
classification of all left actions on ${\rm SO}(3) \times S^1$.
Furthermore, for the left actions on ${\rm SO}(3) \times S^1$,
we study the relationship among their topological conjugacy, algebraic conjugacy and smooth conjugacy.

Assume that
$$
H_2={{\rm U}}(2)/G_2' \cong {{\rm SU}}(2)/\mathbb{Z}_2 \times S^1/\mathbb{Z}_2\cong{{\rm SO}}(3)\times S^1,
$$
where
$$
G_2'=\{(e_1, e_2),\,\,\,(e_1, -e_2),\,\,\,(-e_1, e_2),\,\,\,(-e_1, -e_2)\} \cong \mathbb{Z}_2 \oplus \mathbb{Z}_2,
$$
and $e_1,\,e_2$ are the identity elements of ${{\rm SU}}(2)$ and $S^1$, respectively.
From this view, we see that ${\rm U}(2) \cong {\rm SU}(2) \times S^1$ is a 4-fold covering space of ${\rm SO}(3) \times S^1$, and then there exist four lifts of every left action on ${\rm SO}(3) \times S^1$, and all of them are in the set $\mathfrak{M}_{T_{{\rm U}(2)}}$.
For any self-homeomorphism $h$ of ${\rm SO}(3) \times S^1$, it is easy to see that
$$
\pi_*(h_*(\pi_1({\rm U}(2))))=\pi_*(\pi_1({\rm U}(2))),
$$
where $\pi$ is the $4$-fold covering map from ${\rm U}(2)$ to ${\rm SO}(3) \times S^1$,
and $h_*,\,\pi_*$ are the group homomorphisms on fundamental group induced by $h$ and $\pi$, respectively.
Then according to the map lifting theorem, we know that $h$ can be lifted under $\pi$.

Let $T_{{\rm U}(2)}$ be the maximal torus of ${\rm U}(2)$ chosen in Section \ref{sec:4}.
Then fix one maximal torus $T_{{\rm SO}(3) \times S^1}$ of ${\rm SO}(3) \times S^1$
such that $T_{{\rm SO}(3) \times S^1}$ is just the image of $T_{{\rm U}(2)}$ under the covering map $\pi$.
Notice that $T_{{\rm SO}(3) \times S^1} \cong T_{{\rm SO}(3)} \times S^1$,
and then according to the properties of the covering map $\pi:\,{\rm U}(2)
\rightarrow {\rm SO}(3) \times S^1$,
one can see that every element of $T_{{\rm SO}(3) \times S^1}$ can be denoted by $(u,\,\lambda)$,
where
$$
u=\left(\begin{array}{ccc}
\cos2\pi\theta&-\sin2\pi\theta&0\\
\sin2\pi\theta&\cos2\pi\theta&0\\
0&0&1
\end{array}\right) \in T_{{\rm SO}(3)},
\,\,\, \lambda=\textrm{e}^{2\pi{\rm i}\varphi} \in S^1,
\quad \theta,\,\varphi \in [0, 1).
$$
Obviously,
$$
T_{{\rm SO}(3) \times S^1} \cong T_{{\rm SO}(3)} \times S^1 \cong T^2.
$$
Define an isomorphism $\Phi:\,T_{{\rm SO}(3) \times S^1} \rightarrow T^2$ by
$$
\Phi:\,(u,\,\lambda)\mapsto\left(\begin{array}{cc}
\textrm{e}^{2\pi{\rm i}\theta}&0\\
0&\textrm{e}^{2\pi{\rm i}\varphi}
\end{array}\right),
\quad \forall\, (u,\,\lambda) \in T_{{\rm SO}(3) \times S^1}.
$$
Then for any $L_g \in \mathfrak{M}_{T_{{\rm SO}(3) \times S^1}}$,
set
$$
f=\Phi \circ L_g|_{T_{{\rm SO}(3) \times S^1}} \circ \Phi^{-1}.
$$
One can see that $f$ is a rotation of $T^2$ satisfying
$$
f(t)=\left(\begin{array}{cc}
{\textrm{e}}^{2\pi{\rm i}\theta}&0\\
0&{\textrm{e}}^{2\pi{\rm i}\varphi}
\end{array}\right)t,
\qquad \forall\, t \in T^2,
$$
where $\theta,\,\varphi \in [0, 1)$.
This fact indicates that $L_g|_{T_{{\rm SO}(3) \times S^1}}$ is topologically conjugate to some rotation $f$ of $T^2$.
Thus,  according to Definition \ref{def:3}, we define the rotation vector of the left action $L_g$ under the representation $(T_{{\rm SO}(3) \times S^1}, \Phi)$ by
$$
\rho(L_g)\triangleq\rho(f)=\left(\begin{array}{c}
\theta\\
\varphi
\end{array}\right),
\qquad \theta,\,\varphi \in [0,1).
$$

\begin{theorem}\label{the:4}
For the left actions $L_g,\,L_{g'} \in \mathfrak{M}_{T_{{\rm SO}(3) \times S^1}}$ with
$$
\rho(L_g)=\left(\arraycolsep=1pt\begin{array}{c}
\theta\\
\varphi
\end{array}\right),\qquad \rho(L_{g'})=\left(\arraycolsep=1pt\begin{array}{c}
\theta'\\
\varphi'
\end{array}\right),
$$
$L_g$ and $L_{g'}$ are topologically conjugate if and only if
$$
\left\{
\begin{array}{l}
\theta'=\pm\,\theta+n\varphi+n',\qquad n,\,n'\in\mathbb{Z},
\\
\varphi'=\pm\,\varphi\quad(\rm{mod}\,\,\,\mathbb{Z}).
\end{array}
\right.
$$
\end{theorem}

\begin{proof}
First, we prove the sufficiency.

It is not difficult to verify that for any $L_g \in \mathfrak{M}_{T_{{\rm SO}(3) \times S^1}}$ with $\rho(L_g)=\left(\arraycolsep=1pt\begin{array}{c}
\theta\\
\varphi
\end{array}\right)$, the lifts of $L_g$ denoted by $\rho({\tilde L}_{g_i})$, where $i=1,\,2,\,3,\,4$ satisfy
$$
\rho({\tilde L}_{g_1})=\left(\arraycolsep=1pt\begin{array}{c}
\theta/2\\
\varphi/2
\end{array}\right),
\qquad \rho({\tilde L}_{g_2})=\left(\arraycolsep=1pt\begin{array}{c}
\theta/2+1/2\\
\varphi/2
\end{array}\right),
$$
$$
\rho({\tilde L}_{g_3})=\left(\arraycolsep=1pt\begin{array}{c}
\theta/2\\
\varphi/2+1/2
\end{array}\right),
\quad \rho({\tilde L}_{g_4})=\left(\arraycolsep=1pt\begin{array}{c}
\theta/2+1/2\\
\varphi/2+1/2
\end{array}\right).
$$
Since
$$
\rho(L_g)=\left(\arraycolsep=1pt\begin{array}{c}
\theta\\
\varphi
\end{array}\right),\qquad \rho(L_{g'})=\left(\arraycolsep=1pt\begin{array}{c}
\theta'\\
\varphi'
\end{array}\right),
$$
and
$$
\left\{
\begin{array}{l}
\theta'=\pm\,\theta+n\varphi+n',\qquad n,\,n'\in\mathbb{Z},
\\
\varphi'=\pm\,\varphi\quad(\rm{mod}\,\,\,\mathbb{Z}),
\end{array}
\right.
$$
we have
$$
\left\{
\begin{array}{l}
\theta'/2=\pm\,\theta/2+n\varphi/2+n'/2,\qquad n,\,n'\in\mathbb{Z},
\\
\varphi'/2=\varphi/2\,\,\,\hbox{or}\,\,\, 1/2-\varphi/2.
\end{array}
\right.
$$
Then it follows from Theorem \ref{the:2} that if $n'$ is an even number, $\tilde{L}_{g_1}$ and $\tilde{L}_{g_1'}$ are topologically conjugate, or  $\tilde{L}_{g_1}$ and $\tilde{L}_{g_3'}$ are topologically conjugate,
and hence
$$
\left\{
\begin{array}{l}
\theta'/2=\pm\,\theta/2+n\varphi/2+N,\qquad n,\,N\in\mathbb{Z},
\\[0.25 cm]
\varphi'/2=\varphi/2\quad\hbox{or}\quad\varphi'/2+1/2=1-\varphi/2.
\end{array}
\right.
$$
According to the above relationship between $\rho(\tilde{L}_g)$ and $\rho(\tilde{L}_{g'})$,
we take
$$
\tilde{h}=h_1,\,\,\,h_2,\,\,\,h_3,\,\,\,h_4
$$
defined in the proof of the sufficient of Theorem \ref{the:2} corresponding to the four different cases, respectively.
And through a simple verification, we have
$$
h_i(G_2')=G_2',\qquad i=1,\,2,\,3,\,4,
$$
that means
$$
\tilde{h}(G'_2)=G'_2.
$$
Then Lemma \ref{lem:1} implies that there exists some self-homeomorphism $h$ of ${\rm SO}(3) \times S^1$ induced by $\tilde{h}$ such that
$$
\pi \circ \tilde{h}=h \circ \pi,
$$
where $\pi$ is the $4$-fold covering map from ${\rm U}(2)$ to ${\rm SO}(3) \times S^1$.
Since $\tilde{h}$, ${\tilde L}_g$, and ${\tilde L}_{g'}$ are the lifts of $h$, $L_g$ and $L_{g'}$, respectively, satisfying
$$
\tilde{h} \circ {\tilde L}_g={\tilde L}_{g'} \circ \tilde{h},
$$
it follows from Lemma \ref{lem:5} that $L_g$ and $L_{g'}$ are topologically conjugate.
If $n'$ is an odd number, then $\tilde{L}_{g_1}$ and $\tilde{L}_{g_2'}$ are topologically conjugate,
or  $\tilde{L}_{g_1}$ and $\tilde{L}_{g_4'}$ are topologically conjugate,
and hence
$$
\left\{
\begin{array}{l}
\theta'/2+1/2=\pm\,\theta/2+n\varphi/2+N',\qquad n,\,N'\in\mathbb{Z},
\\[0.25 cm]
\varphi'/2=\varphi/2\quad\hbox{or}\quad\varphi'/2+1/2=1-\varphi/2.
\end{array}
\right.
$$
Using the above observations, one can see that $L_g$ and $L_{g'}$ are topologically conjugate.

Next, we prove the necessity.

Let $L_g,\,L_{g'} \in \mathfrak{M}_{T_{{\rm SO}(3) \times S^1}}$ satisfy
$$
\rho(L_g)=\left(\arraycolsep=1pt\begin{array}{c}
\theta\\
\varphi
\end{array}\right),\qquad \rho(L_{g'})=\left(\arraycolsep=1pt\begin{array}{c}
\theta'\\
\varphi'
\end{array}\right),
$$
and $h$ is a topological conjugacy from $L_g$ to $L_{g'}$.
By Lemma \ref{lem:5}, if we fix one lift of $L_g$ denoted by ${\tilde L}_{g_1}$ with $\rho({\tilde L}_{g_1})=\left(\arraycolsep=1pt\begin{array}{c}
\theta/2\\
\varphi/2
\end{array}\right)$ and one lift of $h$ denoted by $\tilde{h}$,
there exists some lift of $L_{g'}$ denoted by ${\tilde L}_{g'}$ such that
$$
\tilde{h} \circ {\tilde L}_g={\tilde L}_{g'} \circ \tilde{h},
$$
and
$$
\rho({\tilde L}_{g'})=\left(\arraycolsep=1pt\begin{array}{c}
\theta'/2\\
\varphi'/2
\end{array}\right)
\,\hbox{or}\,
\left(\arraycolsep=1pt\begin{array}{c}
\theta'/2+1/2\\
\varphi'/2
\end{array}\right)
\,\hbox{or}\,
\left(\arraycolsep=1pt\begin{array}{c}
\theta'/2\\
\varphi'/2+1/2
\end{array}\right)
\,\hbox{or}\,
\left(\arraycolsep=1pt\begin{array}{c}
\theta'/2+1/2\\
\varphi'/2+1/2
\end{array}\right).
$$
According to Theorem \ref{the:2}, if $\rho({\tilde L}_{g'})=\left(\arraycolsep=1pt\begin{array}{c}
\theta'/2\\
\varphi'/2
\end{array}\right)$ or $\left(\arraycolsep=1pt\begin{array}{c}
\theta'/2\\
\varphi'/2+1/2
\end{array}\right)$, we have
$$
\left\{
\begin{array}{l}
\theta'/2=\pm\,\theta/2+n\varphi/2+N,\quad n,\,N\in\mathbb{Z},
\\
\varphi'/2=\varphi/2\,\,\,\hbox{or}\,\,\, 1/2-\varphi/2,
\end{array}
\right.
$$
and hence
$$
\left\{
\begin{array}{l}
\theta'=\pm\,\theta+n\varphi+2N,\quad n,\,N\in\mathbb{Z},
\\
\varphi'=\pm\,\varphi\quad(\rm{mod}\,\,\,\mathbb{Z}).
\end{array}
\right.
$$
If $\rho({\tilde L}_{g'})=\left(\arraycolsep=1pt\begin{array}{c}
\theta'/2+1/2\\
\varphi'/2
\end{array}\right)$ or $\left(\arraycolsep=1pt\begin{array}{c}
\theta'/2+1/2\\
\varphi'/2+1/2
\end{array}\right)$, we get
$$
\left\{
\begin{array}{l}
\theta'/2+1/2=\pm\,\theta/2+n\varphi/2+N,\quad n,\,N\in\mathbb{Z},
\\
\varphi'/2=\varphi/2\,\,\,\hbox{or}\,\,\, 1/2-\varphi/2,
\end{array}
\right.
$$
and then
$$
\left\{
\begin{array}{l}
\theta'=\pm\,\theta+n\varphi+2N-1,\quad n,\,N\in\mathbb{Z},
\\
\varphi'=\pm\,\varphi\quad(\rm{mod}\,\,\,\mathbb{Z}).
\end{array}
\right.
$$
Therefore, if $L_g$ and $L_{g'}$ are topologically conjugate, one can see that
$$
\left\{
\begin{array}{l}
\theta'=\pm\,\theta+n\varphi+n',\qquad n,\,n'\in\mathbb{Z},
\\
\varphi'=\pm\,\varphi\quad(\rm{mod}\,\,\,\mathbb{Z}).
\end{array}
\right.
$$
\end{proof}

Finally, we investigate the relationship among the topologically conjugay, the algebraic conjugacy and the smooth conjugacy of the left actions on ${\rm SO}(3) \times S^1$.

\begin{proposition}\label{prop:8}
There exist some left actions $L_g,\,L_{g'}$ on ${\rm SO}(3) \times S^1$ such that $L_g$ and $L_{g'}$ are topologically conjugate,
but not algebraically conjugate.
\end{proposition}

\begin{proof}
Take two left actions $L_g,\,L_{g'} \in \mathfrak{M}_{T_{{\rm SO}(3) \times S^1}}$ with
$$
\rho(L_g)=\left(\arraycolsep=1pt\begin{array}{c}
\theta\\
\varphi
\end{array}\right),\qquad \rho(L_{g'})=\left(\arraycolsep=1pt\begin{array}{c}
\theta'\\
\varphi'
\end{array}\right),
$$
and
$$
\left\{
\begin{array}{l}
\theta'=\theta+\varphi,
\\[0.25 cm]
\varphi'=\varphi,
\end{array}
\right.
$$
where $\theta,\,\varphi$ are two rational independent irrational numbers.
Then together with Theorem \ref{the:4}, one can see that $\theta',\,\varphi'$ are also two rational independent irrational numbers,
and $L_g$ and $L_{g'}$ are topologically conjugate.
Suppose that $L_g$ and $L_{g'}$ are algebraically conjugate, and the isomorphism $\Phi:\,{\rm SO}(3) \times S^1 \rightarrow {\rm SO}(3) \times S^1$ is a algebraic conjugacy from $L_g$ to $L_{g'}$, i.e.,
$$
\Phi \circ L_g=L_{g'} \circ \Phi,\qquad \Phi(e)=e,
$$
where $e$ is the identity element of ${\rm SO}(3) \times S^1$.
Let $*$ be the group operation of ${\rm SO}(3) \times S^1$,
and the elements of $S^1$ be denoted by $2 \times 2$ matrixes like $\left(\begin{array}{cc}
\lambda&0\\
0&1
\end{array}\right)$, where $\lambda \in \mathbb{C}$ and $|\lambda|=1$.
In the proof of Proposition \ref{prop:5}, we set
$$
u=u_1u_2,\qquad \forall\,u \in {\rm U}(2),
$$
where $u_1 \in S^1 \subseteq T_{{\rm U}(2)}$ and $u_2 \in {\rm SU}(2)$.
Naturally, we can also denote the element $u=u_1u_2$ of ${\rm U}(2) \cong {\rm SU}(2) \times S^1$ by $(u_2, u_1)$,
and then $u_1u_2$ and $(u_2, u_1)$ are obviously the same element of ${\rm U}(2)$.
We know that $4$-fold covering map $\pi:\,{\rm U}(2) \rightarrow {\rm SO}(3) \times S^1$ is a group homomorphism,
so $\pi(u)=\pi(u_1u_2)=\pi(u_1)*\pi(u_1)=[u_1]*[u_2]$,
where $[\,\cdot\,]$ denotes the equivalence class of some element of ${\rm U}(2)$.
Then similar to ${\rm U}(2)$, we say that $[u_1]*[u_2]$ and $([u_2], [u_1])$ are the same element of ${\rm SO}(3) \times S^1$, where $[u_1] \in S^1$ and $[u_2] \in {\rm SO}(3)$.
Define a map $\Phi':\,{\rm SO}(3) \times S^1 \rightarrow S^1$ by
$$
\Phi':\,[u]=([u_2], [u_1])=[u_1]*[u_2] \mapsto [u_1],\quad\forall\,[u] \in {\rm SO}(3) \times S^1.
$$
Take another element $[u']=([u_2'], [u_1'])=[u_1']*[u_2']$,
and then
$$
\Phi'([u])\,\Phi'([u'])=\Phi'([u_1]*[u_2])\,\Phi'([u_1']*[u_2'])=[u_1][u_1'],
$$
and
$$
\Phi'([u]*[u'])=\Phi'(\pi(u)*\pi(u'))=\Phi'(\pi(uu')).
$$
It follows from a simple verification that there exists some $u_3'\in {\rm SU}(2)$ such that $uu'=u_1u_1'u_3'$, where $u_1u_1' \in S^1$.
Hence we have $[u_1][u_1'] \in S^1$, and then
$$
\Phi'([u]*[u'])=\Phi'(\pi(uu'))=\Phi'(\pi(u_1u_1'u_3'))=\Phi'([u_1u_1']*[u_3'])=[u_1u_1']=[u_1][u_1']=\Phi'([u])\,\Phi'([u']).
$$
Thus, $\Phi'$ is a group homomorphism from ${\rm SO}(3) \times S^1$ to $S^1$.
Similar to the proof of Proposition \ref{prop:5}, one can see that $\Phi|_{T_{{\rm SO}(3) \times S^1}}$ is an automorphism of $T_{{\rm SO}(3) \times S^1}$ equivalent to some automorphism of $T^2$ satisfying for any $t \in T_{{\rm SO}(3) \times S^1}$,
$$
\Phi:\,t=\left(\left(\begin{array}{ccc}
\cos2\pi\alpha&-\sin2\pi\alpha&0\\
\sin2\pi\alpha&\cos2\pi\alpha&0\\
0&0&1
\end{array}\right), \textrm{e}^{2\pi{\rm i}\beta}\right)
\mapsto \left(\left(\begin{array}{ccc}
\cos2\pi(\alpha+\beta)&-\sin2\pi(\alpha+\beta)&0\\
\sin2\pi(\alpha+\beta)&\cos2\pi(\alpha+\beta)&0\\
0&0&1
\end{array}\right), \textrm{e}^{2\pi{\rm i}\beta}\right),
$$
and then we have
$$
\Phi'(\Phi([u]))=\Phi'([u]),\qquad \forall\,[u] \in {\rm SO}(3) \times S^1.
$$
This fact indicates that $\Phi|_{{\rm SO}(3)}$ is an automorphism of ${\rm SO}(3)$.
It is well-known that each automorphism of ${\rm SO}(3)$ is an inner automorphism,
so $\Phi|_{{\rm SO}(3)}$ satisfies that
$$
\Phi|_{{\rm SO}(3)}([w])=[v][w][v]^{-1},\qquad \forall\,[w] \in {\rm SO}(3),
$$
where $[v] \in {\rm SO}(3)$.
And by the same way as the proof of Proposition \ref{prop:5}, we can prove that the element $[v]$ which induces the inner automorphism $\Phi|_{{\rm SO}(3)}$ must be some $3 \times 3$ matrix like $\left(\begin{array}{ccc}
\cos2\pi\alpha_0&-\sin2\pi\alpha_0&0\\
\sin2\pi\alpha_0&\cos2\pi\alpha_0&0\\
0&0&1
\end{array}\right)$,
and then according to the discussion in Section \ref{sec:5.1}, one can see that $v \in {\rm SU}(2)$ must be some diagonal matrix like $\left(\begin{array}{cc}
z_0&0\\
0&z_0
\end{array}\right)$, where $z_0 \in \mathbb{C}$ and $|z_0|=1$.
Thus, for any $[u]=[u_1]*[u_2],\,[u']=[u_1']*[u_2'] \in {\rm SO}(3) \times S^1$, we have
$$
\Phi([u])=\Phi([u_1]*[u_2])=\Phi([u_1])*\Phi([u_2])=\Phi([u_1])*([v][u_2][v]^{-1}),
$$
$$
\Phi([u'])=\Phi([u'_1]*[u'_2])=\Phi([u'_1])*\Phi([u'_2])=\Phi([u'_1])*([v][u_2'][v]^{-1})
$$
Assume that $u_1,\,u_1',\,u_2,\,u_2',\,v$ satisfy the conditions in the proof of Proposition \ref{prop:5}.
Then we obtain
$$
\Phi([u])*\Phi([u'])=[vu_2v^{-1}]*\Phi([u'_1])=[v_2]*[v'_1]=\pi(v_2v_1'),
$$
$$
\Phi([u]*[u'])=\Phi([uu'])=\Phi([u_3u_3'])=\Phi([u_3])*[vu_3'v^{-1}]=[v_3]*[v'_3]=\pi(v_3v_3'),
$$
where $u_3 \in S^1$, $v_1'=\Phi([u'_1]),\,v_3=\Phi([u_3]) \in T_{{\rm SO}(3) \times S^1}$ and $u_3',\,v_2=[vu_2v^{-1}],\,v_3'=[vu_3'v^{-1}] \in {\rm SU}(2)$.
It follows from a series of calculations that we obtain
$$
v_2=\left(\arraycolsep=2pt\begin{array}{cc}
z_1&-\bar{z_2}z_0^2\\
z_2\bar{z_0}^2&\bar{z_1'}
\end{array}\right),
\quad v_1'=v_3=\left(\arraycolsep=2pt\begin{array}{cc}
\lambda'^2&0\\
0&\bar{\lambda'}
\end{array}\right),
\quad v_3'=\left(\arraycolsep=2pt\begin{array}{cc}
z_1&-\bar{\lambda'}\bar{z_2}z_0^2\\
\lambda'z_2\bar{z_0}^2&\bar{z_1'}
\end{array}\right),
$$
and then
$$
v_2v_1'=\left(\arraycolsep=2pt\begin{array}{cc}
\lambda'^2z_1&-\bar{\lambda'}\bar{z_2}z_0^2\\
\lambda'^2z_2\bar{z_0}^2&\bar{\lambda'}\bar{z_1'}
\end{array}\right),
\qquad v_3v_3'=\left(\arraycolsep=2pt\begin{array}{cc}
\lambda'^2z_1&-\lambda'\bar{z_2}z_0^2\\
z_2\bar{z_0}^2&\bar{\lambda'}\bar{z_1'}
\end{array}\right).
$$
Since $z_0,\,z_2 \neq 0$ and $\lambda \neq \bar{\lambda}$,
it is easy to verify that $\pi(v_2v_1') \neq \pi(v_3v_3')$ according to the definition of $\pi$,
i.e., $\Phi([u])*\Phi([u']) \neq \Phi([u]*[u'])$.
This fact is in contradiction to $\Phi$ being a isomorphism,
so the left actions $L_g$ and $L_{g'}$ can not be algebraic conjugate if $\rho(L_g)$ and $\rho(L_{g'})$ satisfy the conditions at the beginning of this proof.
\end{proof}

In fact, according to Proposition \ref{prop:5}, we see that the topologically conjugate classification of the left actions on ${\rm SO}(3) \times S^1$ is not equivalent to their algebraically conjugate classification.

\begin{proposition}\label{prop:9}
For any left actions $L_g,\,L_{g'}$ on ${\rm SO}(3) \times S^1$,
$L_g$ and $L_{g'}$ are topologically conjugate if and only if $L_g$ and $L_{g'}$ are smooth conjugate.
\end{proposition}

\begin{proof}
The sufficiency is obviously true.

If $L_g$ and $L_{g'}$ are topologically conjugate,
according to Proposition \ref{prop:3},
there exist some $t,\,t' \in T_{{\rm SO}(3) \times S^1}$ such that $L_g$ and $L_t$ are topologically conjugate, $L_{g'}$ and $L_{t'}$ are topologically conjugate,
and hence $L_t$ and $L_{t'}$ are topologically conjugate.
Assume that $h$ is a topological conjugacy from $L_t$ to $L_{t'}$.
Then it follows from Lemma \ref{lem:5} that there exist some lifts of $L_t,\,L_{t'}$ and $h$ denoted by ${\tilde L}_t,
\,{\tilde L}_{t'}$ and $\tilde{h}$, respectively, such that
$$
\tilde{h} \circ {\tilde L}_t={\tilde L}_{t'} \circ \tilde{h}.
$$
Thus, $\rho({\tilde L}_t)$ and $\rho({\tilde L}_{t'})$ satisfy the relationship in Theorem \ref{the:2}.
Then we can take another topological conjugacy $\tilde{h'}$ from $L_t$ to $L_{t'}$ such that $\tilde{h'}$ is one of the homeomorphisms defined in the proof of the sufficiency of Theorem \ref{the:2}. It is easy to see that $\tilde{h'}$ is a smooth homeomorphism.
By the proof of the sufficiency of Theorem \ref{the:4}, we know that $\tilde{h'}$ can induce a homeomorphism $h':\,{\rm SO}(3) \times S^1 \rightarrow {\rm SO}(3) \times S^1$ such that
$$
\pi \circ \tilde{h'}=h' \circ \pi,
$$
where $\pi:\,{\rm U}(2) \rightarrow {\rm SO}(3) \times S^1$ is the $4$-fold covering map.
Obviously, $h'$ is also smooth.
Then it follows from Lemma \ref{lem:5} that
$$
h' \circ L_t=L_{t'} \circ h',
$$
so $L_t$ and $L_{t'}$ are smooth conjugate.
Since each left action on ${\rm SO}(3) \times S^1$ is a smooth homeomorphism,
one can see that $L_g$ and $L_t$ are smooth conjugate, $L_{g'}$ and $L_{t'}$ are smooth conjugate,
and hence $L_g$ and $L_{g'}$ are smooth conjugate.
\end{proof}

In fact, Proposition \ref{prop:9} implies that the topologically conjugate classification of the left actions on ${\rm SO}(3) \times S^1$ is equivalent to their smooth conjugate classification.

\subsection{Topologically conjugate classification of the left actions on ${\rm Spin}^{\mathbb{C}}(3)$}

In this part, we define the rotation vectors of the left actions in the set
$$
\mathfrak{M}_{T_{{\rm Spin}^{\mathbb{C}}(3)}}=\{L_g:\,{\rm Spin}^{\mathbb{C}}(3)\rightarrow {\rm Spin}^{\mathbb{C}}(3);\,\, g \in T_{{\rm Spin}^{\mathbb{C}}(3)}\},
$$
and utilize the rotation vectors defined to give the topologically conjugate classification of the left actions in $\mathfrak{M}_{T_{{\rm Spin}^{\mathbb{C}}(3)}}$,
and then give the topologically conjugate
classification of all left actions on ${\rm Spin}^{\mathbb{C}}(3)$.
Furthermore, for the left actions on ${\rm Spin}^{\mathbb{C}}(3)$,
we study the relationship among their topological conjugacy, algebraic conjugacy and smooth conjugacy.

Assume that $T_{{\rm U}(2)}$ is the maximal torus of ${\rm U}(2)$ chosen in Section \ref{sec:4}.
Then fix one maximal torus $T_{{\rm Spin}^{\mathbb{C}}(3)}$ of ${\rm Spin}^{\mathbb{C}}(3)$
such that $T_{{\rm Spin}^{\mathbb{C}}(3)}$ is just the image
of $T_{{\rm U}(2)}$ under the covering map $\pi$. Investigate the following diagram.
$$
\xymatrixcolsep{3pc}
\xymatrix{
{\,\,T_{{\rm U}(2)}\,} \ar[d]_-{\pi} \ar[r]^-{\Phi}
& {\,\,T^2\,\,} \ar[d]^-{\pi_0}\\
{\,\,T_{{\rm Spin}^{\mathbb{C}}(3)}\,\,} \ar[r]_-{\Phi'}
& {\,\,T^2\,\,}.}
$$
In this diagram, $\Phi:\,T_{{\rm U}(2)} \rightarrow T^2$ is the isomorphism defined in Section \ref{sec:4},
$\pi$ is the covering map from ${\rm U}(2)$ to ${\rm Spin}^{\mathbb{C}}(3)$,
and $\pi_0$ is a 2-fold covering map defined by
$$
\pi':\,\left(\begin{array}{c}
{\textrm{e}}^{2\pi{\rm i}\theta}\\
{\textrm{e}}^{2\pi{\rm i}\varphi}
\end{array}\right)\mapsto\left(\begin{array}{c}
{\textrm{e}}^{2\pi{\rm i}(\theta-\varphi)}\\
{\textrm{e}}^{2\pi{\rm i}(2\varphi)}
\end{array}\right),
\qquad \forall\, \left(\begin{array}{c}
{\textrm{e}}^{2\pi{\rm i}\theta}\\
{\textrm{e}}^{2\pi{\rm i}\varphi}
\end{array}\right) \in T^2.
$$
It follows from a simple proof that $\pi$ and $\pi_0$ are both group homomorphisms, and
$$
\ker(\pi_0 \circ \Phi)=\ker\pi.
$$
Then it is easy to show that there exists some isomorphism $\Phi':\,T_{{\rm Spin}^{\mathbb{C}}(3)} \rightarrow T^2$ induced by $\Phi,\,\pi$ and $\pi_0$ such that the diagram is commutative.
Therefore, let $(T_{{\rm Spin}^{\mathbb{C}}(3)},\,\Phi')$ be the representation of the maximal torus $T_{{\rm Spin}^{\mathbb{C}}(3)}$. Obviously,
$$
T_{{\rm Spin}^{\mathbb{C}}(3)} \cong T^2.
$$
Then for any $L_g \in \mathfrak{M}_{T_{{\rm Spin}^{\mathbb{C}}(3)}}$,
set
$$
f=\Phi' \circ L_g|_{T_{{\rm Spin}^{\mathbb{C}}(3)}} \circ \Phi'^{-1}.
$$
One can see that $f$ is a rotation of $T^2$ satisfying
$$
f(t)=\left(\begin{array}{cc}
{\textrm{e}}^{2\pi{\rm i}\theta}&0\\
0&{\textrm{e}}^{2\pi{\rm i}\varphi}
\end{array}\right)t,
\qquad \forall\, t \in T^2,
$$
where $\theta,\,\varphi \in [0, 1)$.
This fact indicates that $L_g|_{T_{{\rm Spin}^{\mathbb{C}}(3)}}$ is topologically conjugate to some rotation $f$ of $T^2$.
Thus,  according to Definition \ref{def:3}, we define the rotation vector of the left action $L_g$ under the representation $(T_{{\rm Spin}^{\mathbb{C}}(3)}, \Phi')$ by
$$
\rho(L_g)\triangleq\rho(f)=\left(\begin{array}{c}
\theta\\
\varphi
\end{array}\right),
\qquad \theta,\,\varphi \in [0,1).
$$

\begin{theorem}\label{the:5}
For the left actions $L_g,\,L_{g'} \in \mathfrak{M}_{T_{{\rm Spin}^{\mathbb{C}}(3)}}$ with
$$
\rho(L_g)=\left(\arraycolsep=1pt\begin{array}{c}
\theta\\
\varphi
\end{array}\right),\qquad \rho(L_{g'})=\left(\arraycolsep=1pt\begin{array}{c}
\theta'\\
\varphi'
\end{array}\right),
$$
$L_g$ and $L_{g'}$ are topologically conjugate if and only if
$$
\left\{
\begin{array}{l}
\theta'=\pm\,\theta+n\varphi+n',\qquad n,\,n'\in\mathbb{Z},
\\[0.25 cm]
\varphi'=\pm\,\varphi\quad(\rm{mod}\,\,\,\mathbb{Z}).
\end{array}
\right.
$$
\end{theorem}

\begin{proof}
Assume that $\tilde{L}_g,\,\tilde{L}_{g'} \in \mathfrak{M}_{T_{{\rm U}(2)}}$ satisfy
$$
\rho(\tilde{L}_g)=\left(\arraycolsep=1pt\begin{array}{c}
\theta+\varphi/2+n\\
\varphi/2
\end{array}\right),\qquad \rho(\tilde{L}_{g'})=\left(\arraycolsep=1pt\begin{array}{c}
\theta\\
\varphi/2
\end{array}\right),\qquad n \in \mathbb{Z}.
$$
According to Theorem \ref{the:2}, one can see that $\tilde{L}_g$ and $\tilde{L}_{g'}$ are topologically conjugate.
Define a homeomorphism $h:\,{\rm U}(2) \rightarrow {\rm U}(2)$ by
$$
h:\,u\mapsto\left(\begin{array}{cc}
\bar{\lambda}&0\\
0&\lambda
\end{array}\right)u,
\qquad \forall\, u=\left(\begin{array}{cc}
\lambda z_1&-\lambda\bar{z_2}\\
z_2&\bar{z_1}
\end{array}\right) \in {\rm U}(2).
$$
It is not difficult to verify that $h$ is a topological conjugacy form $\tilde{L}_g$ to $\tilde{L}_{g'}$.
Set
$$
G_3=\{(e_1, e_2),\,(-e_1, -e_2)\},\qquad G_3'=\{(e_1, e_2),\,(e_1, -e_2)\},
$$
where $e_1$ and $e_2$ are the identity elements of ${\rm SU}(2)$ and $S^1$, respectively. Then $G_3$ and $G_3'$ are both normal subgroups of ${\rm U}(2)$.
It follows from a simple proof that $h(G_3)=G_3'$, and
$$
{\rm U}(2)/G_3 \cong {\rm Spin}^{\mathbb{C}}(3),
\qquad {\rm U}(2)/G_3' \cong {\rm SU}(2) \times S^1/\mathbb{Z}_2 \cong {\rm U}(2).
$$
Thus, Lemma \ref{lem:1} implies that there exists some homeomorphism
$$h':\,{\rm Spin}^{\mathbb{C}}(3) \rightarrow {\rm U}(2)
$$
induced by $h$ such that
$$
\pi' \circ h=h' \circ \pi.
$$
where
$$
\pi:\,{\rm U}(2) \rightarrow {\rm U}(2)/G_3 \cong {\rm Spin}^{\mathbb{C}}(3),
\quad \pi':\,{\rm U}(2) \rightarrow {\rm U}(2)/G_3' \cong {\rm U}(2)
$$
are the covering maps.
If the left actions $L_g \in \mathfrak{M}_{T_{{\rm Spin}^{\mathbb{C}}(3)}}$ and $L_{g'} \in \mathfrak{M}_{T_{{\rm U}(2)}}$ satisfy
$$
\rho(L_g)=\left(\arraycolsep=1pt\begin{array}{c}
\theta\\
\varphi
\end{array}\right),\qquad \rho(L_{g'})=\left(\arraycolsep=1pt\begin{array}{c}
\theta\\
\varphi
\end{array}\right),
$$
then it is easy to prove that
$$
\pi \circ \tilde{L}_g=L_g \circ \pi,\qquad \pi' \circ \tilde{L}_{g'}=L_{g'} \circ \pi',
$$
so $\tilde{L}_g$ and $\tilde{L}_{g'}$ are the lifts of $L_g$ and $L_{g'}$ under the covering maps $\pi$ and $\pi'$, respectively.
Therefore, for any $u \in {\rm U}(2)$, we have
\begin{gather*}
\pi' \circ h \circ \tilde{L}_g(u)=h' \circ \pi \circ \tilde{L}_g(u)=h' \circ L_g \circ \pi(u),
\\
\pi' \circ \tilde{L}_{g'} \circ h(u)=L_{g'} \circ \pi' \circ h(u)=L_{g'} \circ h' \circ \pi(u).
\end{gather*}
Notice that
$$
 h \circ \tilde{L}_g=\tilde{L}_{g'} \circ h,
$$
and then
$$
h' \circ L_g \circ \pi(u)=L_{g'} \circ h' \circ \pi(u),
\qquad \forall\, u \in {\rm U}(2).
$$
Since $\pi:\,{\rm U}(2) \rightarrow {\rm Spin}^{\mathbb{C}}(3)$ is a surjection, we get
$$
h' \circ L_g(u)=L_{g'} \circ h'(u),\qquad \forall\,u \in {\rm Spin}^{\mathbb{C}}(3),
$$
so $L_g$ and $L_{g'}$ are topologically conjugate.
In fact, ${\rm Spin}^{\mathbb{C}}(3)$ and ${\rm U}(2)$ are not isomorphic, but they are homeomorphic.
If we fix the same maximal tori $T_{{\rm Spin}^{\mathbb{C}}(3)}$ and $T_{{\rm U}(2)}$, and use the same way to define the rotation vectors of the left actions in ${\mathfrak M}_{T_{{\rm Spin}^{\mathbb{C}}(3)}}$ and ${\mathfrak M}_{T_{{\rm U}(2)}}$ as the previous discussion,
we can obtain that the topologically conjugate classification of the left actions on ${\rm Spin}^{\mathbb{C}}(3)$ is equivalent to the topologically conjugate classification of the left actions on ${\rm U}(2)$.
Therefore, according to Theorem \ref{the:2}, we claim that this theorem is obviously true.
\end{proof}

Finally, we investigate the relationship among the topological conjugacy, the algebraic conjugacy and the smooth conjugacy of the left actions on ${\rm Spin}^{\mathbb{C}}(3)$.

\begin{proposition}\label{prop:10}
There exist some left actions $L_g,\,L_{g'}$ on ${\rm Spin}^{\mathbb{C}}(3)$ such that $L_g$ and $L_{g'}$ are topologically conjugate,
but not algebraically conjugate.
\end{proposition}

\begin{proof}
It is known that
$$
{\rm SO}(3) \times S^1 \cong {\rm U}(2)/G_2' \cong {\rm SU}(2)/\mathbb{Z}_2 \times S^1/\mathbb{Z}_2,
\quad{\rm Spin}^{\mathbb{C}}(3) \cong {\rm U}(2)/G_3,
$$
where
\begin{gather*}
G_2'=\{(e_1, e_2),\,(e_1, -e_2),\,(-e_1, e_2),\,(-e_1, -e_2)\} \cong \mathbb{Z}_2 \oplus \mathbb{Z}_2,
\\
G_3=\{(e_1, e_2),\,(-e_1, -e_2)\} \cong \mathbb{Z}_2,
\end{gather*}
and $e_1,\,e_2$ are the identity elements of ${\rm SU}(2)$ and $S^1$, respectively.
From this view, ${\rm Spin}^{\mathbb{C}}(3)$ can be regarded as a 2-fold covering space of ${\rm SO}(3) \times S^1$.
Then for every left action $L_g \in \mathfrak{M}_{T_{{\rm SO}(3) \times S^1}}$, there exist two lifts of $L_g$ denoted by $\tilde{L}_g$ and $\tilde{L}_{g}'$,
and both of them are in the set $\mathfrak{M}_{T_{{\rm Spin}^{\mathbb{C}}(3)}}$.
It follows from a simple calculation that for any left action $L_g \in \mathfrak{M}_{T_{{\rm SO}(3) \times S^1}}$ with $\rho(L_g)=\left(\arraycolsep=1pt\begin{array}{c}
\theta\\
\varphi
\end{array}\right)$, the lifts of $L_g$ satisfy
$$
\rho(\tilde{L}_g)=\left(\arraycolsep=1pt\begin{array}{c}
\theta/2-\varphi/2+n_1\\
\varphi
\end{array}\right),
\quad\rho(\tilde{L}_g')=\left(\arraycolsep=1pt\begin{array}{c}
\theta/2-\varphi/2+1/2+n_2\\
\varphi
\end{array}\right),
$$
where $n_1,\,n_2 \in \mathbb{Z}$.
Assume that the left actions $\tilde{L}_g,\,\tilde{L}_{g'} \in \mathfrak{M}_{T_{{\rm Spin}^{\mathbb{C}}(3)}}$ are algebraically conjugate, and $\tilde{\Phi}$ is an algebraic conjugacy from $\tilde{L}_g$ to $\tilde{L}_{g'}$.
One can see that there exist some left actions $L_g,\,L_{g'} \in \mathfrak{M}_{T_{{\rm SO}(3) \times S^1}}$ such that
$$
\pi'' \circ \tilde{L}_g=L_g \circ \pi'',
\qquad \pi'' \circ \tilde{L}_{g'}=L_{g'} \circ \pi'',
$$
where $\pi''$ is the covering map from ${\rm Spin}^{\mathbb{C}}(3)$ to ${\rm SO}(3) \times S^1$.
Notice that
$$
{\rm Spin}^{\mathbb{C}}(3)/G \cong {\rm Spin}^{\mathbb{C}}(3)/\mathbb{Z}_2 \cong {\rm SO}(3) \times S^1,
$$
where $G \cong \mathbb{Z}_2$ is a normal subgroup of ${\rm Spin}^{\mathbb{C}}(3)$ consisting of the only two idempotents of ${\rm Spin}^{\mathbb{C}}(3)$,
and $\tilde{\Phi}$ is a automorphism of ${\rm Spin}^{\mathbb{C}}(3)$.
Then we have $\tilde{\Phi}(G)=G$.
Thus, Lemma \ref{lem:1} implies that there exists some isomorphism $\Phi:\,{\rm SO}(3) \times S^1 \rightarrow {\rm SO}(3) \times S^1$ such that
$$
\pi \circ \tilde{\Phi}=\Phi \circ \pi.
$$
And then it follows from a simple verification that
$$
\Phi \circ L_g=L_{g'} \circ \Phi,
$$
so $L_g$ and $L_{g'}$ are algebraically conjugate.
Thus, together with the proof of Proposition \ref{prop:8},
we know that there exist some left actions $L_g,\,L_{g'} \in \mathfrak{M}_{T_{{\rm Spin}^{\mathbb{C}}(3)}}$ such that $L_g$ and $L_{g'}$ are topologically conjugate, but not algebraically conjugate.
\end{proof}

In fact, according to Proposition \ref{prop:10}, we see that the topologically conjugate classification of the left actions on ${\rm Spin}^{\mathbb{C}}(3)$ is not equivalent to their algebraically conjugate classification.

\begin{proposition}\label{prop:11}
For any left actions $L_g,\,L_{g'}$ on ${\rm Spin}^{\mathbb{C}}(3)$,
$L_g$ and $L_{g'}$ are topologically conjugate if and only if $L_g$ and $L_{g'}$ are smooth conjugate.
\end{proposition}

Similar to the proof of Proposition \ref{prop:9},
it is easy to verify this proposition.

Proposition \ref{prop:11} implies that the topologically conjugate classification of the left actions on ${\rm Spin}^{\mathbb{C}}(3)$ is equivalent to their smooth conjugate classification.

\end{document}